\documentclass[a4paper,10pt]{amsart}
\usepackage[utf8]{inputenc}
\usepackage{amsmath,amsfonts,amssymb,mathrsfs,latexsym,bbm}
\usepackage[all,cmtip]{xy}

\usepackage[top=1in, bottom=1in, left=1.25in, right=1.25in]{geometry}
\newtheorem{definition}{Definition}[section]
\newtheorem{theorem}{Theorem}[section]
\newtheorem{lemma}{Lemma}[section]
\newtheorem{proposition}{Proposition}[section]
\newtheorem{axiom}{Axiom}
\newtheorem{remark}{Remark}[section]
\begin{document}
\title{General Motivic Cohomology and Symplectic Orientation}
\author{Nanjun Yang}
\thanks{This work has been partially supported by ERC ALKAGE}
\address{Institut Fourier-UMR 5582 \\ Universit\'e Grenoble-Alpes \\ CS 40700 \\ 38058 Grenoble Cedex 9 \\ France}
\email{ynj.t.g@126.com}

\keywords{Correspondence theory, Motivic cohomology, Symplectic orientation}
\subjclass[2010]{Primary: 14C15, 14C17, 14F42}

\begin{abstract}
In this paper, we present a general approach to establish motivic cohomology and build part of its six operations formalism. Applying this together with symplectic orientation on MW-motivic cohomology, we discuss the embedding theorem of effective Chow-Witt motives.
\end{abstract}

\maketitle
\tableofcontents
\section{Introduction}\label{Introduction}
Suppose \(E^n_C(X)\) is a kind of cohomology theory for smooth schemes, for example, \(E\) is Chow rings (CH) or \(E\) is Chow-Witt rings (\(\widetilde{CH}\), See \cite{F}). We could construct a series of bigraded cohomologies \(H_E^{p,q}(X,\mathbb{Z})\) such that
\[H_E^{2p,p}(X,\mathbb{Z})=E^p(X).\]
They are called \emph{motivic cohomologies}. The case \(E=CH\) has been done by V. Voevoedsky and the case \(E=\widetilde{CH}\) is developing by B. Calm\`es, F. D\'eglise and J. Fasel (See \cite{CF}, \cite{DF}). The main feature in the case \(E=\widetilde{CH}\) is that the Thom isomorphism like
\[E_C^n(X)\longrightarrow E_C^{n+rk(V)}(V)\]
doesn't exist any more for \(V\) vector bundle over \(X\) and \(C\) closed in \(X\). This gives an example of a non-oriented cohomology and the only solution is to make the cohomology theory dependent on the choice of vector bundles, namely, to expand the notation to \(E^n_C(X,V)\) so that 
\[E^n_C(X,V)\cong E_C^{n+rk(V)}(V)\]
canonically. Here \(V\) is called the \emph{twist} of the correspondence \(E^n_C(X,V)\) by historical reasons. So this gives rise to our general definition of motivic cohomology, which still works under the non-oriented cases.

The first step is to generalize the concept of correspondence in the case of non-oriented cohomology. This requires us to understand the twist \(V\). The serious approach to that is to regard it as an element in the category of virtual vector bundles. We will describe what is called a correspondence theory (See Section \ref{Correspondence Theory}) and provide a systematic method characterizing behaviors of twists and doing calculation (See Section \ref{Virtual Objects and Their Calculation}) on them. And as an example, we will prove that the theory of MW-correspondences is indeed a correspondence theory (See Section \ref{MW-Correspondence as a Correspondence Theory}, this part is still in progress).

Then, given a correspondence theory \(E\), we will establish the theory of sheaves with \(E\)-transfers over any smooth base and some operations, such as \(f^*\), \(f_{\#}\) and \(\otimes\), on those sheaves (See Section \ref{Sheaves}). And the category of effective and stabilized motives will be realized by localizing bounded above derived categories of sheaves with \(E\)-transfers, with the same operations above inherited, as derived functors (See Section \ref{Operations on Localized Categories}). Finally we will compare our approach with the method using unbounded complexes and its model structures, for example, \cite{CD}, \cite{CD1} and \cite{DF}.

Last but not the least, in the case \(E=\widetilde{CH}\), we will calculate the Thom space of symplectic bundles over any smooth base by using results in \cite{Y}, which will enable us to calculate the inverse Thom space of any vector bundle. Then by using the operations we defined in Section \ref{Operations on Localized Categories} together with the duality in the stable \(\mathbb{A}^1\)-derived categories (See \cite{CD1}, \cite{DF1}), we could figure out the \(Hom\)-group between motives of proper schemes, as shown below (See \cite{DF} and \cite{MVW} for notations):
\begin{theorem}
(See Theorem \ref{embedding}) Let \(X, Y\in Sm/k\) with \(Y\) proper, then we have
\[Hom_{\widetilde{DM}^{eff,-}}(\widetilde{M}(X),\widetilde{M}(Y))\cong\widetilde{CH}^{d_Y}(X\times Y,\omega_{X\times Y/X}).\]
\end{theorem}
This shows that the opposite category of effective Chow-Witt motives (See Definition \ref{effective Chow-Witt}) is a full subcategory of \(\widetilde{DM}^{eff,-}\), which is called the embedding theorem of \(\widetilde{CH}^{eff}\). A parallel result in motivic cohomology defined by V. Voevodsky says that (See \cite[Proposition 20.1]{MVW})
\[Hom_{DM^{eff,-}}(M(X),M(Y))\cong CH^{d_Y}(X\times Y),\]
for proper \(Y\), hence the opposite category of Grothendieck's effective Chow motives (\(CH^{eff}\)) is a full subcategory of \(DM^{eff,-}\).

Throughout in this article, we denote by \(Sm/k\) the category of separated schemes being smooth over \(k\) with some relative dimension (See \cite[Chapter 10]{H}), where \(k\) is an infinite perfect field with \(char(k)\neq 2\). For any \(X\in Sm/k\), we denote \(dimX\) by \(d_X\) and for any \(f:X\longrightarrow Y\) in \(Sm/k\), we set \(d_f=d_X-d_Y\).
\section{Virtual Objects and Their Calculation}\label{Virtual Objects and Their Calculation}
In this section we will introduce the category of virtual vector bundles and explain basic techniques of calculation. The definitions all come from \cite[Section 4]{D1} and we may state them here as well for clarity.
\begin{definition}\label{Picard}
(See \cite[4.1]{D1}) A category \(\mathscr{C}\) is called a commutative Picard category if
\begin{enumerate}
\item All morphisms are isomorphisms.
\item There is a bifunctorial pairing
\[+:\mathscr{C}\times\mathscr{C}\longrightarrow\mathscr{C}\]
satisfying
\begin{enumerate}
\item For every \(x,y,z\in\mathscr{C}\), an associativity isomorphism
\[a(x,y,z):(x+y)+z\longrightarrow x+(y+z);\]
\item For every \(x,y\in\mathscr{C}\), a commutativity isomorphism
\[c(x,y):x+y\longrightarrow y+x.\]
\end{enumerate}
And they satisfy associativity and commutativity constraints (See \cite{Mac}).
\item For every \(P\in\mathscr{C}\), the functors \(X\longmapsto P+X\) and \(X\longmapsto X+P\) are equivalences of categories. Thus there is a unity element \(0\) such that \(0+X\cong X\) for every \(X\in\mathscr{C}\), there is a \(-X\in\mathscr{C}\) such that \(X+(-X)\cong 0\).
\end{enumerate}
\end{definition}
\begin{definition}
Let \(X\) be a scheme, define \(Vect(X)\) to be the category of vector bundles over \(X\). Denote by \((Vect(X),iso)\) the subcategory of \(Vect(X)\) with same objects but picking only isomorphisms as morphisms.
\end{definition}
\begin{definition}\label{bracket}
(See \cite[4.3]{D1}) Let \(\mathscr{C}\) be a commutative Picard category and \(X\) be a scheme, we say it has a bracket functor on \(X\) if there is a functor
\[[-]:(Vect(X),iso)\longrightarrow\mathscr{C}\]
such that
\begin{enumerate}
\item For any exact sequence of vector bundles
\[0\longrightarrow E_1\longrightarrow E_2\longrightarrow E_3\longrightarrow 0,\]
there is an isomorphism \(\Sigma:[E_2]\longrightarrow[E_1]+[E_3]\) being natural with respect to isomorphisms between exact sequences.
\item There is an isomophism \(z:[0]\longrightarrow 0\) such that for every \(E\in Vect(X)\), the composition
\[\xymatrix{[E]\ar[r]^-{\Sigma}&[0]+[E]\ar[r]^z&0+[E]\ar[r]&[E]}\]
is \(id_{[E]}\).
\item (See Remark \ref{precise}) For every consecutive subbundle inclusions \(E_1\subseteq E_2\subseteq E_3\), there is a commutative diagram
\[
	\xymatrix
	{
		[E_3]\ar[r]^-{\Sigma}\ar[d]_{\Sigma}	&[E_1]+[E_3/E_1]\ar[d]_{\Sigma}\\
		[E_2]+[E_3/E_2]\ar[r]^-{\Sigma}		&[E_1]+[E_2/E_1]+[E_3/E_2]
	}.
\]
\item For every \(E_1, E_2\), there is a commutative diagram
\[
	\xymatrix
	{
		[E_1\oplus E_2]\ar[r]^-{\Sigma}\ar[d]_{\Sigma}	&[E_1]+[E_2]\ar[dl]^{c(E_1,E_2)}\\
		[E_2]+[E_1]							&
	}.
\]
\end{enumerate}
\end{definition}

The following comes from \cite[4.3]{D1}:
\begin{proposition}
Let \(X\) be a scheme. There is a commutative Picard category \(V(Vect(X))\) with a bracket functor on \(X\), which is called the category of virtual vector bundles over \(X\), such that for every commutative Picard category \(\mathscr{C}\) with a bracket functor on \(X\), there is a unique additive functor \(F:V(Vect(X))\longrightarrow\mathscr{C}\) making the following diagram commute
\[
	\xymatrix
	{
		(Vect(X),iso)\ar[r]^{[-]}\ar[d]_{[-]}	&V(Vect(X))\ar[ld]_F\\
		\mathscr{C}					&
	}.
\]
\end{proposition}
For convenience, we will denote \([E]\) still by \(E\) in the sequel.

The following proposition strengthens Definition \ref{bracket}, (4) a little bit.
\begin{proposition}\label{basicvvb}
Suppose we have a commutative diagram among vector bundles over \(X\) with exact row and column
\[
	\xymatrix
	{
				&						&0\ar[d]					&		&\\
				&						&D\ar[d]_d\ar[rd]^{\cong}_v	&		&\\
		0\ar[r]	&A\ar[r]^a\ar[rd]_{\cong}^u	&B\ar[r]_b\ar[d]^c			&C\ar[r]	&0\\
				&						&E\ar[d]					&		&\\
				&						&0						&		&
	}.
\]
Then the following diagram commutes in \(V(Vect(X))\)
\[
	\xymatrix
	{
		B\ar[r]\ar[d]	&A+C\ar[ld]^{c(E,D)\circ(u+v^{-1})}\\
		{D+E}			&
	}.
\]
\end{proposition}
\begin{proof}
Since \(v^{-1}\circ b\) splits \(d\), it's a standard argument that there exists a unique \(\xi:E\longrightarrow B\) such that
\[\xi\circ c=1-d\circ v^{-1}\circ b, c\circ\xi=id_E.\]
So we have commutative diagrams with exact rows:
\[
	\xymatrix
	{
			0\ar[r]	&D\ar[r]^-{i_1}\ar@{=}[d]	&D\oplus E\ar[r]^-{p_2}\ar[d]_{d+\xi}	&E\ar[r]\ar@{=}[d]	&0\\
			0\ar[r]	&D\ar[r]^{d}			&B\ar[r]^{c}					&E\ar[r]			&0
	}
\]
\[
	\xymatrix
	{
			0\ar[r]	&E\ar[r]^-{i_2}\ar[d]_{u^{-1}}	&D\oplus E\ar[r]^-{p_1}\ar[d]_{d+\xi}		&D\ar[r]\ar[d]_v	&0\\
			0\ar[r]	&A\ar[r]^{a}				&B\ar[r]^{b}						&C\ar[r]			&0
	}.
\]
Hence the statement follows from the commutative diagram by Definition \ref{bracket}, (4):
\[
	\xymatrix
	{
									&					&A+C\\
		B\ar[rru]\ar[r]\ar[rd]_{(d+\xi)^{-1}}	&D+E\ar[r]^{c(D,E)}		&E+D\ar[u]_{u^{-1}+v}\\
									&D\oplus E\ar[u]\ar[ru]	&
	}.
\]
\end{proof}

The next theorem is a fundamental tool for calculations in virtual vector bundles.
\begin{theorem}\label{four diagrams}
\begin{enumerate}
\item Suppose we have a commutative diagram with exact rows and columns
\[
	\xymatrix
	{
				&0\ar[d]			&0\ar[d]			&				&\\
				&K\ar[d]\ar@{=}[r]	&K\ar[d]			&				&\\
		0\ar[r]	&V_1\ar[r]\ar[d]	&V_2\ar[d]\ar[r]	&C\ar[r]\ar@{=}[d]	&0\\
		0\ar[r]	&W_1\ar[r]\ar[d]	&W_2\ar[r]\ar[d]	&C\ar[r]			&0\\
				&0				&0				&				&
	}
\]
among vector bundles over \(X\). Then we have a commutative diagram in \(V(Vect(X))\)
\[
	\xymatrix
	{
		V_2\ar[r]\ar[d]	&K+W_2\ar[d]\\
		V_1+C\ar[r]	&K+W_1+C.
	}
\]
\item Suppose we have a commutative diagram with exact rows and columns
\[
	\xymatrix
	{
				&0\ar[d]			&0\ar[d]			&				&\\
		0\ar[r]	&V_1\ar[r]\ar[d]	&V_2\ar[d]\ar[r]	&C\ar[r]\ar@{=}[d]	&0\\
		0\ar[r]	&W_1\ar[r]\ar[d]	&W_2\ar[r]\ar[d]	&C\ar[r]			&0\\
				&D\ar[d]\ar@{=}[r]	&D\ar[d]			&				&\\
				&0				&0				&				&
	}
\]
among vector bundles over \(X\). Then we have a commutative diagram in \(V(Vect(X))\)
\[
	\xymatrix
	{
		W_2\ar[r]\ar[d]	&V_2+D\ar[r]	&V_1+C+D\ar[ddll]^{c(C,D)}\\
		W_1+C\ar[d]	&			&\\
		V_1+D+C.		&			&
	}
\]
\item Suppose we have a commutative diagram with exact rows and columns
\[
	\xymatrix
	{
				&				&0\ar[d]			&0\ar[d]		&\\
				&				&K\ar[d]\ar@{=}[r]	&K\ar[d]		&\\
		0\ar[r]	&T\ar[r]\ar@{=}[d]	&V_1\ar[d]\ar[r]	&V_2\ar[r]\ar[d]&0\\
		0\ar[r]	&T\ar[r]			&W_1\ar[r]\ar[d]	&W_2\ar[r]\ar[d]&0\\
				&				&0				&0			&
	}
\]
among vector bundles over \(X\). Then we have a commutative diagram in \(V(Vect(X))\)
\[
	\xymatrix
	{
		V_1\ar[r]\ar[d]	&T+V_2\ar[r]	&T+K+W_2\ar[ddll]^{c(T,K)}\\
		K+W_1\ar[d]	&			&\\
		K+T+W_2.		&			&
	}
\]
\item Suppose we have a commutative diagram with exact rows and columns
\[
	\xymatrix
	{
				&				&0\ar[d]			&0\ar[d]		&\\
		0\ar[r]	&K\ar[r]\ar@{=}[d]	&V_1\ar[d]\ar[r]	&V_2\ar[r]\ar[d]&0\\
		0\ar[r]	&K\ar[r]			&W_1\ar[r]\ar[d]	&W_2\ar[r]\ar[d]&0\\
				&				&C\ar[d]\ar@{=}[r]	&C\ar[d]		&\\
				&				&0				&0			&
	}
\]
among vector bundles over \(X\). Then we have a commutative diagram in \(V(Vect(X))\)
\[
	\xymatrix
	{
		W_1\ar[r]\ar[d]	&K+W_2\ar[d]\\
		V_1+C\ar[r]	&K+V_2+C
	}.
\]
\end{enumerate}
\end{theorem}
\begin{proof}
\begin{enumerate}
\item
We have injections \(K\longrightarrow V_1\longrightarrow V_2\), which give the diagram by Definition \ref{bracket}, (3).
\item
We have injections \(V_1\longrightarrow V_2\longrightarrow W_2\) and \(V_1\longrightarrow W_1\longrightarrow W_2\). These give two commutative diagrams by Definition \ref{bracket}, (3)
\[
	\xymatrix
	{
		W_2\ar[r]\ar[d]	&V_1+W_2/V_1\ar[d]_{\Sigma_1}\\
		V_2+D\ar[r]	&V_1+C+D
	}
	\xymatrix
	{
		W_2\ar[r]\ar[d]	&V_1+W_2/V_1\ar[d]_{\Sigma_2}\\
		W_1+C\ar[r]	&V_1+D+C
	}.
\]
And we have a commutative diagram with exact row and column
\[
	\xymatrix
	{
						&\Sigma_2:			&0\ar[d]			&		&\\
						&				&D\ar[d]\ar@{=}[rd]	&		&\\
		\Sigma_1:0\ar[r]	&C\ar[r]\ar@{=}[rd]	&W_2/V_1\ar[r]\ar[d]	&D\ar[r]	&0\\
						&				&C\ar[d]			&		&\\
						&				&0				&		&
	}.
\]
Thus we have a commutative diagram
\[
	\xymatrix
	{
		W_2/V_1\ar[r]^{\Sigma_1}\ar[rd]_{\Sigma_2}	&C+D\ar[d]^{c(C,D)}\\
										&D+C
	}
\]
by Proposition \ref{basicvvb}. So combining the diagrams above gives the result.
\item
We denote the morphism \(V_1\longrightarrow V_2\longrightarrow W_2\) by \(\alpha\). There are morphisms \(ker(\alpha)\longrightarrow K\) and \(ker(\alpha)\longrightarrow T\) satisfying the following commutative diagrams
\[
	\xymatrix
	{
		ker(\alpha)\ar[r]\ar[d]	&K\ar[d]\\
		V_1\ar[r]				&V_2
	}
	\xymatrix
	{
		ker(\alpha)\ar[r]\ar[d]	&V_1\ar[d]\\
		T\ar[r]				&W_1
	}
\]
by the universal property of \(K\) and \(T\) as kernels. Then there is a commutative diagram with exact row and column
\[
	\xymatrix
	{
						&\Sigma_2:			&0\ar[d]				&		&\\
						&				&K\ar[d]\ar@{=}[rd]		&		&\\
		\Sigma_1:0\ar[r]	&T\ar[r]\ar@{=}[rd]	&ker(\alpha)\ar[r]\ar[d]	&K\ar[r]	&0\\
						&				&T\ar[d]				&		&\\
						&				&0.					&		&
	}
\]
Hence we have a commutative diagram
\[
	\xymatrix
	{
		ker(\alpha)\ar[r]^{\Sigma_1}\ar[d]_{\Sigma_2}	&T+K\ar[dl]^{c(T,K)}\\
		{K+T}									&\\
	}
\]
by Proposition \ref{basicvvb}.

We have injections \(T\longrightarrow ker(\alpha)\longrightarrow V_1\), \(K\longrightarrow ker(\alpha)\longrightarrow V_1\), which induce the following commutative diagrams by Definition \ref{bracket}, (3):
\[
	\xymatrix
	{
		V_1\ar[r]\ar[d]				&T+V_2\ar[d]\\
		ker(\alpha)+W_2\ar[r]^{\Sigma_1}&T+K+W_2
	}
\]
\[
	\xymatrix
	{
		V_1\ar[r]\ar[d]				&K+W_1\ar[d]\\
		ker(\alpha)+W_2\ar[r]^{\Sigma_2}&K+T+W_2.
	}
\]
So combining the diagrams above gives the result.
\item The diagram is a rotation and reflection of the diagram in (1).
\end{enumerate}
\end{proof}
\begin{remark}\label{precise}
We remark that (1) in the above theorem is actually the \emph{precise} meaning of Definition \ref{bracket}, (3).
\end{remark}
\begin{remark}
We would like to point out that the calculation of virtual objects is \emph{not} trivial, especially when judging commutativity of diagrams. We will see this point in the sections below.
\end{remark}
\section{Correspondence Theory}\label{Correspondence Theory}
In this section, we are going to define what is a correspondence theory by axiomatic method, under the language of virtual vector bundles defined in the previous section.
\begin{definition}
Let \(X\) be a noetherian scheme and \(i\in\mathbb{N}\). We denote \(Z^i(X)\) to be the set of closed subsets in \(X\) whose components have codimension \(i\).
\end{definition}
\begin{definition}
Let \(X\in Sm/k\), \(C\in Z^i(X)\) and \(D\in Z^j(X)\). We say \(C\) and \(D\) intersect properly if \(C\cap D\in Z^{i+j}(X)\).
\end{definition}
\begin{axiom}\label{T}
(Twists) For every \(X\in Sm/k\), we have a commutative Picard category (See Definition \ref{Picard}) \(\mathscr{P}_X\) with an additive functor \(p_X:V(Vect(X))\longrightarrow\mathscr{P}_X\) and a rank morphism \(rk_X:\mathscr{P}_X\longrightarrow\mathbb{Z}/2\mathbb{Z}\) such that:
\begin{enumerate}
\item The following diagram commutes
\[
	\xymatrix
	{
		V(Vect(X))\ar[r]^-{rk}\ar[d]_{p_X}	&\mathbb{Z}\ar[d]\\
		\mathscr{P}_X\ar[r]^-{rk_X}		&\mathbb{Z}/2\mathbb{Z},
	}
\]
where the upper horizontal arrow is defined by \(rk([E])=rk(E)\).
\item For every \(f:X\longrightarrow Y\) in \(Sm/k\), there is a pull-back morphism \(f^*:\mathscr{P}_Y\longrightarrow\mathscr{P}_X\) such that the following diagrams commute
\[
	\xymatrix
	{
		\mathscr{P}_Y\ar[r]^{f^*}\ar[d]_{rk_Y}	&\mathscr{P}_X\ar[dl]^{rk_X}\\
		\mathbb{Z}/2\mathbb{Z}				&
	}
	\xymatrix
	{
		V(Vect(Y))\ar[r]^{f^*}\ar[d]_{p_Y}	&V(Vect(X))\ar[d]_{p_X}\\
		\mathscr{P}_Y\ar[r]^{f^*}			&\mathscr{P}_X
	},
\]
where \(f^*:V(Vect(Y))\longrightarrow V(Vect(X))\) is defined by \(f^*([E])=[f^*E]\). And \(f^*g^*=(g\circ f)^*\) for any morphisms \(f, g\) in \(Sm/k\). Moreover, \(f^*(-v)=-f^*(v)\).
\end{enumerate}
\end{axiom}
\begin{remark}
In application, the categories \(\mathscr{P}_X\) should be chosen as `small' as possible. Since this will allow more isomophisms, such as orientations as we will see in Definition \ref{symplectic oriented}.
\end{remark}
\begin{axiom}\label{C}
(Correspondences) For every \(X\in Sm/k\), \(i\in\mathbb{N}\), \(C\in Z^i(X)\) and \(v\in\mathscr{P}_X\), we associate an abelian group \(E_C^i(X,v)\), which is called the correspondence supported on \(C\) with twist \(v\). They are functorial with respect to \(v\). Moreover, if \(C=\emptyset\), \(E_C^i(X,v)\) is understood as \(0\).
\end{axiom}
\begin{remark}
Note that we always assume that \(C\in Z^i(X)\). This has the advantage to keep the operations very concrete (in the case of Chow or Chow-Witt groups), but has the disadvantage that the group \(E^i(X,v)\) with full support is in general not defined.
\end{remark}

We are going to describe further properties these groups should satisfy.
\begin{axiom}\label{ES}
(Extension of Supports) For every  \(X\in Sm/k\), \(C_1\subseteq C_2\in Z^i(X)\), \(i\in\mathbb{N}\), \(v\in\mathscr{P}_X\), we have a morphism
\[e(C_1,C_2):E^i_{C_1}(X,v)\longrightarrow E^i_{C_2}(X,v)\]
which is called the extension of support. This map is functorial with respect to \(v\).

For any disjoint \(C_1, C_2\in Z^i(X)\), we have
\[E^i_{C_1\cup C_2}(X,v)\cong E^i_{C_1}(X,v)\oplus E^i_{C_2}(X,v)\]
via extension of supports. Moreover, for any \(C_1\subseteq C_2\subseteq C_3\) we have
\[e(C_2,C_3)\circ e(C_1,C_2)=e(C_1,C_3).\]
\end{axiom}
\begin{axiom}\label{P}
(Product) Suppose \(X\in Sm/k\), \(v_1,v_2\in\mathscr{P}_X\), \(C_1, C_2\in Z^i(X)\) and \(i,j\in\mathbb{N}\). Suppose \(C_1\) and \(C_2\) intersect properly, then we have a product
\[\xymatrix{E_{C_1}^i(X,v_1)\times E_{C_2}^j(X,v_2)\ar[r]^-{\cdot}&E_{C_1\cap C_2}^{i+j}(X,v_1+v_2)},\]
And this product is functorial with respect to twists and extension of supports.
\end{axiom}
\begin{axiom}\label{A}
(Associativity) For any \(X\in Sm/k\), \(v_a\in\mathscr{P}_X\) and \(C_a\in Z^{i_a}(X)\), \(a=1,2,3\), the following diagram commutes
\[
	\xymatrix
	{
		E_{C_1}^{i_1}(X,v_1)\times E_{C_2}^{i_2}(X,v_2)\times E_{C_3}^{i_3}(X,v_3)\ar[d]_{\cdot\times id}\ar[r]^-{id\times\cdot}	&E_{C_1}^{i_1}(X,v_1)\times E_{C_2\cap C_3}^{i_2+i_3}(X,v_2+v_3)\ar[d]_{\cdot}\\
		E_{C_1\cap C_2}^{i_1+i_2}(X,v_1+v_2)\times E_{C_3}^{i_3}(X,v_3)\ar[d]_{\cdot}									&E_{C_1\cap C_2\cap C_3}^{i_1+i_2+i_3}(X,v_1+(v_2+v_3))\ar[dl]_{a(v_1,v_2,v_3)^{-1}}\\
		E_{C_1\cap C_2\cap C_3}^{i_1+i_2+i_3}(X,(v_1+v_2)+v_3)													&
	},
\]
where every intersection appeared above is proper.
\end{axiom}
\begin{axiom}\label{CC}
(Conditional Commutativity) Let \(X\in Sm/k\), \(C_a\in Z^{i_a}(X)\), \(i_a\in\mathbb{N}\), \(v_a\in\mathscr{P}_X\) where \(a=1,2\). If \((i_1+rk_X(v_1))(i_2+rk_X(v_2))=0\in\mathbb{Z}/2\mathbb{Z}\) and \(C_1\) and \(C_2\) intersect properly, the following diagram commutes:
\[
	\xymatrix
	{
		E^{i_1}_{C_1}(X,v_1)\times E^{i_2}_{C_2}(X,v_2)\ar[r]^-{\cdot}\ar[d]	&E^{i_1+i_2}_{C_1\cap C_2}(X,v_1+v_2)\ar[d]_{c(v_1,v_2)}\\
		E^{i_2}_{C_2}(X,v_2)\times E^{i_1}_{C_1}(X,v_1)\ar[r]^-{\cdot}		&E^{i_1+i_2}_{C_1\cap C_2}(X,v_2+v_1)
	}.
\]
\end{axiom}
\begin{axiom}\label{I}
(Identity) For any \(X\in Sm/k\), there is an element \(e\) in \(E^0_X(X,0)\) such that for any \(v\in\mathscr{P}_X\), \(i\in\mathbb{N}\) and \(C\in Z^i(X)\), the following diagrams commute
\[
	\xymatrix
	{
		E^i_C(X,v)\ar[r]^-{e\cdot}\ar@{=}[d]	&E^i_C(X,0+v)\ar[dl]_u\\
		E^i_C(X,v)						&
	}
	\xymatrix
	{
		E^i_C(X,v)\ar[r]^-{\cdot e}\ar@{=}[d]	&E^i_C(X,v+0)\ar[dl]_u\\
		E^i_C(X,v)						&
	},
\]
where \(u\) are the unity isomorphisms of \(0\) in \(\mathscr{P}_X\). We call \(e\) the identity and denote it by \(1\).
\end{axiom}
\begin{axiom}\label{PB}
(Pull-Back) Suppose \(f:X\longrightarrow Y\) is morphism in \(Sm/k\), \(i\in\mathbb{N}\), \(C\in Z^i(Y)\), \(f^{-1}(C)\in Z^i(X)\) and \(v\in\mathscr{P}_Y\). Then we have a pull-back morphism
\[E^i_C(Y,v)\longrightarrow E^i_{f^{-1}(C)}(X,f^*v).\]
This morphism is functorial with respect to \(v\) and extension of supports.
\end{axiom}
\begin{axiom}\label{FPB}
(Functoriality of Pull-Back) Let \(\xymatrix{X\ar[r]^g&Y\ar[r]^f&Z}\) be morphisms in \(Sm/k\), \(i\in\mathbb{N}\), \(C\in Z^i(Z)\), \(f^{-1}(C)\in Z^i(Y)\), \(g^{-1}f^{-1}(C)\in Z^i(X)\) and \(v\in\mathscr{P}_Z\). We have
\[(f\circ g)^*=g^*\circ f^*.\]
And the pull-back of identity morphism is just the identity morphism.
\end{axiom}
\begin{axiom}\label{CPB}
(Compability of Pull-Back) Suppose \(f:X\longrightarrow Y\) is a morphism in \(Sm/k\), \(i,j\in\mathbb{N}\) and \(C_1\in Z^i(Y)\), \(C_2\in Z^j(Y)\) and they intersect properly (the same for their preimages). For any \(v_1,v_2\in\mathscr{P}_Y\), we have a commutative diagram
\[
	\xymatrix
	{
		E^i_{C_1}(Y,v_1)\times E^j_{C_2}(Y,v_2)\ar[r]^-{\cdot}\ar[d]^{f^*\times f^*}	&E^{i+j}_{C_1\cap C_2}(Y,v_1+v_2)\ar[d]^{f^*}\\
		E^i_{f^{-1}(C_1)}(X,f^*(v_1))\times E^j_{f^{-1}(C_2)}(X,f^*(v_2))\ar[r]^-{\cdot}&E^{i+j}_{f^{-1}(C_1\cap C_2)}(X,f^*(v_1+v_2))
	}.
\]

And we always have \(f^*(1)=1\).
\end{axiom}

We recall some facts about tangent bundles and normal bundles.
\begin{lemma}\label{t}
Let \(\xymatrix{X\ar[r]^f&Y\ar[r]^g&Z}\) be morphisms in \(Sm/k\).
\begin{enumerate}
\item If \(f\), \(g\) are smooth, then we have an exact sequence
\[0\longrightarrow T_{X/Y}\longrightarrow T_{X/Z}\longrightarrow f^*T_{Y/Z}\longrightarrow 0.\]
\item If \(f\) is a closed immersion and \(g\), \(g\circ f\) are smooth, then we have an exact sequence
\[0\longrightarrow T_{X/Z}\longrightarrow f^*T_{Y/Z}\longrightarrow N_{X/Y}\longrightarrow 0.\]
\item If \(g\) is smooth and \(f\), \(g\circ f\) are closed immersions, then we have an exact sequence
\[0\longrightarrow f^*T_{Y/Z}\longrightarrow N_{X/Y}\longrightarrow N_{X/Z}\longrightarrow 0.\]
\item If \(f\), \(g\) are closed immersions, then we have an exact sequence
\[0\longrightarrow N_{X/Y}\longrightarrow N_{X/Z}\longrightarrow f^*N_{Y/Z}\longrightarrow 0.\]
\end{enumerate}
\end{lemma}
\begin{proof}
See \cite[Chapter II, Proposition 8.11, Proposition 8.12 and Theorem 8.17 and Chapter III, Proposition 10.4]{H}.
\end{proof}
\begin{lemma}\label{t1}
Suppose we have a Cartesian square of schemes
\[
	\xymatrix
	{
		X'\ar[r]^v\ar[d]^{g}	&X\ar[d]^f\\
		Y'\ar[r]^{u}		&Y.
	}
\]
Then the composition \(T_{X'/Y'}\longrightarrow T_{X'/Y}\longrightarrow v^*T_{X/Y}\) is an isomorphism.
\end{lemma}
\begin{proof}
See \cite[Chapter II, Proposition 8.10]{H}.
\end{proof}
\begin{lemma}\label{t2}
Suppose we have a Cartesian square in \(Sm/k\)
\[
	\xymatrix
	{
		X'\ar[r]^v\ar[d]^{g}	&X\ar[d]^f\\
		Y'\ar[r]^{u}		&Y
	}
\]
such that \(f\) is a closed immersion. If one of the following conditions holds
\begin{enumerate}
\item\(u\) is smooth
\item\(u\) is a closed immersion and \(dimX'-dimY'=dimX-dimY\),
\end{enumerate}
the natural morphism \(\gamma\) defined by the following commutative diagram with exact rows
\[
	\xymatrix
	{
		0\ar[r]	&v^*T_{X/k}\ar[r]			&v^*f^*T_{Y/k}\ar[r]				&v^*N_{X/Y}\ar[r]			&0\\
		0\ar[r]	&T_{X'/k}\ar[r]\ar[u]_{\alpha}	&g^*T_{Y'/k}\ar[r]\ar[u]_{\beta}	&N_{X'/Y'}\ar[r]\ar[u]_{\gamma}	&0
	}
\]
is an isomorphism.
\end{lemma}
\begin{proof}
If \(u\) is smooth, then \(\alpha\) and \(\beta\) are surjective and have the same kernel by the previous two lemmas. So \(\gamma\) is an isomorphism by the snake lemma.

In the other case, the dimension condition implies \(N_{X'/Y'}\) and \(N_{X/Y}\) have the same rank. So we only have to prove \(\gamma^{\vee}\) is surjective. Then we could assume all schemes are affine. Suppose \(Y=Spec A\), \(X=Spec A/J\), \(Y'=Spec A/I\) and \(X'=Spec A/(I+J)\). Then \(N_{X/Y}^{\vee}=I/I^2\) and \(N_{X'/Y'}^{\vee}=(I+J)/(I^2+J)\). Now the morphism \(\gamma\) is given by
\[\begin{array}{ccccc}I/I^2&\otimes_{A/I}&A/(I+J)&\longrightarrow&(I+J)/(I^2+J)\\(\overline{i}&,&\overline{a})&\longmapsto&\overline{ai}\end{array},\]
which is obviously surjective.
\end{proof}
\begin{axiom}\label{P-FSM}
(Push-Forward for Smooth Morphisms) Suppose \(f:X\longrightarrow Y\) is a smooth morphism in \(Sm/k\), \(n\in\mathbb{N}\), \(v\in\mathscr{P}_X\) and \(C\in Z^{n+d_f}(X)\) is finite over \(Y\). Then we have a morphism
\[f_*:E_C^{n+d_f}(X,f^*v-T_{X/Y})\longrightarrow E_{f(C)}^n(Y,v),\]
which is also functorial with respect to \(v\) and extension of supports. The push-forward of the identity morphism is just the identity morphism, by using \(T_{X/Z}=0\).
\end{axiom}
We may also use the simplified notation
\[f^*v-T_{X/Y}\longrightarrow v\]
to denote \(f_*\). Moreover, we could talk about the push-forward of the form
\[f_*:E_C^{n+d_f}(X,f^*v_1-T_{X/Y}+f^*v_2)\longrightarrow E_{f(C)}^n(Y,v_1+v_2).\]
It is defined by the composition of the push-forward defined above and the commutativity isomorphism \(c(-T_{X/Y},f^*v_2)\).
\begin{axiom}\label{FPFSM}
(Functoriality of Push-Forward for Smooth Morphisms) Suppose \(\xymatrix{X\ar[r]^g&Y\ar[r]^f&Z}\) are smooth morphisms in \(Sm/k\), \(i\in\mathbb{N}\), \(C\in Z^{i+d_X-d_Z}(X)\) is finite over \(Z\) and \(v\in\mathscr{P}_Z\). Then we have a commutative diagram
\[
	\xymatrix
	{
		E^{i+d_X-d_Z}_C(X,(f\circ g)^*v-T_{X/Z})\ar[r]^-{\varphi}\ar[ddr]^{(f\circ g)_*}	&E^{i+d_X-d_Z}_C(X,(f\circ g)^*v-g^*T_{Y/Z}-T_{X/Y})\ar[d]_{g_*}\\
																		&E^{i+d_Y-d_Z}_{g(C)}(Y,f^*v-T_{Y/Z})\ar[d]_{f_*}\\
																		&E^i_{f(g(C))}(Z,v)
	},
\]
where \(\varphi\) is obtained via
\begin{eqnarray*}
(f\circ g)^*v-T_{X/Z}	&\longrightarrow	&(f\circ g)^*v-(T_{X/Y}+g^*T_{Y/Z})\\
				&\longrightarrow 	&(f\circ g)^*v-g^*T_{Y/Z}-T_{X/Y}.
\end{eqnarray*}
\end{axiom}
\begin{axiom}\label{P-FCI}
(Push-Forward for Closed Immersions) Suppose \(f:X\longrightarrow Y\) is a closed immersion in \(Sm/k\), \(v\in\mathscr{P}_Y\) and \(C\in Z^{n+d_f}(X)\). Then we have an \emph{isomorphism}
\[f_*:E_C^{n+d_f}(X,N_{X/Y}+f^*v)\longrightarrow E_{f(C)}^n(Y,v),\]
And this morphism is also functorial in  \(v\) and extension of supports. The push-forward of the identity is just the identity, by using \(N_{X/Y}=0\).
\end{axiom}
So given a vector bundle \(V\) over \(X\), the definition above gives the isomorphism \(E^n_C(X,V)\cong E_C^{n+rk_X(V)}(V,0)\) (See Section \ref{Introduction}) via the push-forward of the zero section.

We may also use the simplified notation
\[N_{X/Y}+f^*v\longrightarrow v\]
to denote \(f_*\). Moreover, we could talk about the push-forward of the form
\[f_*:E_C^{n+d_f}(X,f^*v_1+N_{X/Y}+f^*v_2)\longrightarrow E_{f(C)}^n(Y,v_1+v_2).\]
It is defined by the composition of the push-forward defined above and the commutativity isomorphism \(c(f^*v_1,N_{X/Y})\).
%
%
\begin{axiom}\label{FPFCI}
(Functoriality of Push-Forward for Closed Immersions) Suppose \(\xymatrix{X\ar[r]^g&Y\ar[r]^f&Z}\) are closed immersions in \(Sm/k\), \(C\in Z^{i+d_X-d_Z}(X)\) and \(v\in\mathscr{P}_Z\). Then we have a commutative diagram
\[
	\xymatrix
	{
		E^{i+d_X-d_Z}_C(X,N_{X/Z}+(f\circ g)^*v)\ar[r]^-{\varphi}\ar[ddr]^{(f\circ g)_*}	&E^{i+d_X-d_Z}_C(X,N_{X/Y}+g^*N_{Y/Z}+(f\circ g)^*v)\ar[d]_{g_*}\\
																		&E^{i+d_Y-d_Z}_{g(C)}(Y,N_{Y/Z}+f^*v)\ar[d]_{f_*}\\
																		&E^i_{f(g(C))}(Z,v)
	},
\]
where \(\varphi\) is got by the isomorphism \(N_{X/Z}+(f\circ g)^*v\cong N_{X/Y}+g^*N_{Y/Z}+(f\circ g)^*v\).
\end{axiom}
\begin{axiom}\label{BCSM}
(Base Change for Smooth Morphisms) Suppose we have a Cartesian square with all schemes being smooth
\[
	\xymatrix
	{
		X'\ar[r]^v\ar[d]^{g}	&X\ar[d]^f\\
		Y'\ar[r]^{u}		&Y
	}
\]
where \(c=d_X-d_Y=d_{X'}-d_{Y'}\), \(n\in\mathbb{N}\), \(s\in\mathscr{P}_Y\), \(f\) smooth, \(C\in Z^{n+c}(X)\) is finite over \(Y\) and \(v^{-1}(C)\in Z^{n+c}(X)\). Then the following diagram commutes
\[
	\xymatrix
	{
		E_C^{n+c}(X,f^*s-T_{X/Y})\ar[r]^-{f_*}\ar[d]^{v^*}		&E^n_{f(C)}(Y,s)\ar[d]^{u^*}\\
		E_{v^{-1}(C)}^{n+c}(X',v^*f^*s-v^*T_{X/Y})\ar[r]^-{g_*}	&E_{g(v^{-1}(C))}^n(Y',u^*s)
	}.
\]
Here we have used the canonical isomorphism \(T_{X'/Y'}\longrightarrow v^*T_{X/Y}\) in Lemma \ref{t1}.
\end{axiom}
\begin{axiom}\label{BCCI}
(Base Change for Closed Immersions) Suppose we have a Cartesian square with all schemes being smooth
\[
	\xymatrix
	{
		X'\ar[r]^v\ar[d]^{g}	&X\ar[d]^f\\
		Y'\ar[r]^{u}		&Y
	}
\]
where \(c=d_X-d_Y=d_{X'}-d_{Y'}\), \(s\in\mathscr{P}_Y\), \(f\) is a closed immersion, \(C\in Z^{n+c}(X)\) and \(v^{-1}(C)\in Z^{n+c}(X)\). Then the following diagram commutes
\[
	\xymatrix
	{
		E_C^{n+c}(X,N_{X/Y}+f^*s)\ar[r]^-{f_*}\ar[d]^{v^*}		&E^n_{f(C)}(Y,s)\ar[d]^{u^*}\\
		E_{v^{-1}(C)}^{n+c}(X',v^*N_{X/Y}+v^*f^*s)\ar[r]^-{g_*}	&E_{g(v^{-1}(C))}^n(Y',u^*s)
	}.
\]
\end{axiom}
\begin{axiom}\label{PFSM}
(Projection Formula for Smooth Morphisms) Suppose we have a smooth morphism \(f:X\longrightarrow Y\) in \(Sm/k\), \(n,m\in\mathbb{N}\), \(C\in Z^{n+d_f}(X)\) being finite over \(Y\), \(D\in Z^m(Y)\), \(C\) and \(f^{-1}(D)\) intersect properly and \(v_1,v_2\in\mathscr{P}_Y\). Then the following diagrams commute
\[
	\xymatrix
	{
		E_C^{n+d_f}(X,f^*v_1-T_{X/Y})\times E_D^m(Y,v_2)\ar[r]^-{id\times f^*}\ar[d]^{f_*\times id}	&E_C^{n+d_f}(X,f^*v_1-T_{X/Y})\times E_{f^{-1}(D)}^m(X,f^*v_2)\ar[d]^{\cdot}\\
		E_C^n(Y,v_1)\times E_D^m(Y,v_2)\ar[d]^{\cdot}									&E_{C\cap f^{-1}(D)}^{n+m+d_f}(X,f^*v_1-T_{X/Y}+f^*v_2)\ar[dl]^{f_*}\\
		E_{Y,f(C)\cap D}^{n+m}(v_1+v_2)												&
	}
\]
\[
	\xymatrix
	{
		E_D^m(Y,v_2)\times E_C^{n+d_f}(X,f^*v_1-T_{X/Y})\ar[r]^-{f^*\times id}\ar[d]^{id\times f_*}	&E_{f^{-1}(D)}^m(X,f^*v_2)\times E_C^{n+d_f}(X,f^*v_1-T_{X/Y})\ar[d]^{\cdot}\\
		E_D^m(Y,v_2)\times E_C^n(Y,v_1)\ar[d]^{\cdot}									&E_{C\cap f^{-1}(D)}^{n+m+d_f}(X,f^*v_2+f^*v_1-T_{X/Y})\ar[dl]^{f_*}\\
		E_{f(C)\cap D}^{n+m}(Y,v_2+v_1)												&
	}.
\]
\end{axiom}
\begin{axiom}\label{PFCI}
(Projection Formula for Closed Immersions) Suppose we have a closed immersion \(f:X\longrightarrow Y\) in \(Sm/k\), \(n,m\in\mathbb{N}\), \(C\in Z^{n+d_f}(X)\), \(D\in Z^m(Y)\), \(f^{-1}(D)\in Z^m(X)\), \(C\) and \(f^{-1}(D)\) intersect properly and \(v_1,v_2\in\mathscr{P}_Y\). Then the following diagrams commute
\[
	\xymatrix
	{
		E_C^{n+d_f}(X,N_{X/Y}+f^*v_1)\times E_D^m(Y,v_2)\ar[r]^-{id\times f^*}\ar[d]^{f_*\times id}	&E_C^{n+d_f}(X,N_{X/Y}+f^*v_1)\times E_{f^{-1}(D)}^m(X,f^*v_2)\ar[d]^{\cdot}\\
		E_C^n(Y,v_1)\times E_D^m(Y,v_2)\ar[d]^{\cdot}									&E_{C\cap f^{-1}(D)}^{n+m+d_f}(X,N_{X/Y}+f^*v_1+f^*v_2)\ar[dl]^{f_*}\\
		E_{f(C)\cap D}^{n+m}(Y,v_1+v_2)												&
	}
\]
\[
	\xymatrix
	{
		E_D^m(Y,v_2)\times E_C^{n+d_f}(X,N_{X/Y}+f^*v_1)\ar[r]^-{f^*\times id}\ar[d]^{id\times f_*}	&E_{f^{-1}(D)}^m(X,f^*v_2)\times E_C^{n+d_f}(X,N_{X/Y}+f^*v_1)\ar[d]^{\cdot}\\
		E_D^m(Y,v_2)\times E_C^n(Y,v_1)\ar[d]^{\cdot}									&E_{C\cap f^{-1}(D)}^{n+m+d_f}(X,f^*v_2+N_{X/Y}+f^*v_1)\ar[dl]^{f_*}\\
		E_{f(C)\cap D}^{n+m}(Y,v_2+v_1)												&
	}.
\]
\end{axiom}

We still need a compability between the two push-forwards introduced.
\begin{axiom}\label{CTPF}
(Compability between the two Push-Forwards)
\begin{enumerate}
\item Suppose \(\xymatrix{X\ar[r]^f&Z\ar[r]^g&Y}\) are morphisms in \(Sm/k\), \(C\in Z^{i+d_X-d_Y}(X)\) is finite over \(Y\), \(i\in\mathbb{N}\) and \(v\in\mathscr{P}_Y\), where \(f\) is a closed immersion and \(g\), \(g\circ f\) are smooth. Then the following diagram commutes
\[
	\xymatrix
	{
		E^{i+d_X-d_Y}_C(X,N_{X/Z}+f^*g^*v-f^*T_{Z/Y})\ar[d]_{f_*}\ar[r]^{c(N_{X/Z},f^*g^*v)}	&E^{i+d_X-d_Y}_C(X,f^*g^*v+N_{X/Z}-f^*T_{Z/Y})\ar[d]^{\varphi}\\
		E_{f(C)}^{i+d_Z-d_Y}(Z,g^*v-T_{Z/Y})\ar[d]_{g_*}							&E^{i+d_X-d_Y}_C(X,f^*g^*v-T_{X/Y})\ar[dl]_{(g\circ f)_*}\\
		E_{g(f(C))}^i(Y,v)													&
	},
\]
where \(\varphi\) is induced by Lemma \ref{t}, (2).
\item Suppose \(\xymatrix{X\ar[r]^f&Z\ar[r]^g&Y}\) are morphisms in \(Sm/k\), \(C\in Z^{i+d_X-d_Y}(X)\) finite over \(Y\), \(i\in\mathbb{N}\) and \(v\in\mathscr{P}_Y\), where \(g\) is smooth and \(f\), \(g\circ f\) are closed immersions. Then the following diagram commutes
\[
	\xymatrix
	{
		E^{i+d_X-d_Y}_C(X,N_{X/Z}+f^*g^*v-f^*T_{Z/Y})\ar[d]_{f_*}\ar[r]_{c(N_{X/Z}+f^*g^*v,-f^*T_{Z/Y})}	&E^{i+d_X-d_Y}_C(X,-f^*T_{Z/Y}+N_{X/Z}+f^*g^*v)\ar[d]^{\varphi}\\
		E_{f(C)}^{i+d_Z-d_Y}(Z,g^*v-T_{Z/Y})\ar[d]_{g_*}										&E^{i+d_X-d_Y}_C(X,N_{X/Y}+f^*g^*v)\ar[dl]_{(g\circ f)_*}\\
		E_{g(f(C))}^i(Y,v)																&
	},
\]
where \(\varphi\) is induced by Lemma \ref{t}, (3).
\item Suppose we have a Cartesian square with all schemes being smooth
\[
	\xymatrix
	{
		X'\ar[r]^v\ar[d]^{g}	&X\ar[d]^f\\
		Y'\ar[r]^{u}		&Y,
	}
\]
where \(u\) is smooth and \(f\) is a closed immersion. Suppose moreover that \(C\in Z^{n+d_f+d_v}(X')\) is finite over \(Y\), \(s\in\mathscr{P}_Y\). Then the following diagram commutes
\[
	\xymatrix
	{
		E_C^{n+d_f+d_v}(X',N_{X'/Y'}+g^*u^*s-g^*T_{Y'/Y})\ar[r]^-{g_*}\ar[d]	&E_{g(C)}^{n+d_u}(Y',u^*s-T_{Y'/Y})\ar[d]_{u_*}\\
		E_C^{n+d_f+d_v}(X',v^*N_{X/Y}+u^*f^*s-T_{X'/X})\ar[d]_{v_*}			&E_{u(g(C))}^n(Y,s)\\
		E_{v(C)}^{n+d_f}(X,N_{X/Y}+f^*s)\ar[ru]_{f_*}					&
	}.
\]
\end{enumerate}
\end{axiom}
\begin{axiom}\label{EE}
(\'Etale Excision) Suppose \(f:X\longrightarrow Y\) is an \'etale morphism in \(Sm/k\), \(C\in Z^i(Y)\) and the morphism \(f:f^{-1}(C)\longrightarrow C\) is an isomorphism (where bothe terms are endowed with their reduced closed subscheme structures), then for any \(i\in\mathbb{N}\) and \(v\in\mathscr{P}_Y\), the pull-back morphism
\[f^*:E^i_C(Y,v)\longrightarrow E^i_{f^{-1}(C)}(X,f^*(v))\]
is an isomorphism between abelian groups with inverse \(f_*\).
\end{axiom}
\begin{definition}
If the categories \(\mathscr{P}_X\) and groups \(E_C^i(X,v)\) satisfy all the axioms above, then they are called a correspondence theory.
\end{definition}
\begin{remark}
The first example of a correspondence theory is just the correspondence of cycles
\[CH_C^i(X,v):=\textrm{the free abelian group generated by components of \(C\)},\]
where we pick \(\mathscr{P}_X=\mathbb{Z}/2\mathbb{Z}\). And unfortunately, to prove that MW-correspondence is a correspondence theory needs to rewrite everything from the begining, for example, the Rost-Schmid complexes, as we will see in the next section.
\end{remark}
\section{MW-Correspondence as a Correspondence Theory}\label{MW-Correspondence as a Correspondence Theory}
In this section, we are going to give a plan for the proof that MW-correspondence is a correspondence theory. It's incomplete and will be completed in the future.

We will always assume \(E=\widetilde{CH}\) in this section.

For any scheme \(X\) and \(x\in X\), set \(\Omega_x=m_x/m_x^2\) and \(\Lambda_x=det(m_x/m_x^2)\).
\begin{definition}
Let \(G\) be an abelian group, \(M\), \(N\) be \(G\)-sets, define
\[M\times_GN=M\times N/\sim,(m,n)\sim(m',n')\Longleftrightarrow(m,n)=(gm',g^{-1}n')\textrm{ for some }g\in G\]
as a \(G\)-sets.
\end{definition}
\begin{definition}
Let \(G\) be an abelian group and \(M\) be a \(G\)-set. We denote the group algebra of \(G\) over \(\mathbb{Z}\) by \(\mathbb{Z}[G]\) and the free abelian group generated by \(M\) by \(\mathbb{Z}[M]\). Then \(\mathbb{Z}[M]\) is a \(\mathbb{Z}[G]\)-module.
\end{definition}

The following lemma is straighforward.
\begin{lemma}
\begin{enumerate}
\item Let \(M\), \(N\) be \(G\)-sets, then
\[\mathbb{Z}[M\times_GN]\cong\mathbb{Z}[M]\otimes_{\mathbb{Z}[G]}\mathbb{Z}[N].\]
\item Let \(G\longrightarrow H\) be a morphism between abelian groups, \(M\) be a \(G\)-set, then
\[\mathbb{Z}[M\times_GH]\cong\mathbb{Z}[M]\otimes_{\mathbb{Z}[G]}\mathbb{Z}[H]\]
as \(\mathbb{Z}[H]\)-modules.
\end{enumerate}
\end{lemma}
\begin{definition}
Let \(R\) be a commutative ring, define \(Q(R)=R^*/(R^*)^2\) as an abelian group and for any one dimensional free \(R\)-module \(L\), define
\[Q(L)=(L\setminus \{0\})/\sim,x\sim y\Longleftrightarrow x=r^2y\textrm{ for some }r\in R^*\]
as a \(Q(R)\)-sets.
\end{definition}

The following lemma is straighforward.
\begin{lemma}
\begin{enumerate}
\item Let \(L_1\), \(L_2\) be one dimensional \(R\)-modules, then
\[Q(L_1\otimes_RL_2)\cong Q(L_1)\times_{Q(R)}Q(L_2).\]
\item Let \(L\) be a one dimensional \(R\)-module, then
\[Q(L^{\vee})\cong Hom_{Q(R)}(Q(L),Q(R)).\]
\item Let \(S\) be an \(R\)-algebra and \(L\) be a one dimensional \(R\)-module, then
\[Q(L\otimes_RS)\cong Q(L)\times_{Q(R)}Q(S)\]
as \(Q(S)\)-sets.
\end{enumerate}
\end{lemma}
\begin{definition}
Let \(X\) be a scheme, define a category \(\mathscr{P}_X\) with objects of the form \((E_1,\cdots,E_n)\), namely a series of vector bundles over \(X\). For any \((E_1,\cdots,E_n)\) and \((F_1,\cdots,F_m)\), if
\[rkE_1+\cdots+rkE_n\equiv rkF_1+\cdots+rkF_m (mod 2)\]
(\(rk=rank\)) define
\[Hom_{\mathscr{P}_X}((E_1,\cdots,E_n),(F_1,\cdots,F_m))\]
to be isomorphisms in
\[Hom_{O_X}(det(E_1,\cdots,E_n),det(F_1,\cdots,F_m)),\]
where
\[det(E_1,\cdots,E_n)=detE_1\otimes\cdots\otimes detE_n.\]
Otherwise define it to be empty.

The composition law is the same with that in the category of line bundles. (Here we define \(det(0)=O_X\))
\end{definition}
\begin{proposition}\label{px}
The categories \(\mathscr{P}_X\) for \(X\in Sm/k\) satisfy the Axiom \ref{T}.
\end{proposition}
\begin{proof}
From the definition of \(\mathscr{P}_X\), we see that for every \(A=(E_1,\cdots,E_n)\), \(rkE_1+\cdots+rkE_n\) is well defined in \(\mathbb{Z}/2\mathbb{Z}\), independent of isomorphisms in \(\mathscr{P}_X\). Hence there is a rank morphism \(rk_X:\mathscr{P}_X\longrightarrow\mathbb{Z}/2\mathbb{Z}\).

Define a bifunctor
\[\begin{array}{cccccc}+:&\mathscr{P}_X&\times&\mathscr{P}_X&\longrightarrow&\mathscr{P}_X\\&((E_1,\cdots,E_n)&,&(F_1,\cdots,F_m))&\longmapsto&(E_1,\cdots,E_n,F_1,\cdots,F_m)\end{array}.\]
Then it makes \(\mathscr{P}_X\) a Picard category with \(-(E_1,\cdots,E_n)=(E_n^{\vee},\cdots,E_1^{\vee})\). For any \(A, B\in\mathscr{P}_X\), we attach a commutativity isomorphism
\[c=c(A,B):A\oplus B\longrightarrow B\oplus A\]
by
\[(-1)^{rk_X(A)rk_X(B)}id_{det(A)\otimes det(B)}.\]
This makes \(\mathscr{P}_X\) a commutative Picard category.

There is a functor \(i:(Vect(X),iso)\longrightarrow\mathscr{P}_X\) sending \(E\) to \((E)\) and \(f:E_1\longrightarrow E_2\) to \(det(f)\). And for every exact sequence
\[0\longrightarrow E_1\longrightarrow E_3\longrightarrow E_2\longrightarrow 0,\]
we attach an isomorphism \((E_3)\longrightarrow (E_1,E_2)\) by the isomorphism \(detE_3\longrightarrow detE_1\otimes detE_2\) sending \(\alpha\wedge\beta\) to \(\alpha\otimes\beta\) for any local base \(\alpha\) (resp. \(\beta\)) of \(E_1\) (resp. \(E_3\)). This functor satisfies all conditions given in Definition \ref{bracket}.

Finally, for any \(f:X\longrightarrow Y\) in \(Sm/k\), we define \(f^*:\mathscr{P}_Y\longrightarrow\mathscr{P}_X\) by \(f^*(E_1,\cdots,E_n)=(f^*E_1,\cdots,f^*E_n)\).
\end{proof}

We set \(K_n^{MW}(F,L)=K_n^{MW}(F)\otimes_{\mathbb{Z}[Q(F)]}\mathbb{Z}[Q(L)]\) (See \cite[Remark 2.21]{Mo}) for every one dimensional \(F\)-vector space \(L\). For every \(X\in Sm/k\), \(x\in X\), \(T\) closed in \(X\) and \(v\in\mathscr{P}_X\), define
\[K_n^{MW}(k(x),\Lambda_x^*\otimes v)=K_n^{MW}(k(x),\Lambda_x^*\otimes_{k(x)}det(v)|_{k(x)})\]
and
\[C_{RS,T}^n(X;\underline{K}_m^{MW};v)=\bigoplus_{y\in X^{(n)}\cap T}K^{MW}_{m-n}(k(y),\Lambda_y^*\otimes v),\]
where \(X^{(n)}\) means the points of codimension \(n\) in \(X\) (See \cite[Chapter 4]{Mo}).

Now for every \(X\in Sm/k\), \(i\in\mathbb{N}\), \(v\in\mathscr{P}_X\) and \(T\) closed in \(X\), we define
\[\widetilde{CH}_T^i(X,v)=H^i(C_{RS,T}^*(X;\underline{K}_i^{MW};v))\]
to give Axiom \ref{C}. And the Axiom \ref{ES} just comes from the extension of supports in Chow-Witt rings.
\subsection{Operations without Intersection}
\begin{lemma}\label{l1}
Let \(f:X\longrightarrow X'\) be a smooth morphism in \(Sm/k\), \(x\in X\) with \([k(x):k(f(x))]<\infty\). Then we have an isomorphism
\[Q(\Lambda_x^*)\cong Q(\omega_{X/X'}^{\vee}|_{k(x)})\times_{Q(k(x))}Q(\Lambda_{f(x)}^*\otimes k(x))\]
(the \(Q\)s will be ignored in the sequel for convenience).
\end{lemma}
\begin{proof}
If \(k(x)\) is separable over \(k(f(x))\), then we have a commutative diagram with exact rows and columns
\[
	\xymatrix
	{
				&0								&0										&										&\\
				&\Omega_{X/X'}|_{k(x)}\ar[r]^=\ar[u]		&\Omega_{X/X'}|_{k(x)}\ar[u]					&										&\\
		0\ar[r]	&\Omega_x\ar[r]\ar[u]					&\Omega_{X/k}|_{k(x)}\ar[r]\ar[u]				&\Omega_{k(x)/k}\ar[r]						&0\\
		0\ar[r]	&\Omega_{f(x)}\otimes k(x)\ar[r]\ar[u]	&\Omega_{X'/k}|_{k(f(x))}\otimes k(x)\ar[r]\ar[u]	&\Omega_{k(f(x))/k}\otimes k(x)\ar[r]\ar[u]_{\cong}	&0\\
				&0\ar[u]							&0\ar[u]									&										&
	},
\]
so we have an isomorphism
\[\Lambda_x^*\cong\omega_{X/X'}|_{k(x)}^{\vee}\otimes(\Lambda_{f(x)}^*\otimes k(x))\]
which induces an isomophism
\[Q(\Lambda_x^*)\cong Q(\omega_{X/X'}|_{k(x)}^{\vee})\times_{Q(k(x))}Q(\Lambda_{f(x)}^*\otimes k(x)).\]
In general cases, we only have the horizontal exact sequences and the middle vertical arrows. But \(Q(\omega_{k(x)/k})\cong Q(\omega_{k(f(x))/k}\otimes k(x))\) still holds (See \cite[Lemma 4.1]{Mo}), so we have isomorphisms
\begin{align*}
	&Q(\Lambda_x^*)\\
\cong	&Q(\omega_{k(x)/k})\times_{Q(k(x))}Q(\omega_{k(x)/k}^{\vee})\times_{Q(k(x))}Q(\Lambda_x^*)\\
\cong	&Q(\omega_{k(x)/k})\times_{Q(k(x))}Q(\omega_{X/k}^{\vee}|_{k(x)})\\
\cong	&Q(\omega_{k(x)/k})\times_{Q(k(x))}Q(\omega_{X/X'}^{\vee}|_{k(x)})\times_{Q(k(x))}Q(\omega_{X'/k}^{\vee}|_{k(f(x))}\otimes k(x))\\
\cong	&Q(\omega_{k(x)/k})\times_{Q(k(x))}Q(\omega_{X/X'}^{\vee}|_{k(x)})\times_{Q(k(x))}Q(\omega_{k(f(x))/k}^{\vee}\otimes k(x))\times_{Q(k(x))}Q(\Lambda_{f(x)}^*\otimes k(x))\\
\cong	&Q(\omega_{k(x)/k})\times_{Q(k(x))}Q(\omega_{X/X'}^{\vee}|_{k(x)})\times_{Q(k(x))}Q(\omega_{k(x)/k}^{\vee})\times_{Q(k(x))}Q(\Lambda_{f(x)}^*\otimes k(x))\\
\cong	&Q(\omega_{X/X'}^{\vee}|_{k(x)})\times_{Q(k(x))}Q(\Lambda_{f(x)}^*\otimes k(x))
\end{align*}
This coincides with the isomorphism we got in the case of separate extension by applying Theorem \ref{four diagrams}, (2) to the diagram above.
\end{proof}
\begin{lemma}\label{l2}
Let \(f:X\longrightarrow X'\) be a closed immersion in \(Sm/k\) and \(x\in X\) (so \(k(x)=k(f(x))\)), then we have an isomorphism
\[\Lambda_{f(x)}^*\cong\Lambda_x^*\otimes detN_{X/X'}^{\vee}|_{k(x)}.\]
\end{lemma}
\begin{proof}
This follows by the following commutative diagram with exact rows and columns
\[
	\xymatrix
	{
				&0								&0							&							&\\
		0\ar[r]	&\Omega_x\ar[r]\ar[u]					&\Omega_{X/k}|_{k(x)}\ar[r]\ar[u]	&\Omega_{k(x)/k}\ar[r]			&0\\
		0\ar[r]	&\Omega_{f(x)}\ar[r]\ar[u]				&\Omega_{X'/k}|_{k(f(x))}\ar[r]\ar[u]	&\Omega_{k(f(x))/k}\ar[r]\ar[u]_=	&0\\
				&N_{X/X'}^{\vee}|_{k(x)}\ar[r]^=\ar[u]	&N_{X/X'}^{\vee}|_{k(x)}\ar[u]		&							&\\
				&0\ar[u]							&0\ar[u]						&							&
	}.
\]
\end{proof}
\begin{lemma}\label{l3}
Let \(f:X\longrightarrow X'\) be a smooth morphism in \(Sm/k\), \(x\in X\) with \(codim(x)=codim(f(x))\), then we have an isomorphism
\[\Omega_x\cong\Omega_{f(x)}\otimes_{k(f(x))}k(x).\]
\end{lemma}
\begin{proof}
Because the cotangent map
\[\Omega_{f(x)}\otimes_{k(f(x))}k(x)\longrightarrow\Omega_x\]
of \(f\) is injective and the two cotangent spaces have the same dimension \(codim(x)\).
\end{proof}
\begin{lemma}\label{l4}
Let \(X_1, X_2\in Sm/k\), \(x_1\in X_1\), \(x_2\in X_2\) and \(y\) be the generic point of some component of \(\overline{x_1}\times\overline{x_2}\). Then we have an isomorphism
\[\Omega_y\cong\Omega_{x_1}\otimes_{k(x_1)}k(y)\oplus\Omega_{x_2}\otimes_{k(x_2)}k(y).\]
\end{lemma}
\begin{proof}
This is because we have the following commutative diagram with exact rows and columns (same if we exchange \(X_1\) and \(X_2\))
\[
	\xymatrix
	{
				&0\ar[d]							&0\ar[d]								&0\ar[d]												&\\
		0\ar[r]	&\Omega_{x_1}\otimes{k(y)}\ar[r]\ar[d]	&p_1^*\Omega_{X_1/k}|_{k(y)}\ar[r]\ar[d]		&q_1^*\Omega_{\overline{x_1}/k}|_{k(y)}\ar[r]\ar[d]				&0\\
		0\ar[r]	&\Omega_y\ar[r]\ar[d]					&\Omega_{X_1\times X_2/k}|_{k(y)}\ar[r]\ar[d]	&\Omega_{\overline{x_1}\times\overline{x_2}/k}|_{k(y)}\ar[r]\ar[d]	&0\\
		0\ar[r]	&\Omega_{x_2}\otimes{k(y)}\ar[r]\ar[d]	&p_2^*\Omega_{X_2/k}|_{k(y)}\ar[r]\ar[d]		&q_2^*\Omega_{\overline{x_2}/k}|_{k(y)}\ar[r]\ar[d]				&0\\
				&0								&0									&0,													&
	}
\]
where \(p_i:X_1\times X_2\longrightarrow X_i\) and \(q_i:\overline{x_1}\times\overline{x_2}\longrightarrow\overline{x_i}\) are projections and \(\Omega_{\overline{x_1}\times\overline{x_2}/k}|_{k(y)}=\Omega_{\overline{y}/k}|_{k(y)}\).
\end{proof}
\begin{definition}
(See Axiom \ref{PB}) Let \(f:X\longrightarrow X'\) be a smooth morphism, \(x\in X\) with \(codim(x)=codim(f(x))\) and \(v\in\mathscr{P}_{X'}\) we have a (obvious) morphism
\[K^{MW}_{n}(k(f(x)),\Lambda_x^*\otimes v)\longrightarrow K^{MW}_{n}(k(x),\Lambda_{f(x)}^*\otimes f^*v)\]
by Lemma \ref{l3}. This induces a pull-back morphism (See \cite[Corollaire 10.4.2]{F})
\[f^*:\widetilde{CH}^n_T(X',v)\longrightarrow\widetilde{CH}^n_{f^{-1}(T)}(X,f^*(v))\]
for every \(T\in Z^n(X)\). It is functorial with respect to \(v\).
\end{definition}
\begin{remark}
The pull-back along closed immersions is much more difficult and we will discuss this in Section \ref{Intersection with Divisors}.
\end{remark}

The following is obvious.
\begin{proposition}
(See Axiom \ref{FPB}) The pull-back between smooth morphisms is functorial. And \(f^*(1)=1\).
\end{proposition}
\begin{definition}\label{p-fsm}
(See Axiom \ref{P-FSM}) Let \(f:X\longrightarrow X'\) be a smooth morphism, \(C\in Z^{i+d_f}(X)\) being finite over \(X'\), we define the push-forward (See Proposition \ref{partial maps and push-forwards for smooth morphisms})
\[f_*:\widetilde{CH}^{i+d_f}_C(X,f^*v-T_{X/X'})\longrightarrow\widetilde{CH}^i_{f(C)}(X',v)\]
by the composition
\[\xymatrix{K_0^{MW}(k(x),\Lambda_x^*\otimes f^*v\otimes\omega_{X/X'})\ar[r]&K_0^{MW}(k(x),\omega_{X/X'}^{\vee}\otimes\Lambda_{f(x)}^*\otimes f^*v\otimes\omega_{X/X'})\ar[d]\\&K_0^{MW}(k(f(x)),\Lambda_{f(x)}^*\otimes f^*v)}\]
for every \(x\in C\cap X^{(i+d_f)}\), where the last arrow is the cancellation morphism between the first and fourth term. Note that we have used Lemma \ref{l1} in case of general finite extensions.

The push-forward for smooth morphisms is functorial with respect to \(v\).
\end{definition}

It's clear that Axiom \ref{EE} is satisfied from the definitions.
\begin{definition}\label{p-fci}
(See Axiom \ref{P-FCI}) Let \(f:X\longrightarrow X'\) be a closed immersion, \(C\in Z^{i+d_f}(X)\), we define the push-forward (See Proposition \ref{partial maps and push-forwards for closed immersions})
\[f_*:\widetilde{CH}^{i+d_f}_C(X,N_{X/X'}+f^*v)\longrightarrow\widetilde{CH}^i_{f(C)}(X',v)\]
by the isomorphism
\[\xymatrix{K_0^{MW}(k(x),\Lambda_x^*\otimes detN_{X/X'}\otimes f^*v)\ar[r]&K_0^{MW}(k(f(x)),\Lambda_{f(x)}^*\otimes f^*v)}\]
for every \(x\in C\cap X^{(i+d_f)}\). Here we have used Lemma \ref{l2}.

The push-forward for closed immersions is functorial with respect to \(v\).
\end{definition}
\begin{remark}
Suppose \(f:X\longrightarrow X'\), \(C\in Z^{i+d_f}(X)\) and \(C=\overline{x}\). If \(f\) is a smooth morphism and \(C\) is closed in \(X'\), we have an exact sequence
\[0\longrightarrow T_{X/X'}|_{k(x)}\longrightarrow\Lambda_x^*\longrightarrow\Lambda_{f(x)}^*\longrightarrow 0;\]
if \(f\) is a closed immersion, we have an exact sequence
\[0\longrightarrow\Lambda_x^*\longrightarrow\Lambda_{f(x)}^*\longrightarrow N_{X/X'}|_{k(x)}\longrightarrow 0.\]
So we can identify \(\Lambda_x^*\) with \(N_{\overline{x}/X}|_{k(x)}\) since the latter has the same exact sequences when \(C\) is smooth. Hence in the context above, the push-forward of \(f\) under the support \(C\) is completely determined by the composition
\[N_{C/X}+f^*v|_C-T_{X/X'}|_C\longrightarrow T_{X/X'}|_C+N_{C/Y}+f^*v|_C-T_{X/X'}|_C\longrightarrow N_{C/Y}+f^*v|_C\]
in case of \(f\) being smooth and by the isomorphism
\[N_{C/X}+N_{X/X'}|_C+f^*v|_C\longrightarrow N_{C/Y}+f^*v|_C\]
if \(f\) is a closed immersion.

This inspires us to convert equations of twisted Chow-Witt groups into equations of virtual objects. And then use the method described in Section \ref{Virtual Objects and Their Calculation}. This is the main idea we will use in this section.
\end{remark}
Now let's explain the differential maps in Rost-Schmid complexes. Suppose \(X\in Sm/k\), \(Y=\overline{y}, y\in X\) and \(Z=\overline{z}, z\in Y^{(1)}\) and \(v\in\mathscr{P}_X\). We want to define the differential map
\[\partial^y_z:K_n^{MW}(k(y),\Lambda_y^*\otimes v)\longrightarrow K_{n-1}^{MW}(k(z),\Lambda_z^*\otimes v).\]

Suppose at first \(Y\) is normal. Then the exact sequence
\[I_Y/I_Y^2\longrightarrow\Omega_{X/k}|_Y\longrightarrow\Omega_{Y/k}\longrightarrow 0\]
is also left exact at the stalk of \(z\), hence we have a commutative diagram with exact rows
\[
	\xymatrix
	{
		0\ar[r]	&I_Y/I_Y^2|_{k(z)}\ar[r]\ar[d]_i	&\Omega_{X/k}|_{k(z)}\ar[r]\ar[d]_{id}	&\Omega_{Y/k}|_{k(z)}\ar[r]\ar[d]	&0\\
		0\ar[r]	&I_Z/I_Z^2|_{k(z)}\ar[r]			&\Omega_{X/k}|_{k(z)}\ar[r]			&\Omega_{Z/k}|_{k(z)}\ar[r]		&0
	}.
\]
The map \(i\) is injective with the cokernel \(m_z/m_z^2\), where \(m_z\) is the maximal ideal of \(O_{Y,z}\), hence we have an exact sequence
\[0\longrightarrow (m_z/m_z^2)^{\vee}\longrightarrow(I_Z/I_Z^2)^{\vee}|_{k(z)}\longrightarrow(I_Y/I_Y^2)^{\vee}|_{k(z)}\longrightarrow 0.\]
Now choose a free basis \(a\) of \((I_Y/I_Y^2)^{\vee}_z\), \(e\) of \((m_z/m_z^2)^{\vee}\) and \(t\) of \(v_z\). Hence \(a\) is also a free basis of \(\Omega_y^*=(I_Y/I_Y^2)^{\vee}|_{k(y)}\) and \((e,a)\) is a free basis of \(\Omega_z^*=(I_Z/I_Z^2)^{\vee}|_{k(z)}\) by the sequence above. We define the map \(\partial\) by
\[
	\begin{array}{ccc}
		K_n^{MW}(k(y),\Lambda_y^*\otimes v)	&\longrightarrow	&K_{n-1}^{MW}(k(z),\Lambda_z^*\otimes v)\\
		s\otimes a\otimes t							&\longmapsto		&\partial^e_z(s)\otimes(e\wedge a)\otimes t
	\end{array},
\]
where \(\partial^e_z\) is the usual partial map of Milnor-Witt groups. This map is independent of the choice of \(a, e, t\).

For general cases, let \(\widetilde{Y}\) be the normalization of \(Y\) with the natural map \(\pi:\widetilde{Y}\longrightarrow Y\) and let \(\{z_i\}=\pi^{-1}(z)\). We have an isomorphism (the same for \(z\))
\[\Lambda_y^*\cong\omega_{k(y)/k}\otimes\omega_{X/k}^{\vee}|_{k(y)}.\]
Now fix an \(i\). We find that \(\Omega_{O_{\widetilde{Y},z_i}/k}\) satisfies
\[\Omega_{O_{\widetilde{Y},z_i}/k}\otimes k(y)=\Omega_{k(y)/k}.\]
And we also have an exact sequence
\[0\longrightarrow m_{z_i}/m_{z_i}^2\longrightarrow\Omega_{O_{\widetilde{Y},z_i}/k}\otimes k(z_i)\longrightarrow\Omega_{k(z_i)/k}\longrightarrow 0.\]
So choose a free basis \(e_i\) of \((m_{z_i}/m_{z_i}^2)^{\vee}\), \(c_i\) of \(\Omega_{O_{\widetilde{Y},z_i}/k}\), \(d\) of \((\Omega_{X/k}^{\vee})_z\) and \(l\) of \(v_z\). We define \(\partial_i\) by the following compositions
\begin{align*}
			&K_n^{MW}(k(y),\Lambda_y^*\otimes v)\\
\longrightarrow	&K_n^{MW}(k(y),\omega_{k(y)/k}\otimes\omega_{X/k}^{\vee}\otimes v)\\
\longrightarrow	&K_{n-1}^{MW}(k(z_i),(m_{z_i}/m_{z_i}^2)^{\vee}\otimes(\omega_{O_{\widetilde{Y},z_i}/k}\otimes k(z_i))\otimes_{k(z_i)}(\omega_{X/k}^{\vee}|_{k(z)}\otimes_{k(z)}v|_{k(z)}))\\
\longrightarrow	&K_{n-1}^{MW}(k(z_i),\omega_{k(z_i)/k}\otimes_{k(z_i)}(\omega_{X/k}^{\vee}|_{k(z)}\otimes_{k(z)}v|_{k(z)}))\\
\longrightarrow	&K_{n-1}^{MW}(k(z_i),(\omega_{k(z)/k}\otimes k(z_i))\otimes_{k(z_i)}(\omega_{X/k}^{\vee}|_{k(z)}\otimes_{k(z)}v|_{k(z)}))\\
\longrightarrow	&K_{n-1}^{MW}(k(z_i),(\omega_{k(z)/k}\otimes_{k(z)}\omega_{X/k}^{\vee}|_{k(z)}\otimes_{k(z)}v|_{k(z)})\otimes_{k(z)}k(z_i))\\
\longrightarrow	&K_{n-1}^{MW}(k(z),\omega_{k(z)/k}\otimes_{k(z)}\omega_{X/k}^{\vee}|_{k(z)}\otimes_{k(z)}v|_{k(z)})\\
\longrightarrow	&K_{n-1}^{MW}(k(z),\Lambda_z^*\otimes v),
\end{align*}
where the second arrow is defined by
\[s\otimes c_i\otimes d\otimes l\longmapsto\partial^{e_i}_{z_i}(s)\otimes e_i\otimes(c_i\otimes 1)\otimes d\otimes l,\]
which is independent of the choice of \(e_i\). Then we define \(\partial^y_z=\sum\partial_i\). This definition coincides with the definition just given when \(Y\) is normal by applying Theorem \ref{four diagrams}, (4) to the following commutative diagram with exact columns and rows
\[
	\xymatrix
	{
				&						&0\ar[d]							&0\ar[d]							&\\
		0\ar[r]	&T_{Z/k}|_{k(z)}\ar[r]\ar[d]_=	&T_{Y/k}|_{k(z)}\ar[d]\ar[r]			&(m_z/m_z^2)^{\vee}\ar[r]\ar[d]			&0\\
		0\ar[r]	&T_{Z/k}|_{k(z)}\ar[r]		&T_{X/k}|_{k(z)}\ar[r]\ar[d]			&(I_Z/I_Z^2)^{\vee}|_{k(z)}\ar[r]\ar[d]	&0\\
				&						&(I_Y/I_Y^2)^{\vee}|_{k(z)}\ar[d]\ar[r]^=	&(I_Y/I_Y^2)^{\vee}|_{k(z)}\ar[d]		&\\
				&						&0								&0								&
	}.
\]
\begin{remark}\label{linearity of partial maps}
Here we would like to treat a kind of linearity of the \(\partial^y_z\). Suppose \(s\in K_n^{MW}(k(y),\Lambda_y^*\otimes v)\).
\begin{enumerate}
\item Suppose \(f\in O_{Y,z}^*\) and \(n=0\), we want to show that
\[\partial^y_z([f]s)=[\overline{f}]\partial^y_z(s).\]
It suffices to show the formula for each \(\partial_i\). We see that \(\partial_i=Tr^{k(z_i)}_{k(z)}\circ\partial^{e_i}_{z_i}\) (the operation of twists could be ignored). And \(\partial^{e_i}_{z_i}([f]s)=\epsilon[\overline{f}]\partial^{e_i}_{z_i}(s)\). Suppose \(\partial^{e_i}_{z_i}(s)=<a>\eta\). Then
\begin{align*}
	&Tr^{k(z_i)}_{k(z)}(\epsilon[\overline{f}]<a>\eta)\\
=	&Tr^{k(z_i)}_{k(z)}((<\overline{f}>-<1>)<a>)\\
=	&(<\overline{f}>-<1>)Tr^{k(z_i)}_{k(z)}(<a>)\\
=	&[\overline{f}]Tr^{k(z_i)}_{k(z)}(<a>\eta).
\end{align*}
Then the claim is proved.
\item If we have another line bundle \(\mathscr{M}\) over \(X\) and \(m\) is a free basis of \(\mathscr{M}_z\) (so it's also a free basis of \(\mathscr{M}_y\)), we have
\[\partial^y_z(s\otimes m)=\partial^y_z(s)\otimes m.\]
But we have to note that this doesn't hold for general free basis of \(\mathscr{M}_y\). Since if we replace \(m\) by \(\lambda\cdot m\), where \(\lambda\in k(y)^*\), then \(\lambda\) has to be moved into \(K_n^{MW}(k(y))\) for computation and \(\lambda\) may have valuation at \(z_i\).
\end{enumerate}
\end{remark}
\begin{remark}\label{naturality of partial maps}
It's obvious that any morphism \(v_1\longrightarrow v_2\) in \(\mathscr{P}_X\) will induce an isomorphism between corresponding Rost-Schmid complexes.
\end{remark}
\begin{definition}\label{exterior products}
(See \cite{CF1}) Let \(X_a\in Sm/k, a=1,2\), \(x_a\in X_a\), \(y\) be the generic point of some component of \(\overline{x_1}\times\overline{x_2}\). For every \(s_a\in K_n^{MW}(k(x_a),\Lambda_{x_a}^*\otimes v_a)\), we define
\[s_1\times s_2=\sum_{y}c(p_1^*(v_1),p_2^*(\Lambda_{x_2}^*))(p_1^*(s_1)\otimes p_2^*(s_2))\in\oplus_{y}K_{n+m}^{MW}(k(y),\Lambda_{y}^*\otimes(p_1^*(v_1)+p_2^*(v_2))),\]
where \(p_i:\overline{y}\longrightarrow\overline{x_i}\) is the projection (here we've used Lemma \ref{l4}). It is called the exterior product between \(s_1\) and \(s_2\).

The exterior product is functorial with respect to twists and extension of supports.
\end{definition}
We will denote \(p_1^*(v_1)+p_2^*(v_2)\) by \(v_1\times v_2\) for convenience.

Now we do a special case of the proof that the right exterior product with an element in supported Chow-Witt groups is a chain complex map between Rost-Schmid complexes, while the left exterior product is not.
\begin{proposition}\label{exterior products and differential maps}
Let \(X, X'\in Sm/k\), \(v\in\mathscr{P}_X\), \(v'\in\mathscr{P}_{X'}\) and \(Y\in Z^i(X)\), \(T\in Z^j(X')\) be smooth. Suppose \(\beta\in\widetilde{CH}^j_T(X',v')\). Then the following diagram commutes
\[
	\xymatrix
	{
		\oplus_{s\in(Y\times T)^{(0)}}K_n^{MW}(k(s),\Lambda^*_s\otimes(v\times v'))\ar[r]^-{\partial}	&\oplus_{u\in(X\times X')^{(i+j+1)}}K_{n-1}^{MW}(k(u),\Lambda^*_u\otimes(v\times v'))\\
		\oplus_{y\in Y^{(0)}}K_n^{MW}(k(y),\Lambda^*_y\otimes v)\ar[r]^-{\partial}\ar[u]_{\times\beta}	&\oplus_{z\in Y\cap X^{(i+1)}}K_{n-1}^{MW}(k(z),\Lambda^*_z\otimes v)\ar[u]_{\times\beta}
	}.
\]
That is, for every \(\beta\in\widetilde{CH}^j_T(X',v')\) and \(\alpha\in\oplus_{y\in Y^{(0)}}K_n^{MW}(k(y),\Lambda^*_y\otimes v)\), we have
\[\partial(\alpha\times\beta)=\partial(\alpha)\times\beta.\]
And moreover, we have
\[\partial(\beta\times\alpha)=<-1>^{j+rk_{X'}(v')}\beta\times\partial(\alpha).\] 
\end{proposition}
\begin{proof}
We may assume \(Y\) and \(T\) are irreducible. We check the commutativity after projecting to each \(u\in(X\times X')^{(i+j+1)}\). It suffices to let \(u\) be a generic point of \(\overline{z}\times T\), where \(z\in Y\cap X^{(i+1)}\), since otherwise both paths vanish. Set \(Z=\overline{z}\). We have a commutative diagram with exact columns and rows (we write \(X\times Y\) by \(XY\) for short)
\[
	\xymatrix
	{
				&0\ar[d]					&0\ar[d]					&							&\\
				&N_{ZT/YT}\ar[d]\ar[r]^=		&N_{ZT/YT}\ar[d]			&							&\\
		0\ar[r]	&N_{ZT/XT}\ar[r]\ar[d]		&N_{ZT/XX'}\ar[d]\ar[r]		&N_{XT/XX'}|_{ZT}\ar[r]\ar[d]_{\cong}	&0\\
		0\ar[r]	&N_{YT/XT}|_{ZT}\ar[r]\ar[d]	&N_{YT/XX'}|_{ZT}\ar[r]\ar[d]	&N_{YT/YX'}|_{ZT}\ar[r]		&0\\
				&0						&0						&							&
	}.
\]
We have projections maps \(p_1:ZT\longrightarrow Z\) and \(p_2:ZT\longrightarrow T\). By Theorem \ref{four diagrams}, (1), we have a commutative diagram
\[
	\xymatrix
	{
		p_1^*(N_{Z/Y}+N_{Y/X}|_Z+v|_Z)+p_2^*(N_{T/X'}+v'|_T)\ar[r]\ar[d]	&p_1^*(N_{Z/Y})+N_{YT/XX'}|_{ZT}+p_1^*(v|_Z)+p_2^*(v'|_T)\ar[d]\\
		p_1^*(N_{Z/X}+v|_Z)+p_2^*(N_{T/X'}+v'|_T)\ar[d]					&N_{ZT/YT}+N_{YT/XX'}|_{ZT}+p_1^*(v|_Z)+p_2^*(v'|_T)\ar[d]\\
		N_{ZT/XT}+p_1^*(v|_Z)+N_{YT/YX'}|_{ZT}+p_2^*(v'|_T)\ar[r]			&N_{ZT/XX'}+p_1^*(v|_Z)+p_2^*(v'|_T)
	},
\]
which gives the first equation. For the second one, we could calculate directly using the first equation by using Proposition \ref{commutativity of exterior product} (which still holds in this context):
\begin{align*}
	&\partial(\beta\times\alpha)\\
=	&\partial(<-1>^{(i+rk_X(v))(j+rk_{X'}(v'))}c(q_1^*v_1,q_2^*v_2)(\alpha\times\beta))\\
=	&<-1>^{(i+rk_X(v))(j+rk_{X'}(v'))}c(q_1^*v_1,q_2^*v_2)(\partial(\alpha\times\beta))\\
=	&<-1>^{(i+rk_X(v))(j+rk_{X'}(v'))}c(q_1^*v_1,q_2^*v_2)(\partial(\alpha)\times\beta)\\
=	&<-1>^{j+rk_{X'}(v')}\beta\times\partial(\alpha),
\end{align*}
where \(q_1\), \(q_2\) are projections of \(X\times X'\).
\end{proof}
\begin{definition}
The exterior product in Definition \ref{exterior products} induces a pairing
\[\widetilde{CH}^{n_1}_{T_1}(X_1,v_1)\times\widetilde{CH}^{n_2}_{T_2}(X_2,v_2)\longrightarrow\widetilde{CH}^{n_1+n_2}_{T_1\times T_2}(X_1\times X_2,v_1\times v_2)\]
for every \(X_a\in Sm/k\), \(T_a\in Z^{n_a}(X_a)\) and \(v_a\in\mathscr{P}_{X_a}\), \(a=1,2\) (at least for the case when \(T_a\) are smooth) by Proposition \ref{exterior products and differential maps}. It's called the exterior product between Chow-Witt groups.
\end{definition}
\begin{proposition}\label{commutativity of exterior product}
(See Axiom \ref{A} and \ref{CC}) In the context above, the exterior product is associative and satisfies
\[s_1\times s_2=<-1>^{(codim(x_1)+rk_{X_1}(v_1))(codim(x_2)+rk_{X_2}(v_2))}c(p_2^*(v_2),p_1^*(v_1))(s_2\times s_1)\]
where \(s_a\in\widetilde{CH}^{n_a}_{T_a}(X_a,v_a)\).
\end{proposition}
\begin{proof}
Associativity comes from Definition \ref{bracket}, (3) and the second statement follows from the definition of commutativity isomorphism in Proposition \ref{px}.
\end{proof}
\begin{proposition}\label{compability of exterior product}
(See Axiom \ref{CPB}) Let \(f_a:Y_a\longrightarrow X_a, a=1,2\) be a smooth morphism in \(Sm/k\), then in the context above, we have
\[(f_1\times f_2)^*(s_1\times s_2)=f_1^*(s_1)\times f_2^*(s_2).\]
\end{proposition}
\begin{proof}
This follows from Lemma \ref{l3} and Lemma \ref{l4}.
\end{proof}
\begin{remark}
If we have defined pull-back along any closed immersion, then the Axiom \ref{A}, \ref{CC} and \ref{CPB} can be deduced from the above two propositions.
\end{remark}
Here we would like to prove a special case that the the push-forwards defined in Definition \ref{p-fsm} and Definition \ref{p-fci} form a chain complex morphism between Rost-Schmid complexes, just to explain how to treat the twists.
\begin{proposition}\label{partial maps and push-forwards for smooth morphisms}
Suppose \(Z\subseteq Y\subseteq X\) are schemes with \(X\) and \(Y\) being smooth, \(Y=\overline{y}\) closed irreducible in \(X\), \(Z=\overline{z}\) closed irreducible in \(Y\) and \(Z\in Y^{(1)}\). Suppose \(f:X\longrightarrow X'\) is a smooth morphism, \(v\in\mathscr{P}_{X'}\) and \(Y\) is also a closed subset of \(X'\). Then we have a commutative diagram
\[
	\xymatrix
	{
		K_n^{MW}(k(y),\Lambda_y^*\otimes f^*v\otimes\omega_{X/X'})\ar[r]^{\partial}\ar[d]_{f_*}	&K_{n-1}^{MW}(k(z),\Lambda_z^*\otimes f^*v\otimes\omega_{X/X'})\ar[d]_{f_*}\\
		K_n^{MW}(k(f(y)),\Lambda_{f(y)}^*\otimes v)\ar[r]^{\partial}						&K_{n-1}^{MW}(k(f(z)),\Lambda_{f(z)}^*\otimes v)
	}.
\]
\end{proposition}
\begin{proof}
We have the following commutative diagram with exact rows and columns
\[
	\xymatrix
	{
				&					&0\ar[d]				&0\ar[d]				&\\
				&					&N_{Z/Y}\ar[d]\ar[r]^=	&N_{Z/Y}\ar[d]			&\\
		0\ar[r]	&T_{X/X'}|_Z\ar[r]\ar[d]_=	&N_{Z/X}\ar[d]\ar[r]		&N_{Z/X'}\ar[r]\ar[d]		&0\\
		0\ar[r]	&T_{X/X'}|_Z\ar[r]		&N_{Y/X}|_Z\ar[r]\ar[d]	&N_{Y/X'}|_Z\ar[r]\ar[d]	&0\\
				&					&0					&0					&
	}.
\]
Now the statement is to prove the following diagram commutes
\[
	\xymatrix
	{
		N_{Z/Y}+N_{Y/X}|_Z+f^*v|_Z-T_{X/X'}|_Z\ar[r]\ar[d]	&N_{Z/Y}+T_{X/X'}|_Z+N_{Y/X'}|_Z+f^*v|_Z-T_{X/X'}|_Z\ar[d]\\
		N_{Z/X}+f^*v|_Z-T_{X/X'}|_Z\ar[d]				&N_{Z/Y}+N_{Y/X'}|_Z+f^*v|_Z\ar[d]\\
		T_{X/X'}|_Z+N_{Z/X'}+f^*v|_Z-T_{X/X'}|_Z\ar[r]		&N_{Z/X'}+f^*v|_Z
	}.
\]
We have the following commutative diagrams
\[
	\xymatrix
	{
		T_{X/X'}|_Z+N_{Z/X'}+f^*v|_Z-T_{X/X'}|_Z\ar[r]				&N_{Z/X'}+f^*v|_Z\\
		T_{X/X'}|_Z+N_{Z/Y}+N_{Y/X'}|_Z+f^*v|_Z-T_{X/X'}|_Z\ar[r]\ar[u]	&N_{Z/Y}+N_{Y/X'}|_Z+f^*v|_Z\ar[u]
	}
\]
\[
	\xymatrix
	{
		N_{Z/Y}+N_{Y/X}|_Z+f^*v|_Z-T_{X/X'}|_Z\ar[r]\ar[d]	&N_{Z/Y}+T_{X/X'}|_Z+N_{Y/X'}|_Z+f^*v|_Z-T_{X/X'}|_Z\ar[d]\\
		N_{Z/X}+f^*v|_Z-T_{X/X'}|_Z\ar[d]				&T_{X/X'}|_Z+N_{Z/Y}+N_{Y/X'}|_Z+f^*v|_Z-T_{X/X'}|_Z\\
		T_{X/X'}|_Z+N_{Z/X'}+f^*v|_Z-T_{X/X'}|_Z\ar[ru]	&
	},
\]
where the second one comes from Theorem \ref{four diagrams}, (3). Then the result follows by combining the two diagrams above.
\end{proof}
\begin{proposition}\label{partial maps and push-forwards for closed immersions}
Suppose \(Z\subseteq Y\subseteq X\) are schemes with \(X\) and \(Y\) smooth, \(Y=\overline{y}\) closed irreducible in \(X\), \(Z=\overline{z}\) closed irreducible in \(Y\) and \(Z\in Y^{(1)}\). Suppose \(f:X\longrightarrow X'\) is a closed immersion and \(v\in\mathscr{P}_{X'}\). Then we have a commutative diagram
\[
	\xymatrix
	{
		K_n^{MW}(k(y),\Lambda_y^*\otimes detN_{X/X'}\otimes f^*v)\ar[r]^{\partial}\ar[d]_{f_*}	&K_{n-1}^{MW}(k(z),\Lambda_z^*\otimes detN_{X/X'}\otimes f^*v)\ar[d]_{f_*}\\
		K_n^{MW}(k(f(y)),\Lambda_{f(y)}^*\otimes v)\ar[r]^{\partial}						&K_{n-1}^{MW}(k(f(z)),\Lambda_{f(z)}^*\otimes v)
	}.
\]
\end{proposition}
\begin{proof}
The diagram commutes because of the following commutative diagram by Definition \ref{bracket}, (3)
\[
	\xymatrix
	{
		N_{Z/Y}+N_{Y/X}|_Z+N_{X/X'}|_Z+f^*v|_Z\ar[r]\ar[d]	&N_{Z/X}+N_{X/X'}|_Z+f^*v|_Z\ar[d]\\
		N_{Z/Y}+N_{Y/X'}|_Z+f^*v|_Z\ar[r]				&N_{Z/X'}+f^*v|_Z
	}.
\]
\end{proof}

Since we haven't defined pull-back along arbitrary closed immersions, the following definition is just an intention and won't be used.
\begin{definition}
(See Axiom \ref{P}) Let \(X\in Sm/k\), \(T_a\in Z^{n_a}(X)\), \(n_a\in\mathbb{N}\) and \(v_a\in\mathscr{P}_X\), \(a=1,2\). If \(T_1\) and \(T_2\) intersect properly, we have a product
\[\widetilde{CH}^{n_1}_{T_1}(X,v_1)\times\widetilde{CH}^{n_2}_{T_2}(X,v_2)\longrightarrow\widetilde{CH}^{n_1+n_2}_{T_1\cap T_2}(X,v_1+v_2)\]
defined by \(a\cdot b=\triangle^*(a\times b)\) where \(\triangle:X\longrightarrow X\times X\) is the diagonal. It is functorial with respect to twists and extension of supports by the same property of the exterior product.
\end{definition}
\begin{proposition}\label{functoriality of push-forwards}
Let \(\xymatrix{X\ar[r]^f&Y\ar[r]^g&Z}\) be morphisms in \(Sm/k\), \(v\in\mathscr{P}_Z\) and \(C\in Z^{i+d_{g\circ f}}(X)\).
\begin{enumerate}
\item(See Axiom \ref{FPFSM}) Suppose \(f\), \(g\) are smooth and \(C\) is a closed subset \(Z\), then the following diagram commutes
\[
	\xymatrix
	{
		\widetilde{CH}^{i+d_{g\circ f}}_C(X,(g\circ f)^*v-T_{X/Z})\ar[r]\ar[ddr]^{(g\circ f)_*}	&\widetilde{CH}^{i+d_{g\circ f}}_C(X,(g\circ f)^*v-f^*T_{Y/Z}-T_{X/Y})\ar[d]_{f_*}\\
																					&\widetilde{CH}^{i+d_g}_{f(C)}(Y,g^*v-T_{Y/Z})\ar[d]_{g_*}\\
																					&\widetilde{CH}^i_{g(f(C))}(Z,v)
	}.
\]
\item(See Axiom \ref{FPFCI}) Suppose \(f\), \(g\) are closed immersions, then the following diagram commutes
\[
	\xymatrix
	{
		\widetilde{CH}^{i+d_{g\circ f}}_C(X,N_{X/Z}+(g\circ f)^*v)\ar[r]\ar[ddr]^{(g\circ f)_*}	&\widetilde{CH}^{i+d_{g\circ f}}_C(X,N_{X/Y}+f^*N_{Y/Z}+(g\circ f)^*v)\ar[d]_{f_*}\\
																					&\widetilde{CH}^{i+d_g}_{f(C)}(Y,N_{Y/Z}+g^*v)\ar[d]_{g_*}\\
																					&\widetilde{CH}^i_{g(f(C))}(Z,v)
	}.
\]
\item(See Axiom \ref{CTPF}, (1)) Suppose \(f\) is a closed immersion, \(g\), \(g\circ f\) are smooth and \(C\) is a closed subset of \(Z\). Then the following diagram commutes
\[
	\xymatrix
	{
		\widetilde{CH}^{i+d_{g\circ f}}_C(X,N_{X/Y}+f^*g^*v-f^*T_{Y/Z})\ar[d]_{f_*}\ar[r]	&\widetilde{CH}^{i+d_{g\circ f}}_C(X,f^*g^*v+N_{X/Y}-f^*T_{Y/Z})\ar[d]\\
		\widetilde{CH}_{f(C)}^{i+d_g}(Y,g^*v-T_{Y/Z})\ar[d]_{g_*}								&\widetilde{CH}^{i+d_{g\circ f}}_C(X,f^*g^*v-T_{X/Z})\ar[dl]_{(g\circ f)_*}\\
		\widetilde{CH}_{g(f(C))}^i(Z,v)															&
	}.
\]
\item(See Axiom \ref{CTPF}, (2)) Suppose \(g\) is smooth, \(f\), \(g\circ f\) are closed immersions. Then the following diagram commutes
\[
	\xymatrix
	{
		\widetilde{CH}^{i+d_{g\circ f}}_C(X,N_{X/Y}+f^*g^*v-f^*T_{Y/Z})\ar[d]_{f_*}\ar[r]	&\widetilde{CH}^{i+d_{g\circ f}}_C(X,-f^*T_{Y/Z}+N_{X/Y}+f^*g^*v)\ar[d]\\
		\widetilde{CH}_{f(C)}^{i+d_g}(Y,g^*v-T_{Y/Z})\ar[d]_{g_*}								&\widetilde{CH}^{i+d_{g\circ f}}_C(X,N_{X/Z}+f^*g^*v)\ar[dl]_{(g\circ f)_*}\\
		\widetilde{CH}_{g(f(C))}^i(Z,v)															&
	}.
\]
\end{enumerate}
\end{proposition}
\begin{proof}
\begin{enumerate}
\item This follows by the following commutative diagram
\[
	\xymatrix
	{
		N_{C/X}+f^*g^*v|_C-T_{X/Z}|_C\ar[r]\ar[dddd]	&N_{C/X}+f^*g^*v|_C-f^*T_{Y/Z}|_C-T_{X/Y}|_C\ar[d]\\
											&T_{X/Y}|_C+N_{C/Y}+f^*g^*v|_C-f^*T_{Y/Z}|_C-T_{X/Y}|_C\ar[d]\\
											&N_{C/Y}+f^*g^*v|_C-f^*T_{Y/Z}|_C\ar[d]\\
											&f^*T_{Y/Z}|_C+N_{C/Z}+f^*g^*v|_C-f^*T_{Y/Z}|_C\ar[d]\\
		T_{X/Z}|_C+N_{C/Z}+f^*g^*v|_C-T_{X/Z}|_C\ar[r]	&N_{C/Z}+f^*g^*v|_C
	}
\]
by using Definition \ref{bracket}, (3).
\item Essentially the same as in (1).
\item
We are going to prove that the following diagram commutes
\[
	\xymatrix
	{
	 	N_{C/X}+N_{X/Y}|_C+f^*g^*v|_C-f^*T_{Y/Z}|_C\ar[r]\ar[d]	&N_{C/X}+f^*g^*v|_C+N_{X/Y}|_C-f^*T_{Y/Z}|_C\ar[d]\\
	 	N_{C/Y}+f^*g^*v|_C-f^*T_{Y/Z}|_C\ar[d]				&N_{C/X}+f^*g^*v|_C+N_{X/Y}|_C-N_{X/Y}|_C-T_{X/Z}|_C\ar[d]\\
	 	f^*T_{Y/Z}|_C+N_{C/Z}+f^*g^*v|_C-f^*T_{Y/Z}|_C\ar[d]		&N_{C/X}+f^*g^*v|_C-T_{X/Z}|_C\ar[d]\\
	 	N_{C/Z}+f^*g^*v|_C 								&T_{X/Z}|_C+N_{C/Z}+f^*g^*v|_C-T_{X/Z}|_C\ar[l]
	}.
\]
Let \(A=T_{X/Z}|_C+N_{X/Y}|_C+N_{C/Z}+f^*g^*v|_C-N_{X/Y}|_C-T_{X/Z}|_C\). We have commutative diagrams
\[
	\xymatrix
	{
		f^*T_{Y/Z}|_C+N_{C/Z}+f^*g^*v|_C-f^*T_{Y/Z}|_C\ar[r]\ar[d]	&A\ar[d]\\
		N_{C/Z}+f^*g^*v|_C\ar[r]							&T_{X/Z}|_C+N_{C/Z}+f^*g^*v|_C-T_{X/Z}|_C
	}
\]
\[
	\xymatrix
	{
		N_{C/X}+N_{X/Y}|_C+f^*g^*v|_C-N_{X/Y}|_C-T_{X/Z}|_C\ar[d]\ar[rd]	&\\
		N_{C/X}+f^*g^*v|_C-T_{X/Z}|_C\ar[d]							&A\ar[dl]\\
		T_{X/Z}|_C+N_{C/Z}+f^*g^*v|_C-T_{X/Z}|_C 						&\\
	}.
\]
Furthermore, there is a commutative diagram with exact rows and columns
\[
	\xymatrix
	{
				&0\ar[d]				&0\ar[d]			&					&\\
		0\ar[r]	&T_{X/Z}|_C\ar[r]\ar[d]	&N_{C/X}\ar[d]\ar[r]	&N_{C/Z}\ar[r]\ar[d]_=	&0\\
		0\ar[r]	&f^*T_{Y/Z}|_C\ar[r]\ar[d]	&N_{C/Y}\ar[r]\ar[d]	&N_{C/Z}\ar[r]			&0\\
				&N_{X/Y}|_C\ar[d]\ar[r]^=	&N_{X/Y}|_C\ar[d]	&					&\\
				&0					&0				&					&
	}
\]
And by Theorem \ref{four diagrams}, (2), we have a commutative diagram
\[
	\xymatrix
	{
		N_{C/X}+N_{X/Y}|_C+f^*g^*v|_C-f^*T_{Y/Z}|_C\ar[r]\ar[d]	&N_{C/X}+N_{X/Y}|_C+f^*g^*v|_C-f^*T_{Y/Z}|_C\ar[d]\\
		N_{C/Y}+f^*g^*v|_C-f^*T_{Y/Z}|_C\ar[d]				&T_{X/Z}|_C+N_{C/Z}|_C+N_{X/Y}|_C+f^*g^*v|_C-f^*T_{Y/Z}|_C\ar[d]\\
		f^*T_{Y/Z}|_C+N_{C/Z}+f^*g^*v|_C-f^*T_{Y/Z}|_C\ar[r]		&A\\
	}.
\]
Then we could easily deduce the first diagram we want.
\item
We are going to prove that the following diagram commutes
\[
	\xymatrix
	{
	 	N_{C/X}+N_{X/Y}|_C+f^*g^*v|_C-f^*T_{Y/Z}|_C\ar[r]\ar[d]	&N_{C/X}-f^*T_{Y/Z}|_C+N_{X/Y}|_C+f^*g^*v|_C\ar[d]\\
	 	N_{C/Y}+f^*g^*v|_C-f^*T_{Y/Z}|_C\ar[d]				&N_{C/X}-f^*T_{Y/Z}|_C+f^*T_{Y/Z}|_C+N_{X/Z}|_C+f^*g^*v|_C\ar[d]\\
	 	f^*T_{Y/Z}|_C+N_{C/Z}+f^*g^*v|_C-f^*T_{Y/Z}|_C\ar[d]		&N_{C/X}+N_{X/Z}|_C+f^*g^*v|_C\ar[dl]\\
	 	N_{C/Z}+f^*g^*v|_C 								&
	}.
\]
We have a commutative diagram
\[
	\xymatrix
	{
		f^*T_{Y/Z}|_C+N_{C/Z}+f^*g^*v|_C-f^*T_{Y/Z}|_C\ar[r]\ar[d]	&f^*T_{Y/Z}|_C+N_{C/X}+N_{X/Z}|_C+f^*g^*v|_C-f^*T_{Y/Z}|_C\ar[d]\\
		N_{C/Z}+f^*g^*v|_C\ar[r]							&N_{C/X}+N_{X/Z}|_C+f^*g^*v|_C
	}.
\]
Furthermore, there is a commutative diagram with exact rows and columns
\[
	\xymatrix
	{
				&						&0\ar[d]						&0\ar[d]				&\\
				&						&f^*T_{Y/Z}|_C\ar[d]\ar[r]^{\cong}	&f^*T_{Y/Z}|_C\ar[d]		&\\
		0\ar[r]	&N_{C/X}\ar[r]\ar[d]_{\cong}	&N_{C/Y}\ar[d]\ar[r]				&N_{X/Y}|_C\ar[r]\ar[d]	&0\\
		0\ar[r]	&N_{C/X}\ar[r]				&N_{C/Z}\ar[r]\ar[d]				&N_{X/Z}|_C\ar[r]\ar[d]	&0\\
				&						&0							&0					&
	}.
\]
And by Theorem \ref{four diagrams}, (3), we have a commutative diagram
\[
	\xymatrix
	{
		N_{C/X}+N_{X/Y}|_C+f^*g^*v|_C-f^*T_{Y/Z}|_C\ar[r]\ar[d]	&N_{C/X}+f^*T_{Y/Z}|_C+N_{X/Z}|_C+f^*g^*v|_C-f^*T_{Y/Z}|_C\ar[dd]\\
		N_{C/Y}+f^*g^*v|_C-f^*T_{Y/Z}|_C\ar[d]				&\\
		f^*T_{Y/Z}|_C+N_{C/Z}+f^*g^*v|_C-f^*T_{Y/Z}|_C\ar[r]		&f^*T_{Y/Z}|_C+N_{C/X}+N_{X/Z}|_C+f^*g^*v|_C-f^*T_{Y/Z}|_C\\
	}.
\]
Then we could easily deduce the first diagram we want.
\end{enumerate}
\end{proof}
\begin{proposition}\label{functoriality of push-forwards1}
(See Axiom \ref{CTPF}, (3)) Suppose we have a Cartesian square with all schemes being smooth
\[
	\xymatrix
	{
		X'\ar[r]^v\ar[d]^{g}	&X\ar[d]^f\\
		Y'\ar[r]^{u}		&Y
	},
\]
where \(u\) is smooth and \(f\) is a closed immersion, \(s\in\mathscr{P}_Y\) and \(C\in Z^{n+d_f+d_v}(X')\) is closed in \(Y\). Then the following diagram commutes
\[
	\xymatrix
	{
		\widetilde{CH}_C^{n+d_f+d_v}(X',N_{X'/Y'}+g^*u^*s-g^*T_{Y'/Y})\ar[r]^-{g_*}\ar[d]	&\widetilde{CH}_{g(C)}^{n+d_u}(X,u^*s-T_{Y'/Y})\ar[d]_{u_*}\\
		\widetilde{CH}_C^{n+d_f+d_v}(X',v^*N_{X/Y}+u^*f^*s-T_{X'/X})\ar[d]_{v_*}			&\widetilde{CH}_{u(g(C))}^n(Y,s)\\
		\widetilde{CH}_{v(C)}^{n+d_f}(X,N_{X/Y}+f^*s)\ar[ru]_{f_*}								&
	}.
\]
\end{proposition}
\begin{proof}
We are going to show the following diagram commutes
\[
	\xymatrix
	{
		N_{C/X'}+N_{X'/Y'}|_C+g^*u^*s|_C-g^*T_{Y'/Y}|_C\ar[d]\ar[r]	&N_{C/Y'}+g^*u^*s|_C-g^*T_{Y'/Y}|_C\ar[d]\\
		N_{C/X'}+v^*N_{X/Y}|_C+g^*u^*s|_C-T_{X'/X}|_C\ar[d]			&g^*T_{Y'/Y}|_C+N_{C/Y}+g^*u^*s|_C-g^*T_{Y'/Y}|_C\ar[d]\\
		T_{X'/X}|_C+N_{C/X}+v^*N_{X/Y}|_C+g^*u^*s|_C-T_{X'/X}|_C\ar[d]	&N_{C/Y}+g^*u^*s|_C\\
		N_{C/X}+v^*N_{X/Y}|_C+g^*u^*s|_C\ar[ru]
	}.
\]
We have a commutative diagram with exact rows and columns
\[
	\xymatrix
	{
				&0\ar[d]						&0\ar[d]			&							&\\
				&T_{X'/X}|_C\ar[d]\ar[r]^{\cong}	&g^*T_{Y'/Y}|_C\ar[d]	&							&\\
		0\ar[r]	&N_{C/X'}\ar[r]\ar[d]				&N_{C/Y'}\ar[d]\ar[r]	&N_{X'/Y'}|_C\ar[r]\ar[d]_{\cong}	&0\\
		0\ar[r]	&N_{C/X}\ar[r]\ar[d]				&N_{C/Y}\ar[r]\ar[d]	&v^*N_{X/Y}|_C\ar[r]				&0\\
				&0							&0				&							&
	}.
\]
So we have a commutative diagram by Theorem \ref{four diagrams}, (1)
\[
	\xymatrix
	{
		N_{C/Y'}\ar[r]\ar[d]						&N_{C/X'}+N_{X'/Y'}|_C\ar[d]\\
		g^*T_{Y'/Y}|_C+N_{C/Y}\ar[d]				&T_{X'/X}|_C+N_{C/X}+N_{X'/Y'}|_{C}\\
		g^*T_{Y'/Y}|_C+N_{C/X}+v^*N_{X/Y}|_C\ar[ru]	&
	}.
\]
And the statement follows easily from the data above.
\end{proof}
\begin{proposition}
Suppose we have a Cartesian square with all schemes being smooth
\[
	\xymatrix
	{
		X'\ar[r]^v\ar[d]^{g}	&X\ar[d]^f\\
		Y'\ar[r]^{u}		&Y
	}.
\]
\begin{enumerate}
\item(See Axiom \ref{BCSM}) Suppose \(f\), \(u\) are smooth, \(s\in\mathscr{P}_Y\) and \(C\in Z^{n+d_f}(X)\) is a closed subset of \(Y\). Then the following diagram commutes
\[
	\xymatrix
	{
		\widetilde{CH}_C^{n+d_f}(X,f^*s-T_{X/Y})\ar[r]^-{f_*}\ar[d]^{v^*}		&\widetilde{CH}^n_{f(C)}(Y,s)\ar[d]^{u^*}\\
		\widetilde{CH}_{v^{-1}(C)}^{n+d_f}(X',v^*f^*s-v^*T_{X/Y})\ar[r]^-{g_*}	&\widetilde{CH}_{g(v^{-1}(C))}^n(Y',u^*s)
	}.
\]
\item(See Axiom \ref{BCCI}) Suppose \(f\) is a closed immersion, \(s\in\mathscr{P}_Y\) and \(C\in Z^{n+d_f}(X)\). Suppose \(u\) is smooth. Then the following diagram commutes
\[
	\xymatrix
	{
		\widetilde{CH}_C^{n+d_f}(X,N_{X/Y}+f^*s)\ar[r]^-{f_*}\ar[d]^{v^*}		&\widetilde{CH}^n_{f(C)}(Y,s)\ar[d]^{u^*}\\
		\widetilde{CH}_{v^{-1}(C)}^{n+d_f}(X',v^*N_{X/Y}+v^*f^*s)\ar[r]^-{g_*}	&\widetilde{CH}_{g(v^{-1}(C))}^n(Y',u^*s)
	}.
\]
\end{enumerate}
\end{proposition}
\begin{proof}
\begin{enumerate}
\item We have a commutative diagram by the functoriality of \(v^*\) respect to twists
\[
	\xymatrix
	{
		N_{C/X}+f^*v|_C-T_{X/Y}|_C\ar[r]\ar[d]										&T_{X/Y}|_C+N_{C/Y}+f^*s|_C-T_{X/Y}|_C\ar[d]\\
		N_{v^{-1}(C)/X'}+v^*f^*s|_{v^{-1}(C)}-T_{X'/Y'}|_{v^{-1}(C)}\ar[d]					&N_{C/Y}+f^*s|_C\ar[d]\\
		T_{X'/Y'}|_{v^{-1}(C)}+N_{v^{-1}(C)/Y'}+v^*f^*s|_{v^{-1}(C)}-T_{X'/Y'}|_{v^{-1}(C)}\ar[r]	&N_{v^{-1}(C)/Y'}+f^*s|_{v^{-1}(C)}
	}.
\]
\item We have a commutative diagram by the functoriality of \(v^*\) respect to twists
\[
	\xymatrix
	{
		N_{C/X}+N_{X/Y}|_C+f^*s|_C\ar[r]\ar[d]						&N_{C/Y}+f^*s|_C\ar[d]\\
		N_{v^{-1}(C)/X'}+N_{X'/Y'}|_{v^{-1}(C)}+v^*f^*s|_{v^{-1}(C)}\ar[r]	&N_{v^{-1}(C)/Y'}+v^*f^*s|_{v^{-1}(C)}
	}.
\]
\end{enumerate}
\end{proof}
\begin{proposition}\label{expf}
\begin{enumerate}
\item (See Axiom \ref{PFSM}) Suppose \(f:X\longrightarrow Y\) is a smooth morphism in \(Sm/k\), \(v\in\mathscr{P}_Y\) and \(C\in Z^{n+d_f}(X)\) is a closed subset of \(Y\). Then for any \(Z\in Sm/k\), \(v'\in\mathscr{P}_Z\) and \(D\in Z^m(Z)\), the following diagrams commute
\[
	\xymatrix
	{
		\widetilde{CH}^{n+d_f}_C(X,f^*v-T_{X/Y})\times\widetilde{CH}^m_D(Z,v')\ar[r]^-{\times}\ar[dd]_{f_*\times id}	&\widetilde{CH}^{n+d_f+m}_{C\times D}(X\times Z,(f^*v-T_{X/Y})\times v')\ar[d]_c\\
																									&\widetilde{CH}^{n+d_f+m}_{C\times D}(X\times Z,(f^*v\times v')-T_{X\times Z/Y\times Z})\ar[d]_{(f\times id)_*}\\
		\widetilde{CH}^n_{f(C)}(Y,v)\times\widetilde{CH}^m_D(Z,v')\ar[r]^-{\times}										&\widetilde{CH}^{n+m}_{f(C)\times D}(Y\times Z,v\times v')
	}
\]
\[
	\xymatrix
	{
		\widetilde{CH}^m_D(Z,v')\times\widetilde{CH}^{n+d_f}_C(X,f^*v-T_{X/Y})\ar[r]^-{\times}\ar[d]_{id\times f_*}	&\widetilde{CH}^{n+d_f+m}_{D\times C}(Z\times X,v'\times(f^*v-T_{X/Y}))\ar[d]_{(id\times f)_*}\\
		\widetilde{CH}^m_D(Z,v')\times\widetilde{CH}^n_{f(C)}(Y,v)\ar[r]^-{\times}										&\widetilde{CH}^{n+m}_{D\times f(C)}(Z\times Y,v'\times v)
	}.
\]
\item (See Axiom \ref{PFCI}) Suppose \(f:X\longrightarrow Y\) is a closed immersion in \(Sm/k\), \(v\in\mathscr{P}_Y\) and \(C\) is a closed subset of \(X\). Then for any \(Z\in Sm/k\), \(v'\in\mathscr{P}_Z\) and \(D\in Z^m(Z)\), the following diagrams commute
\[
	\xymatrix
	{
		\widetilde{CH}^{n+d_f}_C(X,N_{X/Y}+f^*v)\times\widetilde{CH}^m_D(Z,v')\ar[r]^-{\times}\ar[d]_{f_*\times id}	&\widetilde{CH}^{n+d_f+m}_{C\times D}(X\times Z,(N_{X/Y}+f^*v)\times v')\ar[d]_{(f\times id)_*}\\
		\widetilde{CH}^n_{f(C)}(Y,v)\times\widetilde{CH}^m_D(Z,v')\ar[r]^-{\times}									&\widetilde{CH}^{n+m}_{f(C)\times D}(Y\times Z,v\times v')
	}
\]
\[
	\xymatrix
	{
		\widetilde{CH}^m_D(Z,v')\times\widetilde{CH}^{n+d_f}_C(X,N_{X/Y}+f^*v)\ar[r]^-{\times}\ar[dd]_{id\times f_*}	&\widetilde{CH}^{n+d_f+m}_{D\times C}(Z\times X,v'\times(N_{X/Y}+f^*v))\ar[d]_c\\
																									&\widetilde{CH}^{n+d_f+m}_{D\times C}(Z\times X,(v'\times f^*v)+N_{X\times Z/Y\times Z})\ar[d]_{(id\times f)_*}\\
		\widetilde{CH}^m_D(Z,v')\times\widetilde{CH}^n_{f(C)}(Y,v)\ar[r]^-{\times}										&\widetilde{CH}^{n+m}_{D\times f(C)}(Z\times Y,v'\times v)
	}.
\]
\end{enumerate}
\end{proposition}
\begin{proof}
We have projections \(p_1:C\times D\longrightarrow C\) and \(p_2:C\times D\longrightarrow D\).
\begin{enumerate}
\item
For the first diagram, we are going to prove the following diagram commutes
\[
	\xymatrix
	{
		(N_{C/X}+f^*v|_C-T_{X/Y}|_C,N_{D/Z}+v'|_D)\ar[r]\ar[d]_{f_*}	&p_1^*(N_{C/X}+f^*v|_C-T_{X/Y}|_C)+p_2^*(N_{D/Z}+v'|_D)\ar[d]\\
		(N_{C/Y}+f^*v|_C,N_{D/Z}+v'|_D)\ar[d]						&N_{C\times D/X\times Z}+p_1^*(f^*v|_C)+p_2^*(v'|_D)-T_{X\times Z/Y\times Z}|_{C\times D}\ar[d]_{(f\times id)_*}\\
		p_1^*(N_{C/Y}+f^*v|_C)+p_2^*(N_{D/Z}+v'|_D)\ar[r]			&N_{C\times D/Y\times Z}+p_1^*(f^*v|_C)+p_2^*(v'|_D)
	}.
\]
We have a commutative diagram
\[
	\xymatrix
	{
		(N_{C/X}+f^*v|_C-T_{X/Y}|_C,N_{D/Z}+v'|_D)\ar[r]\ar[d]_{f_*}	&p_1^*(N_{C/X}+f^*v|_C-T_{X/Y}|_C)+p_2^*(N_{D/Z}+v'|_D)\ar[ddl]^{p_1^*(f_*)+p_2^*(id)}\\
		(N_{C/Y}+f^*v|_C,N_{D/Z}+v'|_D)\ar[d]						&\\
		p_1^*(N_{C/Y}+f^*v|_C)+p_2^*(N_{D/Z}+v'|_D)				&
	}.
\]
Hence we just have to show the following diagram commutes
\[
	\xymatrix
	{
														&p_1^*(N_{C/X}+f^*v|_C-T_{X/Y}|_C)+p_2^*(N_{D/Z}+v'|_D)\ar[d]\ar@/_1pc/[ddl]_{p_1^*(f_*)+p_2^*(id)}\\
														&N_{C\times D/X\times Z}+p_1^*(f^*v|_C)+p_2^*(v'|_D)-T_{X\times Z/Y\times Z}|_{C\times D}\ar[d]_{(f\times id)_*}\\
		p_1^*(N_{C/Y}+f^*v|_C)+p_2^*(N_{D/Z}+v'|_D)\ar[r]			&N_{C\times D/Y\times Z}+p_1^*(f^*v|_C)+p_2^*(v'|_D)
	}.
\]
This follows by Theorem \ref{four diagrams}, (1) from the following commutative diagram with exact rows and columns
\[
	\xymatrix
	{
				&0\ar[d]							&0\ar[d]								&						&\\
				&p_1^*(T_{X/Y}|_C)\ar[d]\ar[r]^-{\cong}	&T_{X\times Z/Y\times Z}|_{C\times D}\ar[d]	&						&\\
		0\ar[r]	&p_1^*N_{C/X}\ar[r]\ar[d]				&N_{C\times D/X\times Z}\ar[d]\ar[r]			&p_2^*N_{D/Z}\ar[r]\ar[d]_=	&0\\
		0\ar[r]	&p_1^*N_{C/Y}\ar[r]\ar[d]				&N_{C\times D/Y\times Z}\ar[r]\ar[d]			&p_2^*N_{D/Z}\ar[r]			&0\\
				&0								&0									&						&
	}.
\]

For the second diagram, we suppose \(\alpha\in\widetilde{CH}^{n+d_f}_C(X,f^*v-T_{X/Y})\) and \(\beta\in\widetilde{CH}^m_D(Z,v')\). And we have a commutative diagram
\[
	\xymatrix
	{
		X\ar[d]_f	&X\times Z\ar[l]_{p_1}\ar[r]^{p_2}\ar[d]_{f\times id}	&D\\
		Y		&Y\times Z\ar[l]_{q_1}\ar[ru]^{q_2}				&
	}.
\]
Then
\begin{align*}
	&(id\times f)_*(\beta\times\alpha)\\
=	&(f\times id)_*(<-1>^{(n+rk_Y(v))(m+rk_Z(v'))}c(p_1^*(f^*v-T_{X/Y}),p_2^*(v'))(\alpha\times\beta))\\
	&\textrm{by Proposition \ref{commutativity of exterior product}}\\
=	&<-1>^{(n+rk_Y(v))(m+rk_Z(v'))}(f\times id)_*(c(p_1^*(f^*v-T_{X/Y}),p_2^*(v'))(\alpha\times\beta))\\
=	&<-1>^{(n+rk_Y(v))(m+rk_Z(v'))}(f\times id)_*((c(p_1^*(f^*v),p_2^*(v'))\circ c(-p_1^*T_{X/Y},p_2^*(v')))(\alpha\times\beta))\\
=	&<-1>^{(n+rk_Y(v))(m+rk_Z(v'))}c(q_1^*(v),q_2^*(v'))((f\times id)_*(c(-p_1^*T_{X/Y},p_2^*(v'))(\alpha\times\beta)))\\
	&\textrm{by functoriality of push-forward with respect to twists}\\
=	&<-1>^{(n+rk_Y(v))(m+rk_Z(v'))}c(q_1^*(v),q_2^*(v'))(f_*(\alpha)\times\beta))\\
=	&\beta\times f_*(\alpha)\\
	&\textrm{by Proposition \ref{commutativity of exterior product}.}
\end{align*}
\item
For the first diagram, we are going to prove the following diagram commutes
\[
	\xymatrix
	{
		(N_{C/X}+N_{X/Y}|_C+f^*v|_C,N_{D/Z}+v')\ar[r]\ar[d]	&p_1^*(N_{C/X}+N_{X/Y}|_C+f^*v|_C)+p_2^*(N_{D/Z}+v')\ar[d]\\
		(N_{C/Y}+f^*v|_C,N_{D/Z}+v')\ar[d]				&N_{C\times D/X\times Z}+N_{X\times Z/Y\times Z}|_{C\times D}+p_1^*(f^*v|_C)+p_2^*(v')\ar[d]\\
		p_1^*(N_{C/Y}+f^*v|_C)+p_2^*(N_{D/Z}+v')\ar[r]		&N_{C\times D/Y\times Z}+p_1^*(f^*v|_C)+p_2^*(v')
	}.
\]
We have a commutative diagram
\[
	\xymatrix
	{
		(N_{C/X}+N_{X/Y}|_C+f^*v|_C,N_{D/Z}+v')\ar[r]\ar[d]	&p_1^*(N_{C/X}+N_{X/Y}|_C+f^*v|_C)+p_2^*(N_{D/Z}+v')\ar[ddl]^{p_1^*(f_*)+p_2^*(id)}\\
		(N_{C/Y}+f^*v|_C,N_{D/Z}+v')\ar[d]				&\\
		p_1^*(N_{C/Y}+f^*v|_C)+p_2^*(N_{D/Z}+v')			&
	}.
\]
Hence we just have to show the following diagram commutes
\[
	\xymatrix
	{
											&p_1^*(N_{C/X}+N_{X/Y}|_C+f^*v|_C)+p_2^*(N_{D/Z}+v')\ar[d]\ar@/_1pc/[ddl]_{p_1^*(f_*)+p_2^*(id)}\\
											&N_{C\times D/X\times Z}+N_{X\times Z/Y\times Z}|_{C\times D}+p_1^*(f^*v|_C)+p_2^*(v')\ar[d]\\
		p_1^*(N_{C/Y}+f^*v|_C)+p_2^*(N_{D/Z}+v')\ar[r]	&N_{C\times D/Y\times Z}+p_1^*(f^*v|_C)+p_2^*(v')
	}.
\]
This follows by Theorem \ref{four diagrams}, (2) from the following commutative diagram with exact rows and columns
\[
	\xymatrix
	{
				&0\ar[d]							&0\ar[d]								&						&\\
		0\ar[r]	&p_1^*N_{C/X}\ar[r]\ar[d]				&N_{C\times D/X\times Z}\ar[d]\ar[r]			&p_2^*N_{D/Z}\ar[r]\ar[d]_=	&0\\
		0\ar[r]	&p_1^*N_{C/Y}\ar[r]\ar[d]				&N_{C\times D/Y\times Z}\ar[r]\ar[d]			&p_2^*N_{D/Z}\ar[r]			&0\\
				&p_1^*(N_{X/Y}|_C)\ar[d]\ar[r]^-{\cong}	&N_{X\times Z/Y\times Z}|_{C\times D}\ar[d]	&						&\\
				&0								&0									&						&
	}.
\]

The second diagram follows by the same method as in the proof of the second diagram of (1).
\end{enumerate}
\end{proof}
\begin{remark}
If we could define pull-back for closed immersions and prove Axiom \ref{FPB}, then Axiom \ref{PFSM} and \ref{PFCI} can be deduced by the proposition above and the method in \cite[Corollary 3.5]{CF}.
\end{remark}
\subsection{Intersection with Divisors}\label{Intersection with Divisors}
Next we would like to discuss a special case of intersection, namely pull-back along a divisor of smooth support. The constructions here come basically from \cite{CF1}, but the treatments of push-forwards are different.
\begin{definition}
Let \(X\in Sm/k\), \(D=\{(U_i,f_i)\}\) be a Cartier divisor on \(X\). Suppose \(C\in Z^n(X)\), \(s\in\widetilde{CH}^n_C(X,v)\) and \(dim(C\cap|D|)<dimC\). Suppose
\[s=\sum_as_a\otimes u_a\otimes v_a\in\oplus_{y_a\in X^{(n)}}K_0^{MW}(k(y_a),\Lambda_{y_a}^*\otimes v)\]
where \(s_a\otimes u_a\otimes v_a\in K_0^{MW}(k(y_a),\Lambda_{y_a}^*\otimes v)\) and \(y_a\in X^{(n)}\).
For every \(x\in X^{(n+1)}\), suppose \(x\in U_i\) for some \(i\), hence \(y_a\in U_i\) also. And since \(y_a\notin|D|\), \(f_i\in O_{X,y_a}^*\). So we have an element \(\overline{f_i}\in k(y_a)\). Define
\[ord_x(D\cdot s)=\sum_{x\in\overline{y_a}}\partial^{y_a}_x(<-1>^{codim(y_a)}[\overline{f_i}]s_a\otimes u_a\otimes f_i\otimes v_a)\in K_0^{MW}(k(x),\Lambda_x^*\otimes\mathscr{L}(-D)\otimes v).\]
And define
\[D\cdot s=\sum_{x\in X^{(n+1)}}ord_x(D\cdot s)\in\oplus_{x\in X^{(n+1)}}K_0^{MW}(k(x),\Lambda_x^*\otimes\mathscr{L}(-D)\otimes v).\]
It's functorial with respect to \(v\) by Remark \ref{naturality of partial maps}.
\end{definition}
\begin{lemma}
The definition of \(ord_x(D\cdot s)\) above is independent of the choice of \(i\) and \(f_i\) and
\[D\cdot s\in\widetilde{CH}^{n+1}_{C\cap|D|}(X,\mathscr{L}(-D)+v).\]
\end{lemma}
\begin{proof}
For any other \(j\) and \(f_j\) with \(x\in U_j\), \(f_j/f_i\in O_{X,x}^*\) holds. And we have
\[\sum_{x\in\overline{y_a}}\partial^{y_a}_x(s_a\otimes u_a\otimes v_a)=0\]
since \(s\in\widetilde{CH}^n_C(X,v)\). So we have
\[\sum_{x\in\overline{y_a}}\partial^{y_a}_x(s_a\otimes u_a\otimes f_i\otimes v_a)=0\]
and
\[\sum_{x\in\overline{y_a}}\partial^{y_a}_x(s_a\otimes u_a\otimes f_j\otimes v_a)=0\]
by Remark \ref{linearity of partial maps}, (2). Moreover,
\begin{align*}
	&[\overline{f_j}]s_a\otimes u_a\otimes f_j\otimes v_a\\
=	&([\overline{f_j/f_i}]+<\overline{f_j/f_i}>[\overline{f_i}])s_a\otimes u_a\otimes f_j\otimes v_a\\
=	&[\overline{f_j/f_i}]s_a\otimes u_a\otimes f_j\otimes v_a+[\overline{f_i}]s_a\otimes u_a\otimes f_i\otimes v_a.
\end{align*}
Hence
\begin{align*}
	&\sum_{x\in\overline{y_a}}\partial^{y_a}_x([\overline{f_j}]s_a\otimes u_a\otimes f_j\otimes v_a)\\
=	&\partial^{y_a}_x(\sum_{x\in\overline{y_a}}[\overline{f_j/f_i}]s_a\otimes u_a\otimes f_j\otimes v_a)+\partial^{y_a}_x(\sum_{x\in\overline{y_a}}[\overline{f_i}]s_a\otimes u_a\otimes f_i\otimes v_a)\\
=	&\partial^{y_a}_x(\sum_{x\in\overline{y_a}}[\overline{f_i}]s_a\otimes u_a\otimes f_i\otimes v_a),
\end{align*}
which shows that \(ord_x(D\cdot s)\) is well-defined.

If \(x\notin|D|\), then \(\overline{f_i}\in O_{\overline{y_a},x}^*\). So
\begin{align*}
	&ord_x(D\cdot s)\\
=	&\sum_{x\in\overline{y_a}}\partial^{y_a}_x(<-1>^{codim(y_a)}[\overline{f_i}]s_a\otimes u_a\otimes f_i\otimes v_a)\\
=	&\sum_{x\in\overline{y_a}}[\overline{f_i}]\partial^{y_a}_x(<-1>^{codim(y_a)}s_a\otimes u_a\otimes f_i\otimes v_a)\\
	&\textrm{by Remark \ref{linearity of partial maps}, (1)}\\
=	&0.
\end{align*}
Hence the support of \(D\cdot s\) is contained in \(C\cap|D|\).

Finally let's prove that \(\partial(D\cdot s)=0\), where for every \(z\), we denote \(\sum_{y,z\in\overline{y}}\partial^y_z\) by \(\partial_z\) and the differential map \(\partial\) is then just \((\partial_z)\). For this, suppose \(u\in X^{(n+2)}\), we prove that
\[\partial_u(D\cdot s):=\sum_{x\in X^{(n+1)},u\in\overline{x}}\partial^x_u(ord_x(D\cdot s))=0.\]
Suppose \(u\in U_i\). Then
\[ord_x(D\cdot s)=\sum_{x\in\overline{y_a}}\partial^{y_a}_x(<-1>^{codim(y_a)}[\overline{f_i}]s_a\otimes u_a\otimes f_i\otimes v_a)\]
by definition. So let \(t=\sum_a<-1>^{codim(y_a)}[\overline{f_i}]s_a\otimes u_a\otimes f_i\otimes v_a\).
\[\sum_{x\in X^{(n+1)},u\in\overline{x}}\partial^x_u(ord_x(D\cdot s))=\sum_{x\in X^{(n+1)},u\in\overline{x}}\partial^x_u(\partial_x(t))=\partial_u(\partial(t))=0.\]
\end{proof}
\begin{definition}\label{pull-back along divisors}
(See Axiom \ref{PB}) Let \(X\in Sm/k\) and \(D\) be a smooth effective Cartier divisor on \(X\). Let \(i:|D|\longrightarrow X\) be the inclusion. So we have \(N_{D/X}\cong i^*\mathscr{L}(D)\). Suppose \(v\in\mathscr{P}_X\), \(C\in Z^n(X)\), \(s\in\widetilde{CH}^n_C(X,v)\) and \(dim(C\cap|D|)<dimC\). We have a push-forward isomorphism
\[i_*:\widetilde{CH}^n_{C\cap|D|}(|D|,i^*\mathscr{L}(D)+i^*\mathscr{L}(-D)+i^*v)\longrightarrow\widetilde{CH}^{n+1}_{C\cap|D|}(X,\mathscr{L}(-D)+v).\]
Denote by \(s(\mathscr{L}(D))\) the isomorphism \(i^*v\longrightarrow i^*\mathscr{L}(D)+i^*\mathscr{L}(-D)+i^*v\), define
\[i^*(s)\in\widetilde{CH}^n_{C\cap|D|}(|D|,i^*v)\]
to be the unique element such that
\[i_*(s(\mathscr{L}(D))(i^*(s)))=D\cdot s.\]
It's functorial with respect to \(v\).
\end{definition}
\begin{proposition}\label{exdi}
Suppose \(a=1,2\). Let \(X_a\in Sm/k\), \(v_a\in\mathscr{P}_{X_a}\), \(C_a\in Z^{n_a}(X_a)\) be smooth, \(\alpha_a\in\widetilde{CH}^{n_a}_{C_a}(X_a,v_a)\), \(p_a:X_1\times X_2\longrightarrow X_a\) be projections and \(D_a\) be smooth effective Cartier divisors on \(X_a\). Then
\[(D_1\cdot\alpha_1)\times\alpha_2=p_1^*(D_1)\cdot(\alpha_1\times\alpha_2)\]
and
\[c(p_1^*v_1,p_2^*\mathscr{L}(-D_2))(\alpha_1\times(D_2\cdot\alpha_2))=p_2^*(D_2)\cdot(\alpha_1\times\alpha_2).\]
\end{proposition}
\begin{proof}
Let's prove the first equation. Since both sides live in
\[\widetilde{CH}^{n_1+n_2+1}_{p_1^{-1}(|D_1|\cap C_1)\cap p_2^{-1}(C_2)}(X_1\times X_2,\mathscr{L}(-D_1)+(v_1\times v_2)),\]
it suffices to check their components at any generic point \(u\) in \(\overline{t_1}\times\overline{t_2}\) where \(t_1\in (|D_1|\cap C_1)^{(0)}\), \(t_2\in C_2^{(0)}\). Suppose \(D_1=\{(U_i,f_i)\}\), \(t_1\in U_i\). Then at \(u\), we have
\begin{align*}
	&(D_1\cdot\alpha_1)\times\alpha_2\\
=	&\partial_{t_1}(<-1>^{n_1}[\overline{f_i}]\otimes f_i\otimes\alpha_1)\times\alpha_2\\
=	&\partial(<-1>^{n_1}[\overline{f_i}]\otimes f_i\otimes\alpha_1)\times\alpha_2\\
=	&\partial(<-1>^{n_1}([\overline{f_i}]\otimes f_i\otimes\alpha_1)\times\alpha_2)\\
	&\textrm{by Proposition \ref{exterior products and differential maps}.}\\
=	&\partial_u(<-1>^{n_1}([\overline{f_i}]\otimes f_i\otimes\alpha_1)\times\alpha_2)\\
=	&\partial_u(<-1>^{n_1}([\overline{p_1^*(f_i)}]\otimes p_1^*(f_i)\otimes(\alpha_1\times\alpha_2))\\
=	&p_1^*(D_1)\cdot(\alpha_1\times\alpha_2).
\end{align*}
For the second equation, we exchange the role of \(X_1\) and \(X_2\) as before:
\begin{align*}
	&c(p_1^*v_1,p_2^*\mathscr{L}(-D_2))(\alpha_1\times(D_2\cdot\alpha_2))\\
=	&<-1>^{(n_1+rk_{X_1}(v_1))(n_2+rk_{X_2}(v_2))}c(p_2^*v_2,p_1^*v_1)((D_2\cdot\alpha_2)\times\alpha_1)\\
	&\textrm{by Proposition \ref{commutativity of exterior product}}\\
=	&<-1>^{(n_1+rk_{X_1}(v_1))(n_2+rk_{X_2}(v_2))}c(p_2^*v_2,p_1^*v_1)(p_2^*(D_2)\cdot(\alpha_2\times\alpha_1))\\
	&\textrm{by the first equation}\\
=	&<-1>^{(n_1+rk_{X_1}(v_1))(n_2+rk_{X_2}(v_2))}p_2^*(D_2)\cdot c(p_2^*v_2,p_1^*v_1)(\alpha_2\times\alpha_1)\\
	&\textrm{by the functoriality of intersection with respect to twists}\\
=	&p_2^*(D_2)\cdot(\alpha_1\times\alpha_2)\\
	&\textrm{by Proposition \ref{commutativity of exterior product}.}
\end{align*}
\end{proof}
\begin{proposition}\label{pfdi}
\begin{enumerate}
\item (See Axiom \ref{PFSM}) Let \(f:X\longrightarrow Y\) be a smooth morphism in \(Sm/k\), \(C\in Z^{i+d_f}(X)\) be smooth and closed in \(Y\), \(D\) be a Cartier divisor over \(Y\), \(dim(|D|\cap f(C))<dim(f(C))\) and \(\alpha\in\widetilde{CH}^{i+d_f}_C(X,f^*v-T_{X/Y})\). Then
\[D\cdot f_*(\alpha)=f_*(f^*(D)\cdot\alpha).\]
\item (See Axiom \ref{PFCI}) Let \(f:X\longrightarrow Y\) be a closed immersion in \(Sm/k\), \(C\in Z^{i+d_f}(X)\) be smooth, \(D\) be a Cartier divisor over \(Y\), \(dim(|D|\cap f(C))<dim(f(C))\) and \(\alpha\in\widetilde{CH}^{i+d_f}_C(X,N_{X/Y}+f^*v)\). Then
\[D\cdot f_*(\alpha)=f_*(c(\mathscr{L}(-f^*D),N_{X/Y})(f^*(D)\cdot\alpha)).\]
\end{enumerate}
\end{proposition}
\begin{proof}
\begin{enumerate}
\item Both sides live in the same Chow-Witt group, so we check their components at any generic point \(y\) of \(f(C)\cap|D|\). Suppose \(D=\{(U_i,f_i)\}\), \(y\in U_i\).
We have a commutative diagram
\[
	\xymatrix
	{
		(\mathscr{L}(-D)|_C,N_{C/X}+f^*v|_C-T_{X/Y}|_C)\ar[r]^-{(id,f_*)}\ar[d]	&(\mathscr{L}(-D)|_C,N_{C/Y}+f^*v|_C)\ar[d]\\
		\mathscr{L}(-D)|_C+N_{C/X}+f^*v|_C-T_{X/Y}|_C\ar[r]^-{id+f_*}\ar[d]		&\mathscr{L}(-D)|_C+N_{C/Y}+f^*v|_C\ar[d]\\
		N_{C/X}+\mathscr{L}(-D)|_C+f^*v|_C-T_{X/Y}|_C\ar[r]^-{f_*}				&N_{C/Y}+\mathscr{L}(-D)|_C+f^*v|_C
	}.
\]
Then at \(y\), we have
\begin{align*}
	&D\cdot f_*(\alpha)\\
=	&\partial_y(<-1>^i[\overline{f_i}]\otimes f_i\otimes f_*(\alpha))\\
=	&\partial_y(<-1>^{i+d_f}f_*([\overline{f^*(f_i)}]\otimes f^*(f_i)\otimes\alpha))\\
	&\textrm{by the diagram above}\\
=	&f_*\partial_y(<-1>^{i+d_f}[\overline{f^*(f_i)}]\otimes f^*(f_i)\otimes\alpha)\\
	&\textrm{by Proposition \ref{partial maps and push-forwards for smooth morphisms}}\\
=	&f_*(f^*(D)\cdot\alpha).
\end{align*}
\item Both sides live in the same Chow-Witt group, so we check their components at any generic point \(y\) of \(f(C)\cap|D|\). Suppose \(D=\{(U_i,f_i)\}\), \(y\in U_i\).
We have a commutative diagram
\[
	\xymatrix
	{
		(\mathscr{L}(-D)|_C,N_{C/X}+N_{X/Y}|_C+f^*v|_C)\ar[r]^-{(id,f_*)}\ar[d]	&(\mathscr{L}(-D)|_C,N_{C/Y}+f^*v|_C)\ar[d]\\
		\mathscr{L}(-D)|_C+N_{C/X}+N_{X/Y}|_C+f^*v|_C\ar[r]^-{id+f_*}\ar[d]		&\mathscr{L}(-D)|_C+N_{C/Y}+f^*v|_C\ar[d]\\
		N_{C/X}+N_{X/Y}|_C+\mathscr{L}(-D)|_C+f^*v|_C\ar[r]^-{f_*}				&N_{C/Y}+\mathscr{L}(-D)|_C+f^*v|_C
	}.
\]
Then at \(y\), we have
\begin{align*}
	&D\cdot f_*(\alpha)\\
=	&\partial_y(<-1>^i[\overline{f_i}]\otimes f_i\otimes f_*(\alpha))\\
=	&\partial_y(<-1>^{i+d_f}f_*([\overline{f^*(f_i)}]\otimes f^*(f_i)\otimes\alpha))\\
	&\textrm{by the diagram above}\\
=	&f_*\partial_y(<-1>^{i+d_f}[\overline{f^*(f_i)}]\otimes f^*(f_i)\otimes\alpha)\\
	&\textrm{by Proposition \ref{partial maps and push-forwards for closed immersions}}\\
=	&f_*(c(\mathscr{L}(-f^*D),N_{X/Y})(f^*(D)\cdot\alpha)).
\end{align*}
\end{enumerate}
\end{proof}

Now we are ready for basic formulas concerning pull-back along divisors. We will use the notation in Definition \ref{pull-back along divisors}.
\begin{proposition}
(See Axiom \ref{CPB}) Suppose \(a=1,2\). Let \(X_a\in Sm/k\), \(D_a\) be effective smooth divisors over \(X_a\), \(v_a\in\mathscr{P}_{X_a}\), \(C_a\in Z^{n_a}(X_a)\) be smooth, \(dim(C_a\cap|D_a|)<dim(C_a)\), \(\alpha_a\in\widetilde{CH}^{n_a}_{C_a}(X_a,v_a)\) and \(i_a:|D_a|\longrightarrow X_a\) be inclusions. Then we have
\[i_1^*(\alpha_1)\times\alpha_2=(i_1\times id)^*(\alpha_1\times\alpha_2)\]
\[\alpha_1\times i_2^*(\alpha_2)=(id\times i_2)^*(\alpha_1\times\alpha_2).\]
\end{proposition}
\begin{proof}
We denote the projection \(X_1\times X_2\longrightarrow X_a\) by \(p_a\). For the first equation, it suffices to check the equation after applying the isomorphism \((i_1\times id)_*\circ s(\mathscr{L}(p_1^*D_1))\) on both sides. We have
\begin{align*}
	&(i_1\times id)_*(s(\mathscr{L}(p_1^*D_1))(i_1^*(\alpha_1)\times\alpha_2))\\
=	&(i_1\times id)_*((s(\mathscr{L}(D_1))i_1^*(\alpha_1))\times\alpha_2)\\
	&\textrm{by bifunctoriality of exterior product with respect to twists}\\
=	&i_{1*}(s(\mathscr{L}(D_1))i_1^*(\alpha_1))\times\alpha_2\\
	&\textrm{by Proposition \ref{expf}}\\
=	&(D_1\cdot\alpha_1)\times\alpha_2\\
=	&p_1^*(D_1)\cdot(\alpha_1\times\alpha_2)\\
	&\textrm{by Proposition \ref{exdi}}\\
=	&(i_1\times id)_*(s(\mathscr{L}(p_1^*D_1))((i_1\times id)^*(\alpha_1\times\alpha_2))).
\end{align*}
The second equation follows by exchanging the roles of \(X_1\) and \(X_2\):
\begin{align*}
	&\alpha_1\times i_2^*(\alpha_2)\\
=	&<-1>^{(n_1+rk_{X_1}(v_1))(n_2+rk_{X_2}(v_2))}c(q_2^*i_2^*v_2,q_1^*v_1)(i_2^*(\alpha_2)\times\alpha_1)\\
=	&<-1>^{(n_1+rk_{X_1}(v_1))(n_2+rk_{X_2}(v_2))}c(q_2^*i_2^*v_2,q_1^*v_1)((i_2\times id)^*(\alpha_2\times\alpha_1))\\
=	&<-1>^{(n_1+rk_{X_1}(v_1))(n_2+rk_{X_2}(v_2))}(i_2\times id)^*(c(p_2^*v_2,p_2^*v_1)(\alpha_2\times\alpha_1))\\
	&\textrm{by functoriality of pull-back with respect to twists}\\
=	&(id\times i_2)^*(\alpha_1\times\alpha_2).
\end{align*}
\end{proof}
\begin{proposition}
Suppose we have a Cartesian square with all schemes being smooth
\[
	\xymatrix
	{
		X'\ar[r]^v\ar[d]^{g}	&X\ar[d]^f\\
		Y'\ar[r]^{u}		&Y
	},
\]
where \(u\) is a closed immersion, \(dim(X')=dim(X)-1\) and \(dim(Y')=dim(Y)-1\).
\begin{enumerate}
\item (See Axiom \ref{BCCI}) If \(f\) is a closed immersion , \(s\in\mathscr{P}_Y\), \(C\in Z^{n+d_f}(X)\) is smooth and \(dim(u^{-1}(f(C)))<dim(f(C))\), the following diagram commutes
\[
	\xymatrix
	{
		\widetilde{CH}_C^{n+d_f}(X,N_{X/Y}+f^*s)\ar[r]^-{f_*}\ar[d]^{v^*}		&\widetilde{CH}^n_{f(C)}(Y,s)\ar[d]^{u^*}\\
		\widetilde{CH}_{v^{-1}(C)}^{n+d_f}(X',v^*N_{X/Y}+v^*f^*s)\ar[r]^-{g_*}	&\widetilde{CH}_{g(v^{-1}(C))}^n(Y',u^*s)
	}.
\]
\item (See Axiom \ref{BCSM}) If \(f\) is smooth, \(s\in\mathscr{P}_Y\) and \(C\in Z^{n+d_f}(X)\) is smooth and closed in \(Y\), the following diagram commutes
\[
	\xymatrix
	{
		\widetilde{CH}_C^{n+d_f}(X,f^*s-T_{X/Y})\ar[r]^-{f_*}\ar[d]^{v^*}		&\widetilde{CH}^n_{f(C)}(Y,s)\ar[d]^{u^*}\\
		\widetilde{CH}_{v^{-1}(C)}^{n+d_f}(X',v^*f^*s-v^*T_{X/Y})\ar[r]^-{g_*}	&\widetilde{CH}_{g(v^{-1}(C))}^n(Y',u^*s)
	}.
\]
\end{enumerate}
\end{proposition}
\begin{proof}
The conditions give us a unique effective smooth divisor \(D\) (resp. \(D'\)) over \(Y\) (resp. \(X\)) such that \(|D|=Y'\) (resp. \(|D'|=X'\)). And we have \(D'=f^*(D)\). It suffices to check the equation after applying \(u_*\circ s(\mathscr{L}(D))\) on both sides.
\begin{enumerate}
\item Suppose \(\alpha\in\widetilde{CH}_C^{n+d_f}(X,N_{X/Y}+f^*s)\), we have
\begin{align*}
	&u_*(s(\mathscr{L}(D))(u^*f_*(\alpha)))\\
=	&D\cdot f_*(\alpha)\\
=	&f_*(c(\mathscr{L}(-D'),N_{X/Y})(D'\cdot\alpha))\\
	&\textrm{by Proposition \ref{pfdi}, (2)}\\
=	&f_*(c(\mathscr{L}(-D'),N_{X/Y})(v_*(s(\mathscr{L}(D'))(v^*(\alpha)))))\\
=	&f_*v_*((c(v^*\mathscr{L}(-D'),v^*N_{X/Y})\circ s(\mathscr{L}(D')))(v^*(\alpha)))\\
=	&u_*g_*((c(v^*\mathscr{L}(D')+v^*\mathscr{L}(-D'),v^*N_{X/Y})\circ s(\mathscr{L}(D')))(v^*(\alpha)))\\
	&\textrm{by Proposition \ref{functoriality of push-forwards}}\\
=	&u_*(s(\mathscr{L}(D))(g_*(v^*(\alpha))))\\
	&\textrm{by functoriality of push-forwards with respect to twists.}
\end{align*}
\item Suppose \(\alpha\in\widetilde{CH}_C^{n+d_f}(X,f^*s-T_{X/Y})\), we have
\begin{align*}
	&u_*(s(\mathscr{L}(D))(u^*f_*(\alpha)))\\
=	&D\cdot f_*(\alpha)\\
=	&f_*(D'\cdot\alpha)\\
	&\textrm{by Proposition \ref{pfdi}, (1)}\\
=	&f_*(v_*((s(\mathscr{L}(D'))(v^*(\alpha)),v^*\mathscr{L}(D')-v^*\mathscr{L}(D')+v^*f^*s-v^*T_{X/Y})))\\
=	&u_*(g_*((g^*s(\mathscr{L}(D))(v^*(\alpha)),g^*u^*\mathscr{L}(D)-g^*u^*\mathscr{L}(D)+v^*f^*s-T_{X'/Y'})))\\
	&\textrm{by Proposition \ref{functoriality of push-forwards1}}\\
=	&u_*(s(\mathscr{L}(D))(g_*v^*(\alpha)))\\
	&\textrm{by functoriality of push-forwards with respect to twists.}
\end{align*}
\end{enumerate}
\end{proof}
\section{Sheaves with \(E\)-Tranfers and Their Operations}\label{Sheaves}
In this section, we develop the theory of sheaves with \(E\)-transfers over a smooth base as in \cite{D} and \cite{CF}, where \(E\) is a correspondence theory.

Since there will be heavy calculation on twists, from now on, for convenience and clarity, we will use notations like \((\alpha,v)\) for \(\alpha\in E^i_C(X,v)\). With this in mind, we have operations like \((\alpha,v)\cdot(\beta,u)\), \(f^*((\alpha,v))\) with obvious meaning.

Let \(S\in Sm/k\) and denote the category of smooth schemes over \(S\) by \(Sm/S\). We need the notion of admissible subset coming from \cite[Definition 4.1]{CF}.
\begin{definition}
Let \(X, Y\in Sm/S\), we denote by \(\mathscr{A}_S(X,Y)\) the closed subsets \(T\) of \(X\times_SY\) whose irreducible component are all finite over \(X\) and of dimension \(\textrm{dim}X\). They are called admissible subsets from \(X\) to \(Y\) over \(S\).
\end{definition}
\begin{lemma}
In the definition above, \(T\) itself is also finite over \(X\).
\end{lemma}
\begin{proof}
For every affine open subset \(U\) of \(X\), \(T\cap U\) is affine since each of its components are affine (see \cite[Chapter III, Exercises 3.2]{H}). Its structure ring is a submodule of a finite \(O_X(U)\)-module. Hence we conclude that \(T\cap U\) is finite over \(U\).
\end{proof}
\begin{definition}\label{cor}
Let \(S\in Sm/k\), \(X, Y\in Sm/S\), we define
\[\widetilde{Cor}_S(X,Y)=\varinjlim_TE_T^{d_Y-d_S}(X\times_SY,-T_{X\times_SY/X})\]
to be the group of finite \(E\)-correspondences between \(X\) and \(Y\) over \(S\), where \(T\in\mathscr{A}_S(X,Y)\).
\end{definition}

We define a category \(\widetilde{Cor}_S(X,Y)\) whose objects are smooth schemes over \(S\) and morphisms between \(X\) and \(Y\) are just \(\widetilde{Cor}_S(X,Y)\) defined above. Let's now study the compositions in that category. Let's denote, for example, \(X\times_SY\times_SZ\) by \(XYZ\) and the projection \(X\times_SY\times_SZ\longrightarrow Y\times_SZ\) by \(p^{XYZ}_{YZ}\) if no confusion arises.

Given any \(\alpha\in\widetilde{Cor}_S(X,Y)\) and \(\beta\in\widetilde{Cor}_S(Y,Z)\), we may suppose that they come from some groups with admissible supports. Then the image of
\[p^{XYZ}_{XZ*}(p^{XYZ*}_{YZ}((\beta,-T_{YZ/Y}))\cdot p^{XYZ*}_{XY}((\alpha,-T_{XY/X})))\]
in \(\widetilde{Cor}_S(X,Z)\) is just defined as \(\beta\circ\alpha\). This definition is obviously compatible with extension of supports so it's well-defined.
\begin{proposition}\label{associativity}
The composition defined above is associative.
\end{proposition}
\begin{proof}
Suppose \(\xymatrix{X\ar[r]^{\alpha}&Y\ar[r]^{\beta}&Z\ar[r]^{\gamma}&W}\) are morphisms in \(\widetilde{Cor}_S\). As above, we may suppose that they come from groups with admissible supports.

We have Cartesian squares
\[
	\xymatrix
	{
		XYZW\ar[r]\ar[d]	&XZW\ar[d]\\
		XYZ\ar[r]			&XZ
	}
	\xymatrix
	{
		XYZW\ar[r]\ar[d]	&XYW\ar[d]\\
		YZW\ar[r]			&YW
	}
\]
So
\begin{align*}
	&\gamma\circ(\beta\circ\alpha)\\
=	&p^{XZW}_{XW*}(p^{XZW*}_{ZW}((\gamma,-T_{ZW/Z}))p^{XZW*}_{XZ}p^{XYZ}_{XZ*}(p^{XYZ*}_{YZ}((\beta,-T_{YZ/Y}))p^{XYZ*}_{XY}((\alpha,-T_{XY/X}))))\\
	&\textrm{by definition}\\
=	&p^{XZW}_{XW*}(p^{XZW*}_{ZW}(\gamma)p^{XZW*}_{XZ}p^{XYZ}_{XZ*}((p^{XYZ*}_{YZ}(\beta)p^{XYZ*}_{XY}(\alpha),-T_{XYZ/XY}-T_{XYZ/XZ})))\\
	&\textrm{by definition of the product}\\
=	&p^{XZW}_{XW*}(p^{XZW*}_{ZW}(\gamma)p^{XYZW}_{XZW*}p^{XYZW*}_{XYZ}((p^{XYZ*}_{YZ}(\beta)p^{XYZ*}_{XY}(\alpha),-T_{XYZ/XY}-T_{XYZ/XZ})))\\
	&\textrm{by Axiom \ref{BCSM} for the left square above}\\
=	&p^{XZW}_{XW*}(p^{XZW*}_{ZW}(\gamma)p^{XYZW}_{XZW*}(p^{XYZW*}_{YZ}(\beta)p^{XYZW*}_{XY}(\alpha),-T_{XYZW/XYW}-T_{XYZW/XZW}))\\
	&\textrm{by Axiom \ref{FPB} and Axiom \ref{CPB}}\\
=	&p^{XZW}_{XW*}p^{XYZW}_{XZW*}((p^{XYZW*}_{ZW}(\gamma),-T_{XYZW/XYZ})p^{XYZW*}_{YZ}(\beta)p^{XYZW*}_{XY}(\alpha))\\
	&\textrm{by Axiom \ref{PFSM} for \(p^{XYZW}_{XZW}\)}\\
=	&p^{XZW}_{XW*}p^{XYZW}_{XZW*}((\delta,-p^{XYZW*}_{XW}T_{XW/X}-p^{XYZW*}_{XZW}T_{XZW/XW}-T_{XYZW/XZW}))\\
	&\textrm{by definition of the product where \(\delta=p^{XYZW*}_{ZW}(\gamma)p^{XYZW*}_{YZ}(\beta)p^{XYZW*}_{XY}(\alpha)\)}\\
=	&p^{XYZW}_{XW*}((\delta,-p^{XYZW*}_{XW}T_{XW/X}-T_{XYZW/XW}))\\
	&\textrm{by Axiom \ref{FPFSM}}
\end{align*}
\begin{align*}
=	&p^{XYW}_{XW*}p^{XYZW}_{XYW*}((\delta,-p^{XYZW*}_{XW}T_{XW/X}-p^{XYZW*}_{XYW}T_{XYW/XW}-T_{XYZW/XYW}))\\
	&\textrm{by Axiom \ref{FPFSM}, note that we have used \(c(-T_{XYZW/XYW},-T_{XYZW/XZW})\)}\\
=	&p^{XYW}_{XW*}(p^{XYZW}_{XYW*}(p^{XYZW*}_{ZW}(\gamma)p^{XYZW*}_{YZ}(\beta))p^{XYW*}_{XY}(\alpha))\\
	&\textrm{by Axiom \ref{PFSM} for \(p^{XYZW}_{XYW}\)}\\
=	&p^{XYW}_{XW*}(p^{XYZW}_{XYW*}p^{XYZW*}_{YZW}(p^{YZW*}_{ZW}(\gamma)(p^{YZW*}_{YZ}(\beta),-T_{YZW/YW}))p^{XYW*}_{XY}(\alpha))\\
	&\textrm{by Axiom \ref{FPB} and Axiom \ref{CPB}}\\
=	&p^{XYW}_{XW*}(p^{XYW*}_{YW}p^{YZW}_{YW*}(p^{YZW*}_{ZW}(\gamma)p^{YZW*}_{YZ}(\beta))p^{XYW*}_{XY}(\alpha))\\
	&\textrm{by Axiom \ref{BCSM} for the right square above}\\
=	&(\gamma\circ\beta)\circ\alpha\\
	&\textrm{by definition.}\\
\end{align*}
\end{proof}
\begin{definition}\label{graph}
Define a functor
\[\widetilde{\gamma}:Sm/S\longrightarrow\widetilde{Cor}_S,\]
where \(\widetilde{\gamma}(X)=X\). Given an \(S\)-morphism \(f:X\longrightarrow Y\), we have the graph morphism \(\Gamma_f:X\longrightarrow X\times_SY\) and the natural map
\[\Gamma_f^*T_{X\times_SY/X}\longrightarrow N_{X/X\times_SY}\]
in Lemma \ref{t} is an isomorphism. So we have maps
\[\xymatrix{E^0_X(X,0)\ar[r]&E^0_X(X,N_{X/X\times_SY}-\Gamma_f^*T_{X\times_SY/X})\ar[r]^-{\Gamma_{f*}}&E^{d_Y-d_S}_X(X\times_SY,-T_{X\times_SY/X})\ar[d]\\&&\widetilde{Cor}_S(X,Y)}\]
and denote the image of \(1\) under the composition of these maps by \(\widetilde{\gamma}(f)\).
\end{definition}
The following propositions treat some easy cases of composition.
\begin{proposition}\label{easy composition}
Let \(f:X\longrightarrow Y\) be a morphism in \(Sm/S\) and \(g:Y\longrightarrow Z\) be a morphism in \(\widetilde{Cor}_S\). Then we have
\[g\circ\widetilde{\gamma}(f)=(f\times id_Z)^*(g)\]
where the right hand side means its image into the direct limit.
\end{proposition}
\begin{proof}
We have a Cartesian square
\[
	\xymatrix
	{
		X\ar[r]^{\Gamma_f}						&XY\\
		XZ\ar[u]^{p^{XZ}_X}\ar[r]^{\Gamma_f\times id_Z}	&XYZ\ar[u]_{p^{XYZ}_{XY}}
	}
\]
and denote the map \(E^0_X(X,0)\longrightarrow E^0_X(X,N_{X/X\times_SY}-\Gamma_f^*T_{X\times_SY/X})\) by \(t\). Suppose \(g\) comes from some cohomology with support. We have
\begin{align*}
	&g\circ\widetilde{\gamma}(f)\\
=	&p^{XYZ}_{XZ*}(p^{XYZ*}_{YZ}((g,-T_{YZ/Y}))\cdot p^{XYZ*}_{XY}\Gamma_{f*}((t(1),N_{X/XY}-\Gamma_f^*T_{XY/X})))\\
	&\textrm{by definition}\\
=	&p^{XYZ}_{XZ*}(p^{XYZ*}_{YZ}((g,-T_{YZ/Y}))\cdot(\Gamma_f\times id_Z)_*p^{XZ*}_X(t(1)))\\
	&\textrm{by Axiom \ref{BCCI} for the square above}\\
=	&p^{XYZ}_{XZ*}((\Gamma_f\times id_Z)_*((\Gamma_f\times id_Z)^*p^{XYZ*}_{YZ}((g,-T_{YZ/Y}))\cdot p^{XZ*}_X(t(1))))\\
	&\textrm{by Axiom \ref{PFCI} for \(\Gamma_f\times id_Z\)}\\
=	&p^{XYZ}_{XZ*}((\Gamma_f\times id_Z)_*((f\times id_Z)^*((g,-T_{YZ/Y}))\cdot p^{XZ*}_X(t(1))))\\
	&\textrm{by Axiom \ref{FPB}}\\
=	&p^{XYZ}_{XZ*}(\Gamma_f\times id_Z)_*((f\times id_Z)^*(g)\cdot p^{XZ*}_X(t(1)),-T_{XZ/X}+N_{XZ/XYZ}-(\Gamma_f\times id_Z)^*T_{XYZ/XZ})\\
	&\textrm{by definition of the product and pull-back and Lemma \ref{t2}}\\
=	&s(((f\times id_Z)^*(g)\cdot p^{XZ*}_X(t(1)),-T_{XZ/X}+N_{XZ/XYZ}-(\Gamma_f\times id_Z)^*T_{XYZ/XZ}))\\
	&\textrm{by Axiom \ref{CTPF} and \(s\) is the isomorphism cancelling \(N_{XZ/XYZ}\cong(\Gamma_f\times id_Z)^*T_{XYZ/XZ}\)}\\
=	&(f\times id_Z)^*(g)\cdot s(p^{XZ*}_X(t(1)))\\
	&\textrm{by bifunctoriality of product with respect to twists}\\
=	&(f\times id_Z)^*(g)\cdot p^{XZ*}_X(1)\\
	&\textrm{by functoriality of pull-back with respect to twists}\\
=	&(f\times id_Z)^*(g)\\
	&\textrm{by the definition of identity and Axiom \ref{FPB}}
\end{align*}
\end{proof}
\begin{proposition}\label{easy composition1}
Let \(f:X\longrightarrow Y\) be a morphism in \(\widetilde{Cor}_S\) and \(g:Y\longrightarrow Z\) be a smooth morphism in \(Sm/S\). Let \(t\) be the composition
\begin{align*}
			&-T_{XY/X}\\
\longrightarrow	&-(id_X\times\Gamma_g)^*T_{XYZ/XY}+N_{XY/XYZ}-T_{XY/X}\\
\longrightarrow	&-(id_X\times\Gamma_g)^*T_{XYZ/XY}+N_{XY/XYZ}-(id_X\times\Gamma_g)^*T_{XYZ/XZ}\\
\longrightarrow	&-(id_X\times\Gamma_g)^*T_{XYZ/XY}+N_{XY/XYZ}-N_{XY/XYZ}-T_{XY/XZ}\\
\longrightarrow	&-(id_X\times\Gamma_g)^*T_{XYZ/XY}-T_{XY/XZ}\\
\longrightarrow	&-(id_X\times g)^*T_{XZ/X}-T_{XY/XZ}.
\end{align*}
Then we have
\[\widetilde{\gamma}(g)\circ f=(id_X\times g)_*(t(f)),\]
where the right-hand side means the image into the direct limit.
\end{proposition}
\begin{proof}
We have a Cartesian square
\[
	\xymatrix
	{
		Y\ar[r]^{\Gamma_g}						&YZ\\
		XY\ar[u]^{p^{XY}_Y}\ar[r]^{id_X\times\Gamma_g}	&XYZ,\ar[u]_{p^{XYZ}_{YZ}}
	}
\]
an isomorphism \(s:0\longrightarrow N_{Y/YZ}-\Gamma_g^*T_{YZ/Y}\) and an isomorphism \(r:-T_{XY/X}\longrightarrow N_{XY/XYZ}-(id_X\times\Gamma_g)^*T_{XYZ/XY}-T_{XY/X}\). Suppose \(f\) comes from some group with support. So
\begin{align*}
	&\widetilde{\gamma}(g)\circ f\\
=	&p^{XYZ}_{XZ*}(p^{XYZ*}_{YZ}\Gamma_{g*}((s(1),N_{Y/YZ}-\Gamma_g^*T_{YZ/Y}))\cdot p^{XYZ*}_{XY}((f,-T_{XY/X}))\\
	&\textrm{by definition}\\
=	&p^{XYZ}_{XZ*}((id_X\times\Gamma_g)_*p^{XY*}_Y((s(1),N_{Y/YZ}-\Gamma_g^*T_{YZ/Y}))\cdot p^{XYZ*}_{XY}(f))\\
	&\textrm{by Axiom \ref{BCCI} for the square above}\\
=	&p^{XYZ}_{XZ*}(id_X\times\Gamma_g)_*(p^{XY*}_Y((s(1),N_{Y/YZ}-\Gamma_g^*T_{YZ/Y}))\cdot(id_X\times\Gamma_g)^*p^{XYZ*}_{XY}(f))\\
	&\textrm{by Axiom \ref{PFCI} for \(id_X\times\Gamma_g\)}\\
=	&p^{XYZ}_{XZ*}(id_X\times\Gamma_g)_*(r((id_X\times\Gamma_g)^*p^{XYZ*}_{XY}(f)))\\
	&\textrm{by functoriality of pull-back and product with respect to twists}\\
=	&(id_X\times g)_*(t((id_X\times\Gamma_g)^*p^{XYZ*}_{XY}(f)))\\
	&\textrm{by Axiom \ref{CTPF}}\\
=	&(id_X\times g)_*(t(f))\\
	&\textrm{by Axiom \ref{FPB}}.
\end{align*}
\end{proof}
\begin{proposition}\label{easy composition2}
Let \(f:X\longrightarrow Y\) be a morphism in \(\widetilde{Cor}_S\) and \(g:Y\longrightarrow Z\) be a closed immersion in \(Sm/S\). Let \(t'\) be the composition
\begin{align*}
			&-T_{XY/X}\\
\longrightarrow	&-T_{XY/X}+N_{XY/XYZ}-(id_X\times\Gamma_g)^*T_{XYZ/XY}\\
\longrightarrow	&-T_{XY/X}+(id_X\times\Gamma_g)^*T_{XYZ/XZ}+N_{XY/XZ}-(id_X\times\Gamma_g)^*T_{XYZ/XY}\\
\longrightarrow	&-T_{XY/X}+T_{XY/X}+N_{XY/XZ}-(id_X\times\Gamma_g)^*T_{XYZ/XY}\\
\longrightarrow	&N_{XY/XZ}-(id_X\times\Gamma_g)^*T_{XYZ/XY}\\
\longrightarrow	&N_{XY/XZ}-(id_X\times g)^*T_{XZ/X}.
\end{align*}
Then we have
\[\widetilde{\gamma}(g)\circ f=(id_X\times g)_*(t'(f)),\]
where the right-hand side means the image into the direct limit.
\end{proposition}
\begin{proof}
The same as the above proposition.
\end{proof}

Now we would like to simplify the isomorphisms \(t\) and \(t'\) above. This will involve more complicated calculations in the category of virtual vector bundles.
\begin{lemma}\label{hardvvb}
Suppose we have a commutative diagram in \(Sm/k\) with the square being Cartesian:
\[
	\xymatrix
	{
		Y\ar[d]^i\ar[rd]^{id_Y}\ar@/_1pc/[dd]_j	&\\
		X\ar[r]\ar[d]						&Y\ar[d]_f\\
		Z\ar[r]^{g}						&S,
	}
\]
where \(f\), \(g\) are smooth and \(i\) is a closed immersion.
\begin{enumerate}
\item If \(j\) is a closed immersion, then the following diagram commutes
\[
	\xymatrix
	{
		T_{X/Y}|_Y+T_{Y/S}\ar[d]			&T_{X/S}|_Y\ar[l]\ar[r]	&T_{X/Z}|_Y+T_{Z/S}|_Y\ar[dd]\\
		N_{Y/X}+T_{Y/S}\ar[d]				&					&\\
		T_{X/Z}|_Y+N_{Y/Z}+T_{Y/S}\ar[rr]	&					&T_{X/Z}|_Y+T_{Y/S}+N_{Y/Z}.
	}
\]
\item If \(j\) is smooth, then the following diagram commutes
\[
	\xymatrix
	{
		T_{X/Y}|_Y+T_{Y/S}\ar[d]			&T_{X/S}|_Y\ar[l]\ar[r]	&T_{X/Z}|_Y+T_{Z/S}|_Y\ar[dd]\\
		N_{Y/X}+T_{Y/S}\ar[d]				&					&\\
		N_{Y/X}+T_{Y/Z}+T_{Z/S}|_Y\ar[rr]	&					&T_{Y/Z}+N_{Y/X}+T_{Z/S}|_Y.
	}
\]
\end{enumerate}
\end{lemma}
\begin{proof}
In both cases, there is a commutative diagrams with exact row and column
\[
	\xymatrix
	{
				&							&0\ar[d]				&			&\\
				&							&T_{Y/S}\ar[d]\ar@{=}[rd]	&			&\\
		0\ar[r]	&T_{X/Y}|_Y\ar[r]\ar[rd]_{\cong}	&T_{X/S}|_Y\ar[r]\ar[d]	&T_{Y/S}\ar[r]	&0\\
				&							&N_{Y/X}\ar[d]			&			&\\
				&							&0.					&			&
	}
\]
It induces a commutative diagram
\[
	\xymatrix
	{
		T_{X/S}|_Y\ar[r]\ar[d]	&T_{X/Y}|_Y+T_{Y/S}\ar[dl]\\
		T_{Y/S}+N_{Y/X}			&\\
	}
\]
by Theorem \ref{four diagrams}, (3).
\begin{enumerate}
\item We have a commutative diagram with exact columns and rows
\[
	\xymatrix
	{
				&					&0\ar[d]				&0\ar[d]				&\\
				&					&T_{Y/S}\ar[d]\ar@{=}[r]	&T_{Y/S}\ar[d]			&\\
		0\ar[r]	&T_{X/Z}|_Y\ar[r]\ar@{=}[d]	&T_{X/S}|_Y\ar[d]\ar[r]	&T_{Z/S}|_Y\ar[r]\ar[d]	&0\\
		0\ar[r]	&T_{X/Z}|_Y\ar[r]		&N_{Y/X}\ar[r]\ar[d]		&N_{Y/Z}\ar[r]\ar[d]		&0\\
				&					&0					&0,					&
	}
\]

So we get the following commutative diagram by Theorem \ref{four diagrams}, (3)
\[
	\xymatrix
	{
		T_{X/S}|_Y\ar[r]\ar[d]	&T_{X/Z}|_Y+T_{Z/S}|_Y\ar[r]	&T_{X/Z}|_Y+T_{Y/S}+N_{Y/Z}\ar[ddll]\\
		T_{Y/S}+N_{Y/X}\ar[d]		&						&\\
		T_{Y/S}+T_{X/Z}|_Y+N_{Y/Z}.	&						&
	}
\]
Furthermore, there is an obvious commutative diagram
\[
	\xymatrix
	{
		N_{Y/X}+T_{Y/S}\ar[d]				&T_{Y/S}+N_{Y/X}\ar[r]\ar[l]	&T_{Y/S}+T_{X/Z}|_Y+N_{Y/Z}\ar[d]\\
		T_{X/Z}|_Y+N_{Y/Z}+T_{Y/S}\ar[rr]	&						&T_{X/Z}|_Y+T_{Y/S}+N_{Y/Z}
	}.
\]
So the statement follows by combining the diagrams above.

\item We have a commutative diagram with exact columns and rows
\[
	\xymatrix
	{
				&0\ar[d]				&0\ar[d]				&					&\\
		0\ar[r]	&T_{Y/Z}\ar[r]\ar[d]		&T_{Y/S}\ar[d]\ar[r]		&T_{Z/S}|_Y\ar[r]\ar@{=}[d]	&0\\
		0\ar[r]	&T_{X/Z}|_Y\ar[r]\ar[d]	&T_{X/S}|_Y\ar[r]\ar[d]	&T_{Z/S}|_Y\ar[r]		&0\\
				&N_{Y/X}\ar[d]\ar@{=}[r]	&N_{Y/X}\ar[d]			&					&\\
				&0					&0					&					&
	}
\]
Then the statement follows by the same method as in (1) by applying Theorem \ref{four diagrams}, (2) to the diagram above.
\end{enumerate}
\end{proof}
\begin{lemma}\label{hardvvb1}
Suppose \(X, Y, Z\in Sm/S\) and \(g:Y\longrightarrow Z\) is a morphism in \(Sm/S\).
\begin{enumerate}
\item If \(g\) is a closed immersion, then the isomorphism \(t\) in Proposition \ref{easy composition1} is equal to
\[-T_{XY/X}\longrightarrow N_{XY/XZ}-N_{XY/XZ}-T_{XY/X}\longrightarrow N_{XY/XZ}-(id_X\times g)^*T_{XZ/X}.\]
\item If \(g\) is smooth, then the isomorphism \(t'\) in Proposition \ref{easy composition2} is equal to
\[-T_{XY/X}\longrightarrow-(id_X\times g)^*T_{XZ/X}-T_{XY/XZ}.\]
\end{enumerate}
\end{lemma}
\begin{proof}
We have a commutative diagram in \(Sm/k\) with the square being Cartesian:
\[
	\xymatrix@C=5em
	{
		XY\ar[d]^>>{id_X\times\Gamma_g}\ar[rd]^{id_{XY}}\ar@/_1pc/[dd]_{id_X\times g}	&\\
		XYZ\ar[r]_{p^{XYZ}_{XY}}\ar[d]^{p^{XYZ}_{XZ}}							&XY\ar[d]^{p^{XY}_X}\\
		XZ\ar[r]_{p^{XZ}_X}												&X.
	}
\]
\begin{enumerate}
\item We show that the composition
\begin{align*}
			&-T_{XY/X}\\
\longrightarrow	&-T_{XY/X}+N_{XY/XYZ}-(id_X\times\Gamma_g)^*T_{XYZ/XY}\\
\longrightarrow	&-T_{XY/X}+(id_X\times\Gamma_g)^*T_{XYZ/XZ}+N_{XY/XZ}-(id_X\times\Gamma_g)^*T_{XYZ/XY}\\
\longrightarrow	&N_{XY/XZ}-(id_X\times\Gamma_g)^*T_{XYZ/XY}\\
\longrightarrow	&N_{XY/XZ}-(id_X\times g)^*T_{XZ/X}\\
\longrightarrow	&N_{XY/XZ}-N_{XY/XZ}-T_{XY/X}\\
\longrightarrow &-T_{XY/X}
\end{align*}
is just \(id_{-T_{XY/X}}\). It is equal to
\begin{align*}
			&-T_{XY/X}\\
\longrightarrow	&-T_{XY/X}+N_{XY/XYZ}-(id_X\times\Gamma_g)^*T_{XYZ/XY}\\
\longrightarrow	&-T_{XY/X}-(id_X\times\Gamma_g)^*T_{XYZ/XY}+N_{XY/XYZ}\\
\longrightarrow	&-T_{XY/X}-(id_X\times g)^*T_{XZ/X}+N_{XY/XYZ}\\
\longrightarrow	&-T_{XY/X}-N_{XY/XZ}-T_{XY/X}+N_{XY/XYZ}\\
\longrightarrow	&-T_{XY/X}-N_{XY/XZ}-T_{XY/X}+(id_X\times\Gamma_g)^*T_{XYZ/XZ}+N_{XY/XZ}\\
\longrightarrow	&-N_{XY/XZ}-T_{XY/X}+N_{XY/XZ}\\
\longrightarrow	&-T_{XY/X},
\end{align*}
where the sixth arrow is the cancellation map between the first and the fourth term. By Lemma \ref{hardvvb}, (1) and the commutative diagram above, we have a commutative diagram
\[
	\xymatrix
	{
		(id_X\times\Gamma_g)^*T_{XYZ/XY}+T_{XY/X}\ar[d]			&(id_X\times\Gamma_g)^*T_{XYZ/X}\ar[l]\ar[d]\\
		N_{XY/XYZ}+T_{XY/X}\ar[d]							&(id_X\times\Gamma_g)^*T_{XYZ/XZ}+(id_X\times g)^*T_{XZ/X}\ar[d]\\
		(id_X\times\Gamma_g)^*T_{XYZ/XZ}+N_{XY/XZ}+T_{XY/X}\ar[r]	&(id_X\times\Gamma_g)^*T_{XYZ/XZ}+T_{XY/X}+N_{XY/XZ}.
	}
\]
Hence the composition above is equal to
\begin{align*}
			&-T_{XY/X}\\
\longrightarrow	&-T_{XY/X}+N_{XY/XYZ}-(id_X\times\Gamma_g)^*T_{XYZ/XY}\\
\longrightarrow	&-T_{XY/X}-(id_X\times\Gamma_g)^*T_{XYZ/XY}+N_{XY/XYZ}\\
\longrightarrow	&-T_{XY/X}-N_{XY/XYZ}+N_{XY/XYZ}\\
\longrightarrow	&-T_{XY/X}-N_{XY/XZ}-(id_X\times\Gamma_g)^*T_{XYZ/XZ}+N_{XY/XYZ}\\
\longrightarrow	&-T_{XY/X}-N_{XY/XZ}-(id_X\times\Gamma_g)^*T_{XYZ/XZ}+(id_X\times\Gamma_g)^*T_{XYZ/XZ}+N_{XY/XZ}\\
\longrightarrow	&-T_{XY/X},
\end{align*}
which gives the result.
\item We show that the composition
\begin{align*}
			&-T_{XY/X}\\
\longrightarrow	&-(id_X\times\Gamma_g)^*T_{XYZ/XY}+N_{XY/XYZ}-T_{XY/X}\\
\longrightarrow	&-(id_X\times\Gamma_g)^*T_{XYZ/XY}+N_{XY/XYZ}-(id_X\times\Gamma_g)^*T_{XYZ/XZ}\\
\longrightarrow	&-(id_X\times\Gamma_g)^*T_{XYZ/XY}+N_{XY/XYZ}-N_{XY/XYZ}-T_{XY/XZ}\\
\longrightarrow	&-(id_X\times\Gamma_g)^*T_{XYZ/XY}-T_{XY/XZ}\\
\longrightarrow	&-(id_X\times g)^*T_{XZ/X}-T_{XY/XZ}\\
\longrightarrow	&-T_{XY/X}
\end{align*}
is just \(id_{-T_{XY/X}}\). By Lemma \ref{hardvvb}, (2) and the commutative diagram in the beginning, we get a commutative diagram
\[
	\xymatrix
	{
		(id_X\times\Gamma_g)^*T_{XYZ/XY}+T_{XY/X}\ar[d]		&(id_X\times\Gamma_g)^*T_{XYZ/X}\ar[l]\ar[d]\\
		N_{XY/XYZ}+T_{XY/X}\ar[d]						&(id_X\times\Gamma_g)^*T_{XYZ/XZ}+(id_X\times g)^*T_{XZ/X}\ar[d]\\
		N_{XY/XYZ}+T_{XY/XZ}+(id_X\times g)^*T_{XZ/X}\ar[r]	&T_{XY/XZ}+N_{XY/XYZ}+(id_X\times g)^*T_{XZ/X}.
	}
\]
Hence the composition given is equal to
\begin{align*}
			&-T_{XY/X}\\
\longrightarrow	&-(id_X\times\Gamma_g)^*T_{XYZ/XY}+N_{XY/XYZ}-T_{XY/X}\\
\longrightarrow	&-(id_X\times\Gamma_g)^*T_{XYZ/XY}-T_{XY/X}+N_{XY/XYZ}\\
\longrightarrow	&-N_{XY/XYZ}-T_{XY/X}+N_{XY/XYZ}\\
\longrightarrow	&-N_{XY/XYZ}-(id_X\times g)^*T_{XZ/X}-T_{XY/XZ}+N_{XY/XYZ}\\
\longrightarrow	&-(id_X\times g)^*T_{XZ/X}-T_{XY/XZ}\\
\longrightarrow	&-T_{XY/X},
\end{align*}
where the fifth arrow is the cancellation between the first and the fourth term. Hence the result follows.
\end{enumerate}
\end{proof}
\begin{proposition}\label{identity}
For any \(X\in Sm/S\), \(\widetilde{\gamma}(id_X)\) is an identity. That is, for any \(X, Y\in Sm/S\), \(f\in\widetilde{Cor}_S(X,Y)\), \(g\in\widetilde{Cor}_S(Y,X)\), we have
\[\widetilde{\gamma}(id_Y)\circ f=f, g\circ\widetilde{\gamma}(id_X)=g.\]
\end{proposition}
\begin{proof}
The second equation follows by Proposition \ref{easy composition} and the first one follows from Lemma \ref{hardvvb}, (1) and Proposition \ref{easy composition1}.
\end{proof}

So combining Proposition \ref{associativity} and Proposition \ref{identity}, we have proved that \(\widetilde{Cor}_S\) is indeed a category.

Let's prove that \(\widetilde{\gamma}\) is indeed a functor.
\begin{proposition}\label{functor}
For any \(\xymatrix{X\ar[r]^f&Y\ar[r]^g&Z}\) in \(Sm/S\), we have
\[\widetilde{\gamma}(g\circ f)=\widetilde{\gamma}(g)\circ\widetilde{\gamma}(f).\]
\end{proposition}
\begin{proof}
Suppose at first \(f\) is a closed immersion or smooth. We have a Cartesian square
\[
	\xymatrix
	{
		XZ\ar[r]^{f\times id_Z}			&YZ\\
		X\ar[u]^{\Gamma_{g\circ f}}\ar[r]^f	&Y\ar[u]_{\Gamma_g}
	}
\]
and two cancelling ismorphisms \(a:N_{Y/YZ}-\Gamma_g^*T_{YZ/Y}\longrightarrow 0\) and \(b:N_{X/XZ}-\Gamma_f^*T_{XZ/X}\longrightarrow 0\). For convenience, we denote the induced morphisms between \(E\) still by \(a\) and \(b\), respectively. Then we have
\begin{align*}
	&\widetilde{\gamma}(g)\circ\widetilde{\gamma}(f)\\
=	&(f\times id_Z)^*(\widetilde{\gamma}(g))\\
	&\textrm{by Proposition \ref{easy composition}}\\
=	&(f\times id_Z)^*(\Gamma_{g*}(a^{-1}(1),N_{Y/YZ}-\Gamma_g^*T_{YZ/Y}))\\
	&\textrm{by definition of \(\widetilde{\gamma}\)}\\
=	&(\Gamma_{g\circ f})_*f^*(a^{-1}(1))\\
	&\textrm{by Axiom \ref{BCCI} for the square above}\\
=	&(\Gamma_{g\circ f})_*(b^{-1}(1))\\
	&\textrm{by Axiom \ref{FPB} and functoriality of pull-back with respect to twists}\\
=	&\widetilde{\gamma}(g\circ f)\\
	&\textrm{by definition of \(\widetilde{\gamma}\)}.
\end{align*}
Now suppose \(f=p\circ i\) in \(Sm/S\) where \(p\) is smooth and \(i\) is a closed immersion. Then
\[\widetilde{\gamma}(g)\circ\widetilde{\gamma}(f)=\widetilde{\gamma}(g)\circ\widetilde{\gamma}(i)\circ\widetilde{\gamma}(p)=\widetilde{\gamma}(i\circ g)\circ\widetilde{\gamma}(p)=\widetilde{\gamma}(g\circ f)\]
by the statements above.
\end{proof}
\begin{definition}
Define \(\widetilde{PSh}(S)\) to be the category of contravariant additive functors from \(\widetilde{Cor}_S\) to \(Ab\) as in \cite[Definition 1.2.1]{DF} and \cite[Definition 2.1]{MVW}, which are called presheaves with \(E\)-transfers over \(S\). Define moreover \(\widetilde{Sh}(S)\) to be its full subcategory with objects the presheaves whose restriction to \(Sm/S\) via \(\widetilde{\gamma}\) are Nisnevich sheaves. We call them sheaves with \(E\)-transfers over \(S\).
\end{definition}
\begin{definition}
Let \(X, Y\in Sm/S\), we define \(\widetilde{c}_S(X)\) by \(\widetilde{c}_S(X)(Y)=\widetilde{Cor}_S(Y,X)\). It is the presheaf with \(E\)-transfers representing \(X\).
\end{definition}
We recall the following three propositions which are the technical heart when dealing with Nisnevich sheaves:
\begin{proposition}\label{limit}
Let \(f:X\longrightarrow S\) be a morphism locally of finite type between locally noetherian schemes. Suppose \(I\) is a directed set and \(\{T_i\}\) is an inverse system of \(S\)-schemes such that for any \(i_1\preceq i_2\), the morphism \(T_{i_2}\longrightarrow T_{i_1}\) is affine. Then \(\varprojlim_i T_i\) exists in the category of \(S\)-schemes and we have
\[Hom_S(\varprojlim_i T_i,X)=\varinjlim_i Hom_S(T_i,X).\]
\end{proposition}
\begin{proof}
See \cite[Lemma 2.2]{SP} and \cite[Proposition 6.1]{SP}.
\end{proof}
Now let \(A\) be a noetherian ring and \(p\in\textrm{Spec}A\). Consider the set \(I\) whose elements are pairs \((B,q)\), where \(B\) is a connected \'etale \(A\)-algebra, \(q\in\textrm{Spec}B\), \(q\cap A=p\) and \(k(p)=k(q)\). Set \((B_1,q_1)\preceq(B_2,q_2)\) if there is an \(A\)-algebra morphism (always unique if exists) \(f:B_1\longrightarrow B_2\) such that \(f^{-1}(q_2)=q_1\).
\begin{proposition}\label{Henselization}
The set \(I\) is a directed set and we have
\[\varinjlim_{(B,q)}B\cong A_p^h,\]
where the right hand side is the Henselization of \(A_p\).
\end{proposition}
\begin{proof}
See the remarks around \cite[Lemma 4.8]{M} and see for example \cite[Theorem 4.2]{M} for basic properties of Henselian rings.
\end{proof}
\begin{proposition}\label{splitting}
Let \(U\), \(X\), \(Y\) be locally noetherian schemes, \(p:U\longrightarrow X\) be a Nisnevich covering and \(f:X\longrightarrow Y\) be a finite morphism. Then for every \(y\in Y\), there exists a scheme \(V\) with an \'etale morphism \(V\longrightarrow Y\) being Nisnevich at \(y\) such that the morphism \(U\times_YV\longrightarrow X\times_YV\) has a section.
\end{proposition}
\begin{proof}
Consider the following commutative diagram with Cartesian squares:
\[
	\xymatrix
	{
		U\ar[r]^p						&X\ar[r]^f				&Y\\
		R_2\ar[r]^{\alpha}\ar[u]_{\gamma}	&R_1\ar[r]^-{\beta}\ar[u]	&\textrm{Spec}O_{Y,y}^h.\ar[u]
	}
\]
Since \(\beta\) is a finite morphism, \(R_1\) is a finite direct product of Henselian rings (see \cite[Theorem 4.2]{M}). Hence \(\alpha\) has a section \(s\) since it is Nisnevich at every maximal ideal of \(R_1\). Pick an affine neighbourhood \(U_0\) of \(y\). By \cite[Lemma 2.3]{SP} and Proposition \ref{Henselization},
\[R_1=(\varprojlim_{(B,q)\succeq(O_Y(U_0),y)}\textrm{Spec}B)\times_{U_0}f^{-1}(U_0)=\varprojlim_{(B,q)\succeq(O_Y(U_0),y)}(\textrm{Spec}B\times_{U_0}f^{-1}(U_0)),\]
hence there exists a \((B,q)\succeq(O_Y(U_0),y)\) such that \(\gamma\circ s\) factor through the projection
\[\varprojlim_{(B,q)\succeq(O_Y(U_0),y)}(\textrm{Spec}B\times_{U_0}f^{-1}(U_0))\longrightarrow\textrm{Spec}B\times_{U_0}f^{-1}(U_0)\]
by using Proposition \ref{limit} for \(p\). And we finally let \(V=\textrm{Spec}B\).
\end{proof}

Now we are going to prove a similar result as in \cite[Lemma 1.2.6]{DF}.
\begin{proposition}\label{Cech}
Let \(X, U\in Sm/S\) and \(p:U\longrightarrow X\) be a Nisnevich covering. Denote the \(n\)-fold product \(A\times_BA\times_B\cdots A\) by \(A_B^n\) for any schemes \(A\) and \(B\). Then the following complex
\[\xymatrix{\cdots\ar[r]&\widetilde{c}_S(U_X^n)\ar[r]^-{d_n}&\cdots\ar[r]&\widetilde{c}_S(U\times_XU)\ar[r]^-{d_2}&\widetilde{c}_S(U)\ar[r]^-{d_1}&\widetilde{c}_S(X)\ar[r]^-{d_0}&0,}\]
denoted by \(\breve{C}(U/X)\), is exact after sheafifying as a complex of \(\widetilde{PSh}(S)\). Here, if \(p_i:U_X^n\longrightarrow U_X^{n-1}\) is the projection omitting \(i\)-th factor, then \(d_n=\sum_i(-1)^{i-1}\widetilde{c}_S(p_i)\).
\end{proposition}
\begin{proof}
For \(Y\in Sm/S\), we are to prove that the complex is exact at every point \(y\in Y\). Now assume we have an element \(a\in\widetilde{Cor}_S(Y,U^n_X)\) such that \(d_n(a)=0\). We may suppose that there is a \(T\in\mathscr{A}_S(Y,X)\) such that \(a\) comes from \(E_{R^n}^{d_X-d_S}((Y\times_SU)^n_{Y\times_SX},-T_{Y\times_SU_X^n/Y})\) and \(d_n(a)=0\), where \(R^n\) is defined by the following Cartesian squares \((R:=R^1)\)
\[
	\xymatrix
	{
		R^n\ar[r]\ar[d]			&Y\times_SU_X^n\ar[d]\ar[r]	&U_X^n\ar[d]\\
		T\ar[r]				&Y\times_SX\ar[r]			&X.		
	}
\]
By Proposition \ref{splitting}, there is a Nisnevich neighbourhood \(V\) of \(y\) such that the map \(p:R\times_YV\longrightarrow T\times_YV\) has a section \(s\), which is both an open immersion and a closed immersion (see \cite[Corollary 3.12]{M}). And let \(D=(R\times_YV)\setminus s(T\times_YV)\). Then \(d_n(a|_{V\times_SU_X^n})=0\).
We have a commutative diagram
\[
	\xymatrix
	{
		V\times_SU_X^n\ar[r]\ar[d]	&Y\times_SU^n_X\ar[d]\\
		V\times_SX\ar[r]		&Y\times_SX,
	}
\]
Cartesian squares
\[
	\xymatrix
	{
		R^n\times_YV\ar[r]\ar[d]	&V\times_SU_X^n\ar[d]\ar[r]	&V\ar[d]\\
		R^n\ar[r]				&Y\times_SU_X^n\ar[r]			&Y,
	}
	\xymatrix
	{
		T\times_YV\ar[r]^-s\ar[d]_s	&(V\times_SU)\setminus D\ar[d]\\
		R\times_YV\ar[r]			&V\times_SU,
	}
\]
equations
\[Y\times_SU_X^n=(Y\times_SU)^n_{Y\times_SX},\]
\[V\times_SU_X^n=(V\times_SU)^n_{V\times_SX},\]
\[R^n=R^n_T,\]
\[R^n\times_YV=(R\times_YV)^n_{T\times_YV}=(T\times_{Y\times_SX}(V\times_SU))^n_{T\times_YV},\]
\[(R\times_YV)^n_{T\times_YV}=(R\times_YV)^n_{V\times_SX},\]
and a diagram of Cartesian squares with right-hand vertical maps being \'etale:
\[
	\xymatrix
	{
		(R\times_YV)^n_{T\times_YV}\ar[r]\ar[d]^{id^n\times s}\ar@/_3pc/[dd]_{id}	&(V\times_SU)^n_{V\times_SX}\times_{V\times_SX}((V\times_SU)\setminus D):=W^{n+1}\ar[d]_{j_{n+1}}\\
		(R\times_YV)^{n+1}_{T\times_YV}\ar[r]\ar[d]^{p_{n+1}}					&(V\times_SU)^{n+1}_{V\times_SX}\ar[d]_{p_{n+1}}\\
		(R\times_YV)^n_{T\times_YV}\ar[r]								&(V\times_SU)^n_{V\times_SX},
	}
\]
where \(p_{n+1}\) denotes the projection omitting the last factor. Moreover, the maps
\[\xymatrix{E^{d_X-d_S}_{R^n\times_YV}((V\times_SU)^n_{V\times_SX},-T_{V\times_SU^n_X/V})\ar[r]^{(p_{n+1}\circ j_{n+1})^*}&E^{d_X-d_S}_{R^n\times_YV}(W^{n+1},-T_{V\times_SU^{n+1}_X/V}|_{W^{n+1}})}\]
and
\[\xymatrix{E^{d_X-d_S}_{R^n\times_YV}((V\times_SU)^{n+1}_{V\times_SX},-T_{V\times_SU^{n+1}_X/V})\ar[r]^{j_{n+1}^*}&E^{d_X-d_S}_{R^n\times_YV}(W^{n+1},-T_{V\times_SU^{n+1}_X/V}|_{W^{n+1}})}\]
are isomorphisms with respective inverses \((p_{n+1}\circ j_{n+1})_{*}\) and \((j_{n+1})_*\) by Axiom \ref{EE}.

Let's consider the element
\[b:=((j_{n+1}^*)^{-1}\circ(p_{n+1}\circ j_{n+1})^*)(a|_{(V\times_SU)^n_{V\times_SX}})\in E^{d_X-d_S}_{R^n\times_YV}((V\times_SU)^{n+1}_{V\times_SX},-T_{V\times_SU^{n+1}_X/V}),\]
where we have used the isomorphism
\[p_{n+1}^*T_{V\times_SU_X^n/V}\longrightarrow T_{V\times_SU_X^{n+1}/V}\]
since \(U\longrightarrow X\) is \'etale. Then
\[d_{n+1}(b)=\sum_{i=1}^{n+1}(-1)^{i-1}\widetilde{c}_S(p_i)(b)=\sum_{i=1}^{n+1}(-1)^{i-1}p_{i*}(t_{i,n+1}(b))\]
by Proposition \ref{easy composition1}, where
\[t_{i,n+1}:-T_{V\times_SU_X^{n+1}/V}\longrightarrow-(id_V\times_Sp_i)^*T_{(V\times_SU^n_X)/V}-T_{(V\times_SU_X^{n+1})/(V\times_SU_X^n)}\]
is the isomorphism as in the Proposition \ref{easy composition1} applying to
\[\xymatrixcolsep{5pc}\xymatrix{V\ar[r]^b&U^{n+1}_X\ar[r]^{p_i}&U^n_X}.\]
If \(1\leq i<n+1\), we have Cartesian squares
\[
	\xymatrix
	{
		W^{n+1}\ar[r]_-{p_{n+1}\circ j_{n+1}}\ar[d]_{p_i}	&(V\times_SU)^n_{V\times_SX}\ar[d]_{p_i}\\
		W^n\ar[r]^-{p_n\circ j_n}						&(V\times_SU)^{n-1}_{V\times_SX}
	}
	\xymatrix
	{
		W^{n+1}\ar[r]_-{j_{n+1}}\ar[d]_{p_i}	&(V\times_SU)^{n+1}_{V\times_SX}\ar[d]_{p_i}\\
		W^n\ar[r]^-{j_n}			&(V\times_SU)^n_{V\times_SX}.
	}
\]
So
\begin{align*}
	&p_{i*}(t_{i,n+1}(b))\\
=	&(p_{i*}\circ t_{i,n+1}\circ(j_{n+1}^*)^{-1}\circ(p_{n+1}\circ j_{n+1})^*)((a|_{(V\times_SU)^n_{V\times_SX}},-T_{V\times_SU_X^n/V}))\\
	&\textrm{by definition}\\
=	&(p_{i*}\circ(j_{n+1}^*)^{-1}\circ j_{n+1}^*(t_{i,n+1})\circ(p_{n+1}\circ j_{n+1})^*)(a|_{(V\times_SU)^n_{V\times_SX}})\\
	&\textrm{by functoriality of pullback with respect to twists}\\
=	&((j_n^*)^{-1}\circ p_{i*}\circ j_{n+1}^*(t_{i,n+1})\circ(p_{n+1}\circ j_{n+1})^*)(a|_{(V\times_SU)^n_{V\times_SX}})\\
	&\textrm{by Axiom \ref{BCSM} for the right square above}\\
=	&((j_n^*)^{-1}\circ p_{i*}\circ(p_{n+1}\circ j_{n+1})^*\circ t_{i,n})(a|_{(V\times_SU)^n_{V\times_SX}})\\
	&\textrm{by functoriality of pull-back with respect to twists}\\
=	&((j_n^*)^{-1}\circ(p_n\circ j_n)^*\circ p_{i*}\circ t_{i,n})(a|_{(V\times_SU)^n_{V\times_SX}})\\
	&\textrm{by Axiom \ref{BCSM} for the left square above.}
\end{align*}
Now let \(i=n+1\), we have
\begin{align*}
	&p_{n+1*}(t_{n+1,n+1}(b))\\
=	&(p_{n+1*}\circ t_{n+1,n+1}\circ(j_{n+1}^*)^{-1}\circ(p_{n+1}\circ j_{n+1})^*)((a|_{(V\times_SU)^n_{V\times_SX}},-T_{V\times_SU_X^n/V}))\\
	&\textrm{by definition}\\
=	&(p_{n+1*}\circ t_{n+1,n+1}\circ j_{n+1*}\circ(p_{n+1}\circ j_{n+1})^*)(a|_{(V\times_SU)^n_{V\times_SX}})\\
	&\textrm{by Axiom \ref{EE}}\\
=	&(p_{n+1*}\circ j_{n+1*}\circ j_{n+1}^*(t_{n+1,n+1})\circ(p_{n+1}\circ j_{n+1})^*)(a|_{(V\times_SU)^n_{V\times_SX}})\\
	&\textrm{by functoriality of push-forwards with respect to twists}\\
=	&((p_{n+1}\circ j_{n+1})_*\circ(p_{n+1}\circ j_{n+1})^*)(a|_{(V\times_SU)^n_{V\times_SX}})\\
	&\textrm{by Axiom \ref{FPFSM} and Lemma \ref{hardvvb1}, (2)}\\
=	&a|_{(V\times_SU)^n_{V\times_SX}}\\
	&\textrm{by Axiom \ref{EE}}.
\end{align*}
Hence 
\begin{align*}
	&d_{n+1}(b)\\
=	&((j_n^*)^{-1}\circ(p_n\circ j_n)^*\circ d_n)(a|_{(V\times_SU)^n_{V\times_SX}})+(-1)^na|_{(V\times_SU)^n_{V\times_SX}}\\
=	&(-1)^na|_{(V\times_SU)^n_{V\times_SX}}.
\end{align*}
So the complex is exact after Nisnevich sheafication.
\end{proof}
Then by the same proofs as in the \cite[1.2.7-1.2.11]{DF}, we have the following result:
\begin{proposition}\label{sheafication}
(1) The forgetful functor \(\widetilde{o}:\widetilde{Sh}(S)\longrightarrow\widetilde{PSh}(S)\) has a left adjoint \(\widetilde{a}\) such that the following diagram commutes:
\[
	\xymatrix
	{
		PSh(S)\ar[d]_a	&\widetilde{PSh}(S)\ar[l]_{\widetilde{\gamma}_*}\ar[d]_{\widetilde{a}}\\
		Sh(S)			&\widetilde{Sh}(S)\ar[l]_{\widetilde{\gamma}_*}
	},
\]
where \(a\) is the Nisnevich sheafication functor with respect to the smooth site over \(S\).\\
(2) The category \(\widetilde{Sh}(S)\) is a Grothendieck abelian category and the functor \(\widetilde{a}\) is exact.\\
(3) The functor \(\widetilde{\gamma}_*\) appearing in the lower line of the preceding diagram, admits a left adjoint \(\widetilde{\gamma}^*\), and commutes with every limits and colimits.
\end{proposition}
\begin{proof}
The same as \cite[Proposition 1.2.11]{DF}.
\end{proof}
\begin{definition}
Given any \(X\in Sm/S\), we define \(\widetilde{\mathbb{Z}}_S(X)=\widetilde{a}(\widetilde{c}_S(X))\). We denote \(\widetilde{\mathbb{Z}}_S(S)\) by \(\mathbbm{1}_S\).
\end{definition}
\begin{proposition}\label{MV-sequence}
Let \(X\in Sm/S\) and \(U_1\cup U_2=X\) be a Zariski covering. Then the following complex is exact as sheaves with \(E\)-transfers:
\[0\longrightarrow\widetilde{\mathbb{Z}}_S(U_1\cap U_2)\longrightarrow\widetilde{\mathbb{Z}}_S(U_1)\oplus\widetilde{\mathbb{Z}}_S(U_2)\longrightarrow\widetilde{\mathbb{Z}}_S(X)\longrightarrow 0.\]
\end{proposition}
\begin{proof}
See \cite[Proposition 6.14]{MVW} with use of Proposition \ref{Cech}.
\end{proof}
\begin{definition}
Let \(X_i, Y_i\in Sm/S, i=1,2\), for any \(f_1\in\widetilde{Cor}_S(X_1,Y_1)\), \(f_2\in\widetilde{Cor}_S(X_2,Y_2)\). Define
\[f_1\times_Sf_2=p_1^*f_1\cdot p_2^*f_2\in\widetilde{Cor}(X_1\times_SX_2,Y_1\times_SY_2)\]
to be the exterior product of \(f_1\) and \(f_2\), where \(p_i:X_1\times_SX_2\times_SY_1\times_SY_2\longrightarrow X_i\times_SY_i, i=1,2\) are the projections. Here we have used the isomorphism \(-T_{X_1X_2Y_1Y_2/X_1X_2}\longrightarrow-T_{X_1X_2Y_1Y_2/X_1X_2Y_2}-T_{X_1X_2Y_1Y_2/X_1X_2Y_1}\).
\end{definition}
Now we are going to show that exterior products are compatible with compositions.
\begin{lemma}\label{exterior}
Let \(X_i, Y_i, Z_i\in Sm/S, i=1,2\), we have maps \(p_{13}^i:X_iY_iZ_i\longrightarrow X_iZ_i\), \(a_i:X_1X_2Y_1Y_2Z_1Z_2\longrightarrow X_iY_iZ_i\) and \(p_{13}:X_1X_2Y_1Y_2Z_1Z_2\longrightarrow X_1X_2Z_1Z_2\). Suppose \(\alpha_i\in E^{d_{Y_i}+d_{Z_i}}_{C_i}((p_{13}^i)^*v_i-T_{X_iY_iZ_i/X_iZ_i})\) where \(C_i\in\mathscr{A}_S(X_i,Y_iZ_i)\) and \(v_i\in\mathscr{P}_{X_iZ_i})\). Then we have
\[b_1^*p_{13*}^1(\alpha_1)\cdot b_2^*p_{13*}^2(\alpha_2)=p_{13*}(a^*_1(\alpha_1)\cdot a^*_2(\alpha_2)),\]
where we have used the isomorphism (exchanging the middle two terms and then merging the last two terms) from
\[a_1^*(p_{13}^1)^*v_1-T_{X_1X_2Y_1Y_2Z_1Z_2/X_1X_2Y_2Z_1Z_2}+a_2^*(p_{13}^2)^*v_2-T_{X_1X_2Y_1Y_2Z_1Z_2/X_1X_2Y_1Z_1Z_2}\]
to
\[p_{13}^*(b_1^*(v_1)+b_2^*(v_2))-T_{X_1X_2Y_1Y_2Z_1Z_2/X_1X_2Z_1Z_2}\]
in the right hand side.
\end{lemma}
\begin{proof}
We have two Cartesian squares
\[
	\xymatrix
	{
		X_2Y_2Z_2\ar[r]^{p_{13}^2}				&X_2Z_2\\
		X_1X_2Y_1Y_2Z_1Z_2\ar[u]^{a_2}\ar[r]^q	&X_1X_2Y_1Z_1Z_2\ar[u]^{p_{25}}
	}
	\xymatrix
	{
		X_1Y_1Z_1\ar[r]^{p_{13}^1}				&X_1Z_1\\
		X_1X_2Y_1Z_1Z_2\ar[r]^{p_{1245}}\ar[u]^{p}	&X_1X_2Z_1Z_2\ar[u]^{b_1}
	}
\]
and equations \(p_{25}=b_2\circ p_{1245}\), \(p\circ q=a_1\) and \(p_{13}=p_{1245}\circ q\). Then we have
\begin{align*}
	&b_1^*p_{13*}^1((\alpha_1,(p_{13}^1)^*v_i-T_{X_1Y_1Z_1/X_1Z_1}))\cdot b_2^*p_{13*}^2(\alpha_2)\\
=	&(p_{1245})_*p^*(\alpha_1)\cdot b_2^*p_{13*}^2(\alpha_2)\\
	&\textrm{by Axiom \ref{BCSM} for the right square above}\\
=	&(p_{1245})_*(p^*(\alpha_1)\cdot p_{1245}^*b_2^*p_{13*}^2(\alpha_2))\\
	&\textrm{by Axiom \ref{PFSM} for \(p_{1245}\)}\\
=	&(p_{1245})_*(p^*(\alpha_1)\cdot p_{25}^*p_{13*}^2(\alpha_2))\\
	&\textrm{by Axiom \ref{FPB}}\\
=	&(p_{1245})_*(p^*(\alpha_1)\cdot q_*a_2^*(\alpha_2))\\
	&\textrm{by Axiom \ref{BCSM} for the left square above}\\
=	&(p_{1245})_*q_*(q^*p^*(\alpha_1)\cdot a_2^*(\alpha_2))\\
	&\textrm{by Axiom \ref{PFSM} for \(q\)}\\
=	&(p_{1245})_*q_*(a_1^*(\alpha_1)\cdot a_2^*(\alpha_2))\\
	&\textrm{by Axiom \ref{FPB}}\\
=	&p_{13*}(a_1^*(\alpha_1)\cdot a_2^*(\alpha_2))\\
	&\textrm{by Axiom \ref{FPFSM}}.
\end{align*}
\end{proof}
\begin{proposition}
Let \(X_i, Y_i, Z_i\in Sm/S\), \(f_i\in\widetilde{Cor}_S(X_i,Y_i)\), \(g_i\in\widetilde{Cor}_S(Y_i,Z_i)\) where \(i=1,2\). Then
\[(g_1\times_Sg_2)\circ(f_1\times_Sf_2)=(g_1\circ f_1)\times_S(g_2\circ f_2).\]
\end{proposition}
\begin{proof}
We have a commutative diagram \((i=1,2)\)
\[
	\xymatrix
	{
		Y_1Y_2Z_1Z_2\ar[r]^-{q_i\times r_i}										&Y_iZ_i 											&\\
		X_1X_2Y_1Y_2Z_1Z_2\ar[u]^{p_{23}}\ar[d]^{p_{12}}\ar[r]^-{a_i}\ar[rrd]_<<<<<<{p_{13}}	&X_iY_iZ_i\ar[r]^{p_{13}^i}\ar[u]_{p_{23}^i}\ar[d]^<<{p_{12}^i}	&X_iZ_i\\
		X_1X_2Y_1Y_2\ar[r]^-{p_i\times q_i}										&X_iY_i 											&X_1X_2Z_1Z_2\ar[u]_{b_i}
	}.
\]
Then 
\begin{align*}
	&(g_1\times_Sg_2)\circ(f_1\times_Sf_2)\\
=	&p_{13*}(p^*_{23}((q_1\times r_1)^*g_1\cdot(q_2\times r_2)^*g_2)\cdot p^*_{12}((q_1\times r_1)^*f_1\cdot(q_2\times r_2)^*f_2)\\
	&\textrm{by definition}\\
=	&p_{13*}(a_1^*(p_{23}^1)^*(g_1)\cdot a_2^*(p_{23}^2)^*(g_2)\cdot a_1^*(p_{12}^1)^*(f_1)\cdot a_2^*(p_{12}^2)^*(f_2))\\
	&\textrm{by Axiom \ref{CPB} and Axiom \ref{FPB}.}\\
%
%
=	&p_{13*}(c(a_1^*(p_{23}^1)^*(g_1)\cdot a_1^*(p_{12}^1)^*(f_1)\cdot a_2^*(p_{23}^2)^*(g_2)\cdot a_2^*(p_{12}^2)^*(f_2)))\\
	&\textrm{by Axiom \ref{CC} and Axiom \ref{BCCI}, \(c=c(a_1^*(p_{12}^1)^*(-T_{X_1Y_1/X_1}),a_2^*(p_{23}^2)^*(-T_{Y_2Z_2/Y_2}))\)}\\
=	&p_{13*}(c(a_1^*((p_{23}^1)^*(g_1)\cdot (p_{12}^1)^*(f_1))\cdot a_2^*((p_{23}^2)^*(g_2)\cdot (p_{12}^2)^*(f_2))))\\
	&\textrm{by Axiom \ref{CPB}}\\
=	&b_1^*p_{13*}^1((p_{23}^1)^*(g_1)\cdot (p_{12}^1)^*(f_1))\cdot b_2^*p_{13*}^2((p_{23}^2)^*(g_2)\cdot (p_{12}^2)^*(f_2))\\
	&\textrm{by Lemma \ref{exterior}}\\
=	&b_1^*(g_1\circ f_1)\cdot b_2^*(g_2\circ f_2)\\
	&\textrm{by definition}\\
=	&(g_1\circ f_1)\times_S(g_2\circ f_2)\\
	&\textrm{by definition.}
\end{align*}
%
%
\end{proof}

Here are some basic constructions of presheaves or sheaves with \(E\)-transfers.

For any \(F\in\widetilde{PSh}(S)\) and \(X\in Sm/S\), we define \(F^X\in\widetilde{PSh}(S)\) by \(F^X(Y)=F(X\times_SY)\). If \(F\in\widetilde{Sh}(S)\), then it's clear that \(F^X\in\widetilde{Sh}(S)\) also. We define \(C_*F\) for any \(F\in\widetilde{Sh}(S)\) to be a complex with \((C_*F)_n=F^{\triangle^n}\) as in \cite[Definition 2.14]{MVW}.

A pointed scheme is a pair \((X,x)\) where \(X\in Sm/S\) and \(x:S\longrightarrow X\) is a S-rational point. 
For any pointed scheme \((X,x)\), we define
\[\widetilde{\mathbb{Z}}_S(X,x)=Coker(\xymatrix{\mathbbm{1}_S\ar[r]^-x&\widetilde{\mathbb{Z}}_S(X)}).\]
Then we define \(\widetilde{\mathbb{Z}}_S(q)=\widetilde{\mathbb{Z}}_S(\mathbb{G}_m,1)^{\otimes q}[-q]\) for \(q\geq 0\) and we set \(\widetilde{\mathbb{Z}}_S=\widetilde{\mathbb{Z}}_S(0)=\mathbbm{1}_S\). The notation like \(\widetilde{\mathbb{Z}}_S(q)[2q]\) means \((\widetilde{\mathbb{Z}}_S(1)[2])^{\otimes q}\).
%
%
%

The following definitions comes from \cite[Lemma 2.1]{SV}.
\begin{definition}
Suppose \(F_i, G\in\widetilde{PSh}(S), i=1,\cdots,n (n\geq 2)\). A multilinear function \(\varphi:F_1\times\cdots\times F_n\longrightarrow G\) is a collection of multilinear maps of abelian groups
\[\varphi_{(X_1,\cdots,X_n)}:F_1(X_1)\times\cdots\times F_n(X_n)\longrightarrow G(X_1\times_S\cdots\times_SX_n)\]
for every \(X_i\in Sm/S\), such that for every \(f\in\widetilde{Cor}_S(X_i,X_i')\), we have a commutative diagram
\[
	\xymatrix
	{
		\cdots\times F_i(X_i')\times\cdots\ar[r]^-{\varphi(\cdots,X_i',\cdots)}\ar[d]_{\cdots\times f\times\cdots}	&G(\cdots\times_SX_i'\times_S\cdots)\ar[d]_{G(\cdots\times f\times\cdots)}\\
		\cdots\times F_i(X_i)\times\cdots\ar[r]^-{\varphi(\cdots,X_i,\cdots)}								&G(\cdots\times_SX_i\times_S\cdots)
	}.
\]
\end{definition}
\begin{definition}
Suppose \(F_i, G\in\widetilde{PSh}(S), i=1,\cdots,n (n\geq 2)\) (resp. \(\widetilde{Sh}(S)\)). The tensor product \(F_1\otimes_S^{pr}\cdots\otimes_S^{pr}F_n\) (resp. \(F_1\otimes_S\cdots\otimes_SF_n\)) is a presheaf (resp. sheaf) with \(E\)-transfers \(G\) such that for any \(H\in\widetilde{PSh}(S)\) (resp. \(\widetilde{Sh}(S)\)), we have
\[Hom_S(G,H)\cong\{\textrm{Multilinear functions }F_1\times\cdots\times F_n\longrightarrow H\}\]
naturally.
\end{definition}
For any \(F, G\in\widetilde{PSh}(S)\), we can construct \(F\otimes^{pr}_SG\in\widetilde{PSh}(S)\) as in the discussion before \cite[Lemma 2.1]{SV}. It has the universal property above. Moreover, we define \(\underline{Hom}_S(F,G)\) to be a presheaf with \(E\)-transfers which sends \(X\in Sm/S\) to \(Hom_S(F,G^X)\). And if \(F, G\) are sheaves with \(E\)-transfers, we define \(F\otimes_SG=\widetilde{a}(F\otimes^{pr}_SG)\). If \(G\) is a sheaf with \(E\)-transfers, it's clear that \(\underline{Hom}_S(F,G)\) is also a sheaf with \(E\)-transfers. Finally, it's clear from definition that \(F\otimes^{pr}_SG\cong G\otimes^{pr}_SF\) and \(F\otimes_SG\cong G\otimes_SF\).
\begin{proposition}\label{hom-tensor}
For any \(F,G,H\in\widetilde{PSh}(S)\), we have isomorphisms
\[Hom_S(F\otimes_S^{pr}G,H)\cong Hom_S(F,\underline{Hom}_S(G,H)),\]
\[Hom_S(F\otimes_S^{pr}G,H)\cong Hom_S(G,\underline{Hom}_S(F,H))\]
being functorial in three variables. Simililarly, for any \(F,G,H\in\widetilde{Sh}(S)\), we have isomorphisms
\[Hom_S(F\otimes_SG,H)\cong Hom_S(F,\underline{Hom}_S(G,H)),\]
\[Hom_S(F\otimes_SG,H)\cong Hom_S(G,\underline{Hom}_S(F,H))\]
being functorial in three variables.
\end{proposition}
\begin{proof}
This is clear from the definition of the bilinear map.
\end{proof}

Moreover, if we have \(F, G, H\in\widetilde{Sh}(S)\), it's easy to see by the above proposition that \((F\otimes_SG)\otimes_SH\) and \(F\otimes_S(G\otimes_SH)\) are all isomorphic to \(F\otimes_SG\otimes_SH\). So the tensor product is associative. And finally one checks that \(\otimes_S\) (resp. \(\otimes_S^{pr}\)) gives \(\widetilde{Sh}(S)\) (resp. \(\widetilde{PSh}(S)\)) a symmetric closed monoidal structure.
\begin{proposition}\label{sheafication1}
If a morphism \(f:F_1\longrightarrow F_2\) of presheaves with \(E\)-transfers becomes an isomorphism after sheafifying, then so does the morphism \(f\otimes^{pr}_SG\) for any presheaf with \(E\)-transfers \(G\).
\end{proposition}
\begin{proof}
The condition is equivalent to the map \(Hom_S(f,H)\) is an isomorphism between abelian groups for any sheaf with \(E\)-transfers \(H\). And
\[Hom_S(f\otimes^{pr}_SG,H)\cong Hom_S(f,\underline{Hom}_S(G,H))\]
by the proposition above.
\end{proof}
\begin{proposition}\label{tensor}
For any \(X, Y\in Sm/S\), we have
\[\widetilde{\mathbb{Z}}_S(X)\otimes_S\widetilde{\mathbb{Z}}_S(Y)\cong\widetilde{\mathbb{Z}}_S(X\times_SY)\]
as sheaves with \(E\)-transfers.
\end{proposition}
\begin{proof}
We have \(\widetilde{c}_S(X)\otimes^{pr}_S\widetilde{c}_S(Y)\cong\widetilde{c}_S(X\times_SY)\) just by the exterior products of correspondences. Then the statement follows by Proposition \ref{sheafication1}.
%
\end{proof}

Now we are going to prove some functorial properties between sheaves with \(E\)-transfers over different bases. Our approach is quite similar as in \cite{D}.

The following lemma is useful when constructing adjunctions, see \cite[2.5.1]{D}.
\begin{lemma}\label{adjunction}
Let \(\varphi:\mathscr{C}\longrightarrow\mathscr{D}\) be a functor between small categories and \(\mathscr{M}\) be a category with arbitrary colimits. Then the functor
\[\varphi_*:PreShv(\mathscr{D},\mathscr{M})\longrightarrow PreShv(\mathscr{C},\mathscr{M})\]
(see \cite[Definition 4.4.1]{A}) defined by \(\varphi_*(F)=F\circ\varphi\) has a left adjoint \(\varphi^*\).
\end{lemma}
\begin{proof}
Suppose \(G\in PreShv(\mathscr{C},\mathscr{M})\). For every object \(Y\in\mathscr{D}\), define \(C_Y\) to be the category whose objects are \(Hom_{\mathscr{D}}(Y,\varphi X)\) and morphisms from \(a_1:Y\longrightarrow\varphi X_1\) to \(a_2:Y\longrightarrow\varphi X_2\) are \(b\in Hom_{\mathscr{C}}(X_1,X_2)\) such that \(a_2=\varphi(b)\circ a_1\). We have a contravariant functor
\[\theta_Y:C_Y\longrightarrow\mathscr{M}\]
defined by \(\theta_Y(Y\longrightarrow\varphi X)=GX\). Then define \((\varphi^*G)Y=\varinjlim\theta_Y\). For any morphism \(c:Y_1\longrightarrow Y_2\) in \(\mathscr{D}\), we define \((\varphi^*G)(c)\) by the following commutative diagram
\[
	\xymatrix
	{
		\theta_{Y_2}(a)\ar[d]^{i_a}\ar[dr]^{i_{a\circ c}}	&\\
		\varinjlim\theta_{Y_2}\ar[r]_{(\varphi^*G)(c)}		&\varinjlim\theta_{Y_1}
	}
\]
for every \(a:Y_2\longrightarrow\varphi X\). One checks it is just what we want.
\end{proof}
\begin{definition}\label{pullbackdef}
Suppose \(f:S\longrightarrow T\) is a morphism in \(Sm/k\). For any \(X\in Sm/T\), set \(X^S=X\times_TS\in Sm/S\). For any \(X_1, X_2\in Sm/T\), denote by \(p_f\) the projection \((X_1\times_TX_2)^S\longrightarrow X_1\times_TX_2\). Define
\[\begin{array}{cccc}\varphi^f:&\widetilde{Cor}_T&\longrightarrow&\widetilde{Cor}_S\\&X&\longmapsto&X^S\\&g&\longmapsto&g^S\end{array},\]
where \(g\longmapsto g^S:\widetilde{Cor}_T(X_1,X_2)\longrightarrow\widetilde{Cor}_S(X_1^S,X_2^S)\) is the unique map such that the following diagram commutes
\[
	\xymatrix
	{
		E_Z^{d_{X_2}-d_T}(X_1\times_TX_2,-T_{X_1\times_TX_2/X_1})\ar[r]^-{p_f^*}\ar[d]	&E_{p_f^{-1}(Z)}^{d_{X_2}-d_S}((X_1\times_TX_2)^S,-T_{(X_1\times_TX_2)^S/X_1^S})\ar[d]\\
		\widetilde{Cor}_T(X_1,X_2)\ar[r]^-{\varphi^f}					&\widetilde{Cor}_S(X_1^S,X_2^S)
	},
\]
for any \(Z\in\mathscr{A}_T(X_1,X_2)\).
\end{definition}
\begin{proposition}
Suppose \(\xymatrix{X_1\ar[r]^{g_1}&X_2\ar[r]^{g_2}&X_3}\) are morphisms in \(\widetilde{Cor}_T\). Then
\[(g_2\circ g_1)^S=g_2^S\circ g_1^S.\]
So \(\varphi^f:\widetilde{Cor}_T\longrightarrow\widetilde{Cor}_S\) is indeed a functor.
\end{proposition}
\begin{proof}
We have diagrams
\[
	\xymatrix
	{
					&X_1\times_TX_2\\
		X_1\times_TX_3	&X_1\times_TX_2\times_TX_3\ar[u]_{p_{12}}\ar[l]_-{p_{13}}\ar[d]^{p_{23}}\\
					&X_2\times_TX_3
	},
	\xymatrix
	{
						&X_1^S\times_SX_2^S\\
		X_1^S\times_SX_3^S	&X_1^S\times_SX_2^S\times_SX_3^S\ar[u]_{q_{12}}\ar[l]_-{q_{13}}\ar[d]^{q_{23}}\\
						&X_2^S\times_SX_3^S
	},
\]
and three Cartesian squares
\[
	\xymatrix
	{
								&X_2\times_SX_3															&\\
		X_1\times_TX_3				&X_1\times_TX_2\times_TX_3\ar[l]_-{p_{13}}\ar[r]^-{p_{12}}\ar[u]_{p_{23}}				&X_1\times_SX_2\\
		X_1^S\times_SX_3^S\ar[u]_{r}	&X_1^S\times_SX_2^S\times_SX_3^S\ar[l]_-{q_{13}}\ar[u]_t\ar[d]^{q_{23}}\ar[r]^-{q_{12}}	&X_1^S\times_SX_2^S\ar[u]_p\\
								&X_2^S\times_SX_3^S\ar@/^2pc/[uuu]^q											&
	}.
\]
Suppose \(g_1\) and \(g_2\) come from cohomologies with admissible supports. We have
\begin{align*}
	&(g_2\circ g_1)^S\\
=	&r^*p_{13*}(p_{23}^*(g_2)\cdot p_{12}^*(g_1))\\
	&\textrm{by definition}\\
=	&q_{13*}t^*(p_{23}^*(g_2)\cdot p_{12}^*(g_1))\\
	&\textrm{by Axiom \ref{BCSM} for the left square above}\\
=	&q_{13*}(q_{12}^*p^*(g_2)\cdot q_{23}^*q^*(g_1))\\
	&\textrm{by Axiom \ref{CPB} and Axiom \ref{FPB}}\\
=	&g_2^S\circ g_1^S\\
	&\textrm{by definition.}
\end{align*}
And it's easy to verify that \(\widetilde{\gamma}(id_Y)^S=\widetilde{\gamma}(id_{Y^S})\) for any \(Y\in Sm/T\). So \(\varphi^f\) is a functor.
\end{proof}
It is straightforward to check that \(\varphi^{f_1\circ f_2}=\varphi^{f_2}\circ\varphi^{f_1}\).
\begin{proposition}\label{ex*}
Suppose \(f_i\in\widetilde{Cor}_T(X_i,Y_i)\) where \(i=1,2\). Then
\[(f_1\times_Tf_2)^S=f_1^S\times_Sf_2^S.\]
\end{proposition}
\begin{proof}
This follows from the commutative diagram
\[
	\xymatrix
	{
									&(X_1Y_1X_2Y_2)^S\ar[rrr]^-{p^f}\ar[rd]\ar[ld]	&							&		&X_1Y_1X_2Y_2\ar[rd]\ar[ld]	&\\
		(X_1Y_1)^{S}\ar@/_1pc/[rrr]^-{p^f}	&									&(X_2Y_2)^{S}\ar@/_1pc/[rrr]^-{p^f}	&X_1Y_1	&						&X_2Y_2
	}.
\]
\end{proof}
\begin{proposition}\label{pullback}
In the notations above, we have an adjoint pair
\[f^*:\widetilde{Sh}(T)\rightleftharpoons\widetilde{Sh}(S):f_*,\]
where \((f_*F)X=F\circ\varphi^f\) for \(F\in\widetilde{Sh}(S)\).
\end{proposition}
\begin{proof}
By Lemma \ref{adjunction} applying to \(\varphi^f\), we obtain an adjunction \(\widetilde{PSh}(T)\rightleftharpoons\widetilde{PSh}(S)\) and we apply the sheafication functor in Proposition \ref{sheafication} to get desired result.
\end{proof}
Obviously, we have \((f_1\circ f_2)^*=f_2^*\circ f_1^*\), \((f_1\circ f_2)_*=f_{1*}\circ f_{2*}\).
\begin{proposition}\label{sheaf formulas}
Suppose \(f:S\longrightarrow T\) is a morphism in \(Sm/k\).
\begin{enumerate}
\item For any \(Y\in Sm/T\),
\[f^*\widetilde{\mathbb{Z}}_T(Y)\cong\widetilde{\mathbb{Z}}_S(Y\times_TS)\]
as sheaves with \(E\)-transfers.
\item For any \(F\in \widetilde{Sh}(S)\) and \(Y\in Sm/T\),
\[(f_*F)^Y\cong f_*(F^{Y\times_TS})\]
as sheaves with \(E\)-transfers.
\item For any \(F\in\widetilde{Sh}(T)\) and \(G\in\widetilde{Sh}(S)\),
\[\underline{Hom}_T(F,f_*G)\cong f_*\underline{Hom}_S(f^*F,G)\]
as sheaves with \(E\)-transfers.
\item For any \(F,G\in\widetilde{Sh}(T)\), we have
\[f^*F\otimes_Sf^*G\cong f^*(F\otimes_TG)\]
as sheaves with \(E\)-transfers.
\end{enumerate}
\end{proposition}
\begin{proof}
\begin{enumerate}
\item We have
\[Hom_S(f^*\widetilde{\mathbb{Z}}_T(Y),-)\cong Hom_T(\widetilde{\mathbb{Z}}_T(Y),f_*-)\cong Hom_S(\widetilde{\mathbb{Z}}_S(Y\times_TS),-).\]
\item This is because for any \(Z\in Sm/T\)
\[(f_*F)^YZ=F((Y\times_TZ)\times_TS)\cong F((Z\times_TS)\times_S(Y\times_TS))\cong (f_*(F^{Y\times_TS}))Z\]
and Proposition \ref{ex*}.
\item This is because for any \(Y\in Sm/T\),
\begin{eqnarray*}
\underline{Hom}_T(F,f_*G)Y	&=		&Hom_T(F,(f_*G)^Y)\\
					&\cong	&Hom_T(F,f_*(G^{Y\times_TS}))\\
					&		&\textrm{by (2)}\\
					&\cong	&Hom_S(f^*F,G^{Y\times_TS})\\
					&=		&(f_*\underline{Hom}_S(f^*F,G))Y
\end{eqnarray*}
\item This is because for any \(H\in\widetilde{Sh}(S)\),
\begin{eqnarray*}
Hom_S(f^*F\otimes_Sf^*G,H)	&\cong	&Hom_S(f^*G,\underline{Hom}_S(f^*F,H))\\
					&\cong	&Hom_T(G,f_*\underline{Hom}_S(f^*F,H))\\
					&\cong	&Hom_T(G,\underline{Hom}_T(F,f_*H))\\
					&		&\textrm{by (3)}\\
					&\cong	&Hom_T(F\otimes_TG,f_*H)\\
					&\cong	&Hom_S(f^*(F\otimes_TG),H)
\end{eqnarray*}
\end{enumerate}
\end{proof}
From now on in this chapter, suppose \(f:S\longrightarrow T\) is a smooth morphism in \(Sm/k\).
\begin{definition}
Suppose \(X_1, X_2\in Sm/S\), we have a closed immersion \(q_f:X_1\times_SX_2\longrightarrow X_1\times_TX_2\). Define
\[\begin{array}{cccc}\varphi_f:&\widetilde{Cor}_S&\longrightarrow&\widetilde{Cor}_T\\&X&\longmapsto&X\\&g&\longmapsto&g_T\end{array},\]
where we see a smooth \(S\)-scheme as a smooth \(T\)-scheme via \(f\) and \(g\longmapsto g_T:\widetilde{Cor}_S(X_1,X_2)\longrightarrow\widetilde{Cor}_T(X_1,X_2)\) is the unique map such that the following diagram commutes
\[
	\xymatrix
	{
		E_Z^{d_{X_1}-d_S}(X_1\times_SX_2-T_{X_1\times_SX_2/X_1})\ar[r]^-{q_{f*}\circ t_f}\ar[d]	&E_{q_f(Z)}^{d_{X_1}-d_T}(X_1\times_TX_2,-T_{X_1\times_TX_2/X_1})\ar[d]\\
		\widetilde{Cor}_T(X_1,X_2)\ar[r]^-{\varphi_f}							&\widetilde{Cor}_S(X_1,X_2)
	},
\]
for any \(Z\in\mathscr{A}_S(X_1,X_2)\) (\(q_f(Z)\in\mathscr{A}_T(X_1,X_2)\) since \(q_f\) is a closed immersion), where \(t_f\) is the isomorphism
\begin{align*}
			&-T_{X_1\times_SX_2/X_1}\\
\longrightarrow	&N_{(X_1\times_SX_2)/(X_1\times_TX_2)}-N_{(X_1\times_SX_2)/(X_1\times_TX_2)}-T_{X_1\times_SX_2/X_1}\\
\longrightarrow	&N_{(X_1\times_SX_2)/(X_1\times_TX_2)}-q_f^*T_{X_1\times_TX_2/X_1}
\end{align*}
\end{definition}
For convenience, we may denote \(T_{X/Y}\) by \(T_f\) for smooth \(f:X\longrightarrow Y\) and denote \(N_{X/Y}\) by \(N_f\) for a closed immersion \(f\) in the following few propositions.
\begin{proposition}
Suppose \(\xymatrix{X_1\ar[r]^{g_1}&X_2\ar[r]^{g_2}&X_3}\) be morphisms in \(\widetilde{Cor}_S\), we have
\[(g_2\circ g_1)_T=g_{2T}\circ g_{1T}.\]
So \(\varphi_f:\widetilde{Cor}_S\longrightarrow\widetilde{Cor}_T\) is indeed a functor.
\end{proposition}
\begin{proof}
We have Cartesian squares
\[
	\xymatrix
	{
		X_1\times_SX_2\ar[r]^{i}						&X_1\times_TX_2\\
		X_1\times_SX_2\times_SX_3\ar[r]^{i'}\ar[u]^{q_{12}'}	&X_1\times_T(X_2\times_SX_3)\ar[u]^{q_{12}}
	},
\]
\[
	\xymatrix
	{
		X_1\times_T(X_2\times_SX_3)\ar[r]^q\ar[d]_{r}	&X_1\times_TX_2\times_TX_3\ar[d]_{p_{23}}\\
		X_2\times_SX_3\ar[r]^j					&X_1\times_TX_3
	},
\]
\[
	\xymatrix
	{
		X_1\times_SX_2\times_SX_3\ar[r]\ar[d]_{i'}	&(X_1\times_SX_2)\times_TX_3\ar[d]\\
		X_1\times_T(X_2\times_SX_3)\ar[r]^q		&X_1\times_TX_2\times_TX_3
	},
\]
\[
	\xymatrix
	{
		X_1\times_SX_2\times_SX_3\ar[d]_{i'}\ar[r]^{q_{13}'}	&X_1\times_SX_3\ar[d]_k\\
		X_1\times_T(X_2\times_SX_3)\ar[r]^{p_{13}\circ q}	&X_1\times_TX_3
	}.
\]
and commutative diagrams
\[
	\xymatrix
	{
		X_1\times_TX_2								&\\
		X_1\times_T(X_2\times_SX_3)\ar[r]^q\ar[u]^{q_{12}}	&X_1\times_TX_2\times_TX_3\ar[ul]_{p_{12}}
	},
\]
\[
	\xymatrix
	{
		X_1\times_SX_2\times_SX_3\ar[r]^{i'}\ar[rd]_{q_{23}'}	&X_1\times_T(X_2\times_SX_3)\ar[d]_r\\
												&X_2\times_SX_3
	}.
\]
Suppose \(g_1\) and \(g_2\) come from cohomologies with admissible supports, then we have
\begin{align*}
	&g_{2T}\circ g_{1T}\\
=	&j_*t_f(g_2)\circ i_*t_f(g_1)\\
	&\textrm{by definition}\\
=	&p_{13*}(p_{23}^*j_*t_f(g_2)\cdot p_{12}^*i_*t_f(g_1))\\
	&\textrm{by definition}\\
=	&p_{13*}(q_*r^*t_f(g_2)\cdot p_{12}^*i_*t_f(g_1))\\
	&\textrm{by Axiom \ref{BCCI} for the second square above}\\
=	&p_{13*}q_*(r^*t_f(g_2)\cdot q^*p_{12}^*i_*t_f(g_1))\\
	&\textrm{by Axiom \ref{PFCI} for \(q\)}\\
=	&p_{13*}q_*(r^*t_f(g_2)\cdot q_{12}^*i_*t_f(g_1))\\
	&\textrm{by Axiom \ref{FPB}}\\
=	&p_{13*}q_*(r^*t_f(g_2)\cdot i'_*q_{12}'^*t_f(g_1))\\
	&\textrm{by Axiom \ref{BCCI} for the first square above}\\
=	&p_{13*}q_*i'_*(i'^*r^*t_f(g_2)\cdot q_{12}'^*t_f(g_1))\\
	&\textrm{by Axiom \ref{PFCI} for \(i'\)}\\
=	&p_{13*}q_*i'_*((q_{23}'^*t_f(g_2)\cdot q_{12}'^*t_f(g_1),i'^*N_q-i'^*q^*p_{13}^*T_{X_1\times_TX_3/X_1}+N_{i'}-i'^*q^*T_{p_{13}}))\\
	&\textrm{by Axiom \ref{FPB}}\\
=	&(p_{13}\circ q)_*i'_*((q_{23}'^*t_f(g_2)\cdot q_{12}'^*t_f(g_1),-i'^*q^*p_{13}^*T_{X_1\times_TX_3/X_1}+N_{i'}-i'^*T_{p_{13}\circ q}))\\
	&\textrm{by Axiom \ref{CTPF}, (1) and functoriality of push-forwards with respect to twists}\\
=	&k_*q'_{13*}((q_{23}'^*t_f(g_2)\cdot q_{12}'^*t_f(g_1),-q_{13}'^*k^*T_{X_1\times_TX_3/X_1}+q_{13}'^*N_k-T_{q_{13}'}))\\
	&\textrm{by Axiom \ref{CTPF}, (3) for the last square above}.
\end{align*}

Now we have to be more careful. We say a morphism \(f:A+B\longrightarrow C+D\) in a Picard category contains a switch if there are morphisms \(g:A\longrightarrow D\) and \(h:B\longrightarrow C\) such that \(f=c(D,C)\circ(g+h)\). Conversely we say it doesn't contains a switch if there are morphisms \(g:A\longrightarrow C\) and \(h:B\longrightarrow D\) such that \(f=g+h\). There is a commutative diagram with \emph{three} squares being Cartesian
\[
	\xymatrix
	{
		X_1\times_SX_2\times_SX_3\ar[rr]^{q'}\ar[rd]^u\ar[dd]_{i'}	&							&(X_1\times_SX_2)\times_TX_3\ar[dd]_{i''}\\
													&(X_1\times_SX_3)\times_TX_2\ar[rd]^v	&\\
		X_1\times_T(X_2\times_SX_3)\ar[rr]^q					&							&X_1\times_TX_2\times_TX_3
	}.
\]
This induces a commutative diagram (all arrows contain a switch)
\[
	\xymatrix
	{
		N_{q'}+q'^*N_{i''}\ar[r]\ar[d]	&N_u+u^*N_v\ar[dl]\\
		N_{i'}+i'^*N_q				&
	}
\]
since they all come from exact sequences related to \(N_{i''\circ q'}=N_{v\circ u}=N_{q\circ i'}\). Then we have a commutative diagram (none arrow contains a switch except \(\varphi\))
\[
	\xymatrix
	{
		q'^*N_{i''}+i'^*N_q\ar[r]\ar[d]					&N_u+N_{q'}\ar[r]\ar[d]	&i'^*N_q+u^*N_v\ar[d]\\
		q'^*N_{i''}+N_{q'}\ar[r]\ar@/_2pc/[rr]^{\varphi}	&N_u+u^*N_v\ar[r]		&i'^*N_q+N_{i'}\\
	}
\]
by the diagram above. Hence the composition among morphisms without switch
\[q'^*N_{i''}+i'^*N_q\longrightarrow N_u+N_{q'}\longrightarrow i'^*N_q+u^*N_v\]
is equal to the morphism with a switch
\[q'^*N_{i''}+i'^*N_q\longrightarrow i'^*N_q+u^*N_v\]
where the morphism \(q'^*N_{i''}\longrightarrow u^*N_v\) is given by the composition
\[q'^*N_{i''}\longrightarrow N_{i'}\longrightarrow u^*N_v.\]
So the composition among morphisms without switch
\[q_{23}'^*N_j+q_{12}^*N_i\longrightarrow N_{q'}+N_u\longrightarrow q_{13}'^*N_k+i'^*N_q\]
is equal to the morphism with a switch
\[q_{23}'^*N_j+q_{12}^*N_i\longrightarrow q_{13}'^*N_k+i'^*N_q\]
where the morphism \(q_{12}^*N_i\longrightarrow q_{13}'^*N_k\) is given by the composition
\[q_{12}^*N_i\longrightarrow N_{i'}\longrightarrow q_{13}'^*N_k\]
and the morphism \(q_{23}'^*N_j\longrightarrow i'^*N_q\) is got by pulling back the morphism
\[r^*N_j\longrightarrow N_q\]
along \(i'\).

Moreover, there are commutative diagrams with squares being Cartesian
\[
	\xymatrix
	{
		X_1\times_SX_2\ar[r]^i								&X_1\times_TX_2\ar[r]						&X_1\\
		X_1\times_SX_2\times_SX_3\ar[u]^{q_{12}'}\ar[r]^u\ar[d]_{i'}	&(X_1\times_SX_3)\times_TX_2\ar[r]\ar[u]\ar[d]	&X_1\times_SX_3\ar[u]\ar[d]\\
		X_1\times_T(X_2\times_SX_3)\ar[r]^q						&X_1\times_TX_3\times_TX_2\ar[r]			&X_1\times_TX_3
	}
\]
\[
	\xymatrix
	{
		X_2\times_SX_3\ar[r]^j									&X_2\times_TX_3\ar[r]						&X_2\\
		X_1\times_SX_2\times_SX_3\ar[r]^{q'}\ar[u]^{q_{23}'}\ar[d]_{q_{13}'}	&(X_1\times_SX_2)\times_TX_3\ar[u]\ar[d]\ar[r]	&X_1\times_SX_2\ar[u]\ar[d]\\
		X_1\times_SX_3\ar[r]^k									&X_1\times_TX_3\ar[r]						&X_1
	},
\]
which induce commutative diagrams where none right-hand vertical map contains a switch
\[
	\xymatrix
	{
		i^*q_{12}'^*T_{X_1\times_TX_2/X_1}\ar[r]						&q_{12}'^*T_{X_1\times_SX_2/X_1}+q_{12}'^*N_i\\
		u^*T_{(X_1\times_SX_3)\times_TX_2/X_1\times_SX_3}\ar[r]\ar[u]\ar[d]	&T_{q_{13}'}+N_u\ar[u]\ar[d]\\
		q^*T_{p_{13}}\ar[r]										&i'^*T_{p_{13}\circ q}+i'^*N_q
	},
\]
\[
	\xymatrix
	{
		q_{23}'^*j^*T_{X_2\times_TX_3/X_2}\ar[r]						&q_{23}'^*T_{X_2\times_SX_3/X_2}+q_{23}'^*N_j\\
		q'^*T_{(X_1\times_SX_2)\times_TX_3/X_1\times_SX_2}\ar[r]\ar[u]\ar[d]	&T_{q_{12}'}+N_{q'}\ar[u]\ar[d]\\
		k^*T_{X_1\times_TX_3/X_1}\ar[r]								&q_{13}'^*T_{X_1\times_SX_3/X_1}+q_{13}'^*N_k
	}.
\]
From all these calculations above together with functoriality of \(q'_{13*}\) with respect to twists, we see that
\begin{align*}
	&k_*q'_{13*}((q_{23}'^*t_f(g_2)\cdot q_{12}'^*t_f(g_1),-q_{13}'^*k^*T_{X_1\times_TX_3/X_1}+q_{13}'^*N_k-T_{q_{13}'}))\\
=	&k_*t_f(q'_{13*}(q_{23}'^*(g_2)\cdot q_{12}'^*(g_1)))\\
=	&(g_2\circ g_1)_T\\
	&\textrm{by definition}.
\end{align*}

Finally, we have to show that \((id_X)_T=id_X\) for any \(X\in Sm/S\). We have the following commutative diagram
\[
	\xymatrix
	{
		X\ar[rd]_-{\triangle_T}\ar[r]^-{\triangle_S}	&X\times_SX\ar[r]\ar[d]_{q_f}	&X\\
											&X\times_TX\ar[ru]			&
	}
\]
where \(\triangle\) means diagonal map. We have to show that the following diagram commutes
\[
	\xymatrix
	{
		N_{\triangle_S}-N_{\triangle_S}\ar[r]	&N_{\triangle_S}-\triangle_S^*T_{X\times_SX/X}\ar[d]\ar[r]^-{\triangle_{S*}}				&-T_{X\times_SX/X}\ar[d]_{t_f}\\
		0\ar[u]\ar[d]					&N_{\triangle_S}+\triangle_S^*(N_{q_f}-q_f^*T_{X\times_TX/X})\ar[r]^-{\triangle_{S*}}\ar[d]	&N_{q_f}-q_f^*T_{X\times_TX/X}\ar[d]_{q_{f*}}\\
		N_{\triangle_T}-N_{\triangle_T}\ar[r]	&N_{\triangle_T}-\triangle_T^*T_{X\times_TX/X}\ar[r]^-{\triangle_{T*}}					&-T_{X\times_TX/X}
	}.
\]
The right two squares come from functoriality of push-forwards with respect to twists and Axiom \ref{FPFCI}. The left square comes from the following commutative diagram with exact rows
\[
	\xymatrix
	{
		0\ar[r]	&N_{\triangle_S}\ar[r]						&N_{\triangle_T}\ar[r]						&\triangle_S^*N_{q_f}\ar[r]			&0\\
		0\ar[r]	&\triangle_S^*T_{X\times_SX/X}\ar[u]^{\cong}\ar[r]	&\triangle_T^*T_{X\times_TX/X}\ar[u]^{\cong}\ar[r]	&\triangle_S^*N_{q_f}\ar[u]^{id}\ar[r]	&0
	}.
\]
\end{proof}
By using Axiom \ref{FPFCI}, it is straightforward to check that \(\varphi_{f_1\circ f_2}=\varphi_{f_1}\circ\varphi_{f_2}\).
\begin{proposition}\label{projection formula for correspondences}
Suppose \(a\in\widetilde{Cor}_S(X_1,X_2)\) and \(b\in\widetilde{Cor}_T(Y_1,Y_2)\). Identifying \((X_1\times_SX_2)\times_TY_1\times_TY_2\) with \(X_1\times_SX_2\times_S(Y_1\times_TY_2)^S\) and \(X_1\times_TY_1\times_TX_2\times_TY_2\) with \((X_1\times_SY_1^S)\times_T(X_2\times_SY_2^S)\), we have
\[a_T\times_Tb=(a\times_Sb^S)_T.\]
\end{proposition}
\begin{proof}
We have a commutative diagram with the square being Cartesian
\[
	\xymatrix
	{
		(X_1\times_SX_2)\times_TY_1\times_TY_2\ar[r]^r\ar[d]_{p_1}\ar[rd]_{p_2}	&X_1\times_TY_1\times_TX_2\times_TY_2\ar[d]_{q_1}\ar[rd]_{q_2}	&\\
		X_1\times_SX_2\ar[r]_t										&X_1\times_TX_2										&Y_1\times_TY_2
	}.
\]
Suppose \(a, b\) come from cohomologies with supports. Denote the isomorphism
\[-q_1^*T_{X_1\times_TX_2/X_1}-q_2^*T_{Y_1\times_TY_2/Y_1}\longrightarrow-T_{X_1\times_TX_2\times_TY_1\times_TY_2/X_1\times_TY_1}\]
by \(\theta\) and the isomorphism
\[-p_1^*T_{X_1\times_SX_2/X_1}-p_2^*T_{Y_1\times_TY_2/Y_1}\longrightarrow-T_{X_1\times_SX_2\times_SY_1^S\times_SY_2^S/X_1\times_SY_1^S}\]
by \(\eta\). Then
\begin{align*}
	&a_T\times_Tb\\
=	&\theta(q_1^*t_*(t_f(a))\cdot q_2^*b)\\
	&\textrm{by definition}\\
=	&\theta(r_*p_1^*(t_f(a))\cdot q_2^*b)\\
	&\textrm{by Axiom \ref{BCCI} for the square in the diagram}\\
=	&\theta(r_*(p_1^*(t_f(a))\cdot p_2^*b))\\
	&\textrm{by Proposition \ref{PFCI} for \(r\) and Axiom \ref{CPB}}\\
=	&r_*r^*(\theta)((p_1^*(t_f(a))\cdot p_2^*b,p_1^*N_t-p_1^*t^*T_{X_1\times_TX_2/X_1}-p_2^*T_{Y_1\times_TY_2/Y_1}))\\
	&\textrm{by functoriality of push-forwards with respect to twists}\\
=	&r_*(t_f(\eta(p_1^*(a)\cdot p_2^*b)))\\
=	&(a\times_Sb^S)_T\\
	&\textrm{by definition.}
\end{align*}
Here the fifth equality comes from the following commutative diagram with exact rows and columns
\[
	\xymatrix
	{
				&0\ar[d]								&0\ar[d]							&					&\\
				&p_2^*T_{Y_1\times_TY_2/Y_1}\ar[d]\ar[r]^=		&p_2^*T_{Y_1\times_TY_2/Y_1}\ar[d]		&					&\\
		0\ar[r]	&r^*T_{X_1Y_1X_2Y_2/X_1Y_1}\ar[r]\ar[d]		&T_{X_1X_2Y_1^SY_2^S/X_1Y_1^S}\ar[d]\ar[r]	&N_r\ar[r]\ar[d]_{\cong}	&0\\
		0\ar[r]	&r^*q_1^*T_{X_1\times_TX_2/X_1}\ar[r]\ar[d]	&p_1^*T_{X_1\times_SX_2/X_1}\ar[r]\ar[d]	&p_1^*N_t\ar[r]			&0\\
				&0									&0								&					&
	}
\]
and Theorem \ref{four diagrams}, (1).
\end{proof}
By the same proof as in Proposition \ref{pullback} applying to \(\varphi_f\), we have the following:
\begin{proposition}\label{lowerhash}
There is an adjoint pair
\[f_{\#}:\widetilde{Sh}(S)\rightleftharpoons\widetilde{Sh}(T):(f_{\#})',\]
where \((f_{\#})'F=F\circ\varphi_f\) for \(F\in\widetilde{Sh}(T)\).
\end{proposition}
The next lemma is important when identifying \((f_{\#})'\). See also \cite[Exercise 1.12]{MVW}.
\begin{lemma}
For any \(U\in Sm/S\), \(X\in Sm/T\), we have an adjoint pair:
\[\widetilde{Cor}_S(U,X^S)=\widetilde{Cor}_T(U,X).\]
\end{lemma}
\begin{proof}
For any \(U\in Sm/S\), \(X\in Sm/T\), we have an isomorphism
\[\theta_{U,X}:U\times_SX^S\longrightarrow U\times_TX.\]
Then define
\[\begin{array}{cccc}\lambda_{U,X}:&\widetilde{Cor}_S(U,X^S)&\longrightarrow&\widetilde{Cor}_T(U,X)\\&W&\longmapsto&\theta_{U,X*}(W)\end{array},\]
which is obviously an isomorphism.

Now suppose we have \(U\in Sm/S\), \(X_1, X_2\in Sm/T\), \(V\in\widetilde{Cor}_T(X_1,X_2)\), \(W\in\widetilde{Cor}_S(U,X_1^S)\), we want to show that
\[\lambda_{U,X_2}(V^S\circ W)=V\circ\lambda_{U,X_1}(W).\]

We have a commutative diagram
\[
	\xymatrix
	{
		X_1\times_TX_2																										&									&U\times_TX_1\\
		X_1^S\times_SX_2^S\ar[u]_{p}																							&U\times_SX_1^S\ar[ru]^{\theta_{U,X_1}}_{\cong}	&\\
		U\times_S(X_1^S\times_SX_2^S)\ar[u]_{p_{23}}\ar[r]^-{p_{13}}\ar[ru]^{p_{12}}\ar@/^2pc/[uu]^{q_{23}}\ar@/^2pc/[rruu]^{q_{12}}\ar@/_2pc/[rr]^{q_{13}}	&U\times_SX_2^S\ar[r]^{\theta_{U,X_2}}_{\cong}	&U\times_TX_2\\
	},
\]
where we've identified \(U\times_S(X_1^S\times_SX_2^S)\) with \(U\times_TX_1\times_TX_2\) for convenience.

Suppose \(V\) and \(W\) come from some elements of Chow-Witt groups with support (see Definition \ref{cor}). We have
\begin{eqnarray*}
\lambda_{U,X_2}(V^S\circ W)	&=	&\theta_{U,X_2*}p_{13*}(p_{23}^*p^*V\cdot p_{12}^*W)\\
						&	&\textrm{by definition}\\
						&=	&q_{13*}(q_{23}^*V\cdot p_{12}^*W)\\
						&	&\textrm{by Axiom \ref{FPFSM}}\\
						&=	&q_{13*}(q_{23}^*V\cdot q_{12}^*\theta_{U,X_1*}W)\\
						&	&\textrm{by Axiom \ref{EE} and Axiom \ref{FPB}}\\
						&=	&V\circ\lambda_{U,X_1}(W)\\
						&	&\textrm{by definition.}
\end{eqnarray*}

Then suppose we have \(U_1, U_2\in Sm/S\), \(X\in Sm/T\), \(V\in\widetilde{Cor}_S(U_1,U_2)\), \(W\in\widetilde{Cor}_S(U_2,X^S)\), we want to show that
\[\lambda_{U_1,X}(W\circ V)=\lambda_{U_2,X}(W)\circ V_T.\]
We have a commutative diagram
\[
	\xymatrix
	{
		U_1\times_TU_2			&U_1\times_T(U_2\times_SX^S)\ar[l]_-d\ar[r]^-{b}\ar@/^2pc/[rr]^{q_{23}}\ar[drr]_{q_{13}}	&U_2\times_SX^S\ar[r]^{\theta_{U_2,X}}	&U_2\times_TX\\
		U_1\times_SU_2\ar[u]_{q_f}	&U_1\times_S(U_2\times_SX^S)\ar[l]_-{p_{12}}\ar[r]_-{p_{13}}\ar[ru]_{p_{23}}\ar[u]_a		&U_1\times_SX^S\ar[r]_{\theta_{U_1,X}}	&U_1\times_TX
	},
\]
where we have identified \(U_1\times_T(U_2\times_SX^S)\) with \(U_1\times_TU_2\times_TX\). Suppose \(V\) and \(W\) come from some elements of Chow-Witt groups with support like above and define \(\theta\) to be the isomorphism
\[-T_{U_1\times_SU_2\times_SX^S/U_1\times_TX}\longrightarrow N_a-a^*T_{U_1\times_T(U_2\times_SX^S)/U_1\times_TX}.\]
We have
\begin{align*}
	&\lambda_{U_1,X}(W\circ V)\\
=	&\theta_{U_1,X*}p_{13*}((p_{23}^*W\cdot p_{12}^*V,-T_{U_1\times_SU_2\times_SX^S/U_2\times_SX^S}-p_{12}^*T_{U_1\times_SU_2/U_1}))\\
	&\textrm{by definition}\\
=	&q_{13*}a_*((a^*q_{23}^*\theta_{U_2,X*}W\cdot\theta(p_{12}^*V),-a^*q_{23}^*T_{U_2\times_TX/U_2}+N_a-a^*T_{U_1\times_T(U_2\times_SX^S)/U_1\times_TX}))\\
	&\textrm{by Axiom \ref{EE} and Axiom \ref{CTPF}, (1)}\\
=	&q_{13*}((q_{23}^*\theta_{U_2,X*}W\cdot a_*\theta(p_{12}^*V),-q_{23}^*T_{U_2\times_TX/U_2}-T_{U_1\times_T(U_2\times_SX^S)/U_1\times_TX}))\\
	&\textrm{by Axiom \ref{PFCI} for \(a\)}\\
=	&q_{13*}((q_{23}^*\theta_{U_2,X*}W\cdot d^*q_{f*}\varphi_f(V),-q_{23}^*T_{U_2\times_TX/U_2}-d^*T_{U_1\times_TU_2/U_1}))\\
	&\textrm{by Axiom \ref{BCCI} for the leftmost square in the diagram above}\\
=	&\lambda_{U_2,X}(W)\circ V_T\\
	&\textrm{by definition}
\end{align*}
\end{proof}
\begin{proposition}
\[(f_{\#})'=f^*.\]
\end{proposition}
\begin{proof}
This is because for any \(Y\in Sm/S\), \(\widetilde{\gamma}(id_Y)\in\widetilde{Cor}_T(Y,Y)=\widetilde{Cor}_S(Y,Y^S)\) is the initial element of \(C_Y\) in Lemma \ref{adjunction} by the lemma above applying to \(\varphi^f\) (see Definition \ref{pullbackdef}). So for any \(F\in\widetilde{PSh}(T)\), we have \((f^*F)Y=FY=((f_{\#})'F)Y\). And this gives an isomorphism between \(f^*F\) and \((f_{\#})'F\) for any presheaf with \(E\)-transfers \(F\) by the lemma above. So it also gives an isomorphism after sheafication.
\end{proof}
\begin{proposition}\label{operations}
Suppose \(f:S\longrightarrow T\) is smooth as before.
\begin{enumerate}
\item For any \(X\in Sm/S\), we have
\[f_{\#}\widetilde{\mathbb{Z}}_S(X)\cong\widetilde{\mathbb{Z}}_T(X).\]
\item For any \(F\in\widetilde{Sh}(T)\) and \(Y\in Sm/T\)
\[f^*(F^Y)\cong (f^*F)^{Y\times_TS}\]
as sheaves with \(E\)-transfers.
\item For any \(F\in\widetilde{Sh}(S)\) and \(G\in\widetilde{Sh}(T)\)
\[\underline{Hom}_T(f_{\#}F,H)\cong f_*\underline{Hom}_S(F,f^*H)\]
as sheaves with \(E\)-transfers.
\item For any \(F\in\widetilde{Sh}(S)\) and \(G\in\widetilde{Sh}(T)\)
\[f_{\#}(F\otimes_Sf^*G)\cong f_{\#}F\otimes_TG\]
as sheaves with \(E\)-transfers.
\end{enumerate}
\end{proposition}
\begin{proof}
\begin{enumerate}
\item This is because for any \(F\in\widetilde{Sh}(T)\)
\[Hom_T(f_{\#}\widetilde{\mathbb{Z}}_S(X),F)\cong Hom_S(\widetilde{\mathbb{Z}}_S(X),f^*F)\cong (f^*F)(X)\cong F(X)\]
by the proposition before.
\item This is because for every \(X\in Sm/S\)
\[(f^*(H^Y))X=H(Y\times_TX)\]
and
\[(f^*H)^{Y\times_TS}X=H((Y\times_TS)\times_SX)\]
by the proposition before. Then one uses Proposition \ref{projection formula for correspondences}.
\item For any \(Y\in Sm/T\)
\begin{eqnarray*}
\underline{Hom}_T(f_{\#}F,H)Y	&=		&Hom_T(f_{\#}F,H^Y)\\
						&\cong	&Hom_S(F,f^*(H^Y))\\
						&\cong	&Hom_S(F,(f^*H)^{Y\times_TS})\\
						&		&\textrm{by (2)}\\
						&=		&\underline{Hom}_S(F,f^*H)(Y\times_TS)\\
						&=		&(f_*\underline{Hom}_S(F,f^*H))Y.
\end{eqnarray*}
\item This is because for any \(H\in\widetilde{Sh}(T)\), we have
\begin{eqnarray*}
Hom_T(f_{\#}(F\otimes_Sf^*G),H)	&\cong	&Hom_S(F\otimes_Sf^*G,f^*H)\\
						&\cong	&Hom_S(f^*G,\underline{Hom}_S(F,f^*H))\\
						&\cong	&Hom_T(G,f_*\underline{Hom}_S(F,f^*H))\\
						&\cong	&Hom_T(G,\underline{Hom}_T(f_{\#}F,H))\\
						&		&\textrm{by (3)}\\
						&\cong	&Hom_T(f_{\#}F\otimes_TG,H).
\end{eqnarray*}
\end{enumerate}
\end{proof}
\section{Operations on Localized Categories}\label{Operations on Localized Categories}
We are going to establish the theory of effective (resp. stabilized) motives on bounded above complexes (See \cite{MVW}) of sheaves of \(E\)-transfers (resp. symmetric spectra). And we will compare our theory with respect to that of \cite{CD}, \cite{CD1} and \cite{DF}, which uses unbounded complexes.
\subsection{For Sheaves with \(E\)-Transfers}\label{effective}
\subsubsection{On Derived Categories}
Denote by \(D^-(S)\) (resp. \(K^-(S)\)) the derived (resp. homotopy) category of bounded above complexes in \(\widetilde{Sh}(S)\). We are going to define \(\otimes_S\) and \(f_{\#}\) and \(f^*\) (See Section \ref{Sheaves}) over those categories. The method is inherited from \cite[Corollary 2.2]{SV} and \cite[Lemma 8.15]{MVW}.
\begin{definition}
We call a presheaf with \(E\)-transfers free if it's a direct sum of presheaves of the form \(\widetilde{c}_S(X)\). We call a presheaf with \(E\)-transfers projective if it's a direct summand of a free presheaf with \(E\)-transfers. An sheaf with \(E\)-transfers is called free (resp. projective) if it's a sheafication of a free (resp. projective) presheaf with \(E\)-transfers. A bounded above complex of sheaves with \(E\)-transfers is called free (projective) if all its terms are free (projective).
\end{definition}

\begin{definition}
A projective resolution of a bounded above sheaf complex \(K\) is a projective complex with a quasi-isomorphism \(P\longrightarrow K\). If \(K\) is already projective, we take \(P=K\).
\end{definition}
Now let \(Y\in Sm/k\) be an \(S\)-scheme and \(Y\in Sm/T\). Consider in this section the functors
\[\begin{array}{cccc}\varphi:&\widetilde{Cor}_S&\longrightarrow&\widetilde{Cor}_T\\&X&\longmapsto&(X^Y)_T\cong X\times_SY\end{array}\]
and
\[\begin{array}{cccc}\psi:&Sm_S&\longrightarrow&Sm_T\\&X&\longmapsto&X\times_SY\end{array}\]
determined by the triple \((Y, S, T)\). We have a commutative diagram
\[
	\xymatrix
	{
		Sm/S\ar[r]^{\psi}\ar[d]_{\widetilde{\gamma}}	&Sm/T\ar[d]_{\widetilde{\gamma}}\\
		\widetilde{Cor}_S\ar[r]^{\varphi}			&\widetilde{Cor}_T
	}.
\]

Recall from Lemma \ref{adjunction} the definition of \(\varphi^*\) and \(\varphi_*\).
\begin{lemma}
For any \(X\in Sm/S\),
\[\varphi^*(\widetilde{c}_S(X))\cong \widetilde{c}_S(\psi X)\]
as presheaves with \(E\)-transfers.
\end{lemma}
\begin{proof}
For any \(F\in\widetilde{PSh}(T)\),
\[Hom_T(\varphi^*(\widetilde{c}_S(X)),F)\cong Hom_S(\widetilde{c}_S(X),\varphi_*F)\cong F(\psi X).\]
\end{proof}
\begin{lemma}
The functor \(\varphi_*\) maps sheaves with \(E\)-transfers to sheaves with \(E\)-transfers.
\end{lemma}
\begin{proof}
It suffices to show that for any finite Nisnevich covering \(\{U_i\}\) of \(X\in Sm/S\), the following sequence is exact
\[0\longrightarrow G(X)\longrightarrow\oplus_iG(U_i)\longrightarrow\oplus_{i,j}G(U_i\times_XU_j)\]
where \(G=\varphi_*F\) for some \(F\in\widetilde{Sh}(T)\). And this follows easily.
\end{proof}
\begin{lemma}
Let \(f:F\longrightarrow G\) be morphism in \(\widetilde{PSh}(S)\) such that \(\widetilde{a}(f)\) is an isomorphism, then \(\widetilde{a}(\varphi^*(f))\) is also an isomorphism.
\end{lemma}
\begin{proof}
This follows from a similar method as in Proposition \ref{sheafication1}.
\end{proof}
\begin{proposition}
For any \(F\in\widetilde{PSh}(S)\),
\[\widetilde{a}((L_i\varphi^*)\widetilde{a}(F))\cong\widetilde{a}((L_i\varphi^*)F)\]
as sheaves with \(E\)-transfers for any \(i\geq 0\), where \(L_i\varphi^*\) means the \(i^{th}\) left derived functor of \(\varphi^*\).
\end{proposition}
\begin{proof}
We show at first that for any presheaf with \(E\)-transfers \(F\) with \(\widetilde{a}(F)=0\)
\[\widetilde{a}(L_i\varphi^*(F))=0\]
for any \(i\geq 0\). Then for any presheaf with \(E\)-transfers \(F\), denote by \(\theta\) the natural map \(F\longrightarrow\widetilde{a}(F)\). We have
\[\widetilde{a}(coker(\theta))=\widetilde{a}(ker(\theta))=0.\]
Hence for any \(i\geq 0\), we have
\[\widetilde{a}(L_i\varphi^*\widetilde{a}(F))\cong\widetilde{a}(L_i\varphi^*Im(\theta))\cong\widetilde{a}(L_i\varphi^*F)\]
by using long exact sequences. Hence the statement follows.

Now we prove the claim. We do induction on \(i\). The claim is true for \(i=0\) and suppose it's true for \(i<n\).

For any \(F\in\widetilde{PSh}(S)\), we have a surjection
\[\oplus_{a\in F(X)}\widetilde{c}_S(X)\longrightarrow F\]
defined by each section of \(F\) on each \(X\in Sm/S\). Since \(\widetilde{a}(F)=0\), for any \(a\in F(X), X\in Sm/S\), there is a finite Nisnevich covering \(U_a\longrightarrow X\) of \(X\) such that \(a|_{U_a}=0\). So the composition
\[\oplus_{a\in F(X)}\widetilde{c}_S(U_a)\longrightarrow\oplus_{a\in F(X)}\widetilde{c}_S(X)\longrightarrow F\]
is zero. Then we have got a surjection
\[\oplus_{a\in F(X)}H_0(\breve{C}(U_a/X))\longrightarrow F\]
with kernel \(K\). Proposition \ref{Cech} implies that
\[\widetilde{a}(H_p(\breve{C}(U/X)))=0\]
for any Nisnevich covering \(U\longrightarrow X\) and \(p\in\mathbb{Z}\). So \(\widetilde{a}(K)=0\) also. We have a hypercohomology spectral sequence
\[(L_p\varphi^*)H_q(\breve{C}(U/X))\Longrightarrow(\mathbb{L}_{p+q}\varphi^*)\breve{C}(U/X).\]
Hence
\[\widetilde{a}((\mathbb{L}_n\varphi^*)\breve{C}(U/X))\cong\widetilde{a}((L_n\varphi^*)H_0(\breve{C}(U/X)))\]
by induction hypothesis. But
\[\widetilde{a}((\mathbb{L}_n\varphi^*)\breve{C}(U/X))\cong\widetilde{a}(H_n(\varphi^*\breve{C}(U/X)))\]
by definition of hypercohomology and the latter one vanishes since we have
\[\varphi^*\breve{C}(U/X)=\breve{C}(\psi U/\psi X)\]
by previous lemmas. So
\[\widetilde{a}((L_n\varphi^*)H_0(\breve{C}(U/X)))=0.\]
So
\[\widetilde{a}(L_n\varphi^*F)\cong\widetilde{a}(L_{n-1}\varphi^*K)=0\]
by long exact sequence and induction hypothesis.
\end{proof}
\begin{proposition}\label{adapted}
Let functor \(\varphi^*\) takes acyclic projective complexes to acyclic projective complexes.
\end{proposition}
\begin{proof}
For any projective \(F\in\widetilde{Sh}(S)\), \(F=\widetilde{a}(G)\) for some projective \(G\in\widetilde{PSh}(S)\) by definition. So
\[\widetilde{a}((L_i\varphi^*)F)\cong\widetilde{a}((L_i\varphi^*)G)=0\]
for any \(i>0\) by the proposition above. Now given a short exact sequence of sheaves with \(E\)-transfers
\[0\longrightarrow K\longrightarrow F\longrightarrow P\longrightarrow 0\]
with \(\widetilde{a}((L_i\varphi^*)P)=0\) for any \(i>0\). Then the sequence is still exact as sheaves with \(E\)-transfers after applying \(\varphi^*\) by long exact sequence. Then the statement follows easily.
\end{proof}
\begin{proposition}\label{derived}
We have an exact functor
\[L\varphi^*:D^-(S)\longrightarrow D^-(T)\]
which maps any \(K\in D^-(S)\) to \(\varphi^*P\), where \(P\) is a projective resolution \(K\).
\end{proposition}
\begin{proof}
By the proposition above, the class of projective complexes is adapted (see \cite[III.6.3]{GM}) to the functor \(\varphi^*\). Now apply \cite[III.6.6]{GM}.
\end{proof}

We will just write \(L\varphi^*\) above by \(\varphi^*\) for convenience.

Now we apply the general results above to \(\otimes_S\), \(f_{\#}\) and \(f^*\).
\begin{proposition}\label{derived functors}
\begin{enumerate}
\item There is a tensor product
\[\begin{array}{cccccc}\otimes_S:&D^-(S)&\times&D^-(S)&\longrightarrow&D^-(S)\\&(K&,&L)&\longmapsto&P\otimes_SQ\end{array},\]
where \(P, Q\) are projective resolutions of \(K, L\), respectively and the last tensor means taking the total complex of the bicomplex \(\{P_i\otimes_SQ_j\}\). And for any \(K\in D^-(S)\), the functor \(K\otimes_S-\) is exact.
\item Suppose \(f:S\longrightarrow T\) is a smooth morphism in \(Sm/k\). There is an exact functor
\[\begin{array}{cccc}f_{\#}:&D^-(S)&\longrightarrow&D^-(T)\\&K&\longmapsto&f_{\#}P\end{array},\]
where \(P\) is a projective resolution of \(K\).
\item Suppose \(f:S\longrightarrow T\) is a morphism in \(Sm/k\). There is an exact functor
\[\begin{array}{cccc}f^*:&D^-(T)&\longrightarrow&D^-(S)\\&K&\longmapsto&f^*P\end{array},\]
where \(P\) is a projective resolution of \(K\).
\end{enumerate}
\end{proposition}
\begin{proof}
\begin{enumerate}
\item Suppose \(Y\in Sm/S\). In the definition of \(\varphi\), we take \((Y, S, T):=(Y, S, S)\). Then \(\varphi^*F\cong F\otimes_S\widetilde{\mathbb{Z}}_S(Y)\) for any \(F\in\widetilde{Sh}(S)\) by Proposition \ref{hom-tensor}.

Now given acyclic projective complex \(P\) and a projective sheaf \(F\). \(F\otimes_SP\) is also acyclic by applying Proposition \ref{adapted} to \(\varphi\) and definition of projectiveness. So for any projective complex \(K\), the complex \(P\otimes_SK\) is also acyclic by the spectral sequence of the bicomplex \(\{P_i\otimes_SK_j\}\). Then for any projective complexes \(P, Q, R\) and quasi-isomorphism \(a:P\longrightarrow Q\), the morphism \(a\otimes_SR\) is still a quasi-isomorphism since we have
\[Cone(a\otimes_SR)\cong Cone(a)\otimes_SR\]
and the latter one is acyclic. So the statement follows easily.
\item In the definition of \(\varphi\), we take \((Y, S, T):=(S, S, T)\) and apply Proposition \ref{derived}.
\item In the definition of \(\varphi\), we take \((Y, S, T):=(T, S, T)\) and apply Proposition \ref{derived}.
\end{enumerate}
\end{proof}
\begin{proposition}\label{derived adjunction}
Suppose \(f:S\longrightarrow T\) is a smooth morphism in \(Sm/k\), we have an adjoint pair
\[f_{\#}:D^-(S)\rightleftharpoons D^-(T):f^*.\]
\end{proposition}
\begin{proof}
By Proposition \ref{lowerhash}, it is easy to see that there is an adjunction
\[f_{\#}:K^-(S)\rightleftharpoons K^-(T):f^*.\]

Since \(f^*:\widetilde{Sh}(T)\longrightarrow\widetilde{Sh}(S)\) has both a left adjoint and a right adjoint, it's an exact functor. So \(Lf^*\cong f^*\) in this case. Suppose \(K\in D^-(S)\), \(L\in D^-(T)\) and \(p:P\longrightarrow K\) be a projective resolution of \(K\). Hence \(f_{\#}K=f_{\#}P\) by definition.

We construct a morphism
\[\theta:Hom_{D^-(S)}(f_{\#}K,L)\longrightarrow Hom_{D^-(T)}(K,f^*L)\]
by the following: Suppose \(s\in Hom_{D^-(S)}(f_{\#}K,L)\) is written as a right roof (see \cite[III.2.9]{GM})
\[
	\xymatrix
	{
						&R 	&\\
		f_{\#}P\ar[ru]_a	&	&L\ar[lu]^b
	}.
\]
By adjunction, \(a\) induces a morphism \(a':P\longrightarrow f^*R\). Then we define \(\theta(s)\) to the composition of the right roof
\[
	\xymatrix
	{
					&f^*R	&\\
		P\ar[ru]_{a'}	&	&f^*L\ar[lu]^{f^*b}
	}
\]
with \(p^{-1}\). This definition is well-defined since \(f^*\) is exact.

Then, we construct another morphism
\[\xi:Hom_{D^-(T)}(K,f^*L)\longrightarrow Hom_{D^-(S)}(f_{\#}K,L)\]
by the following: Suppose \(t\in Hom_{D^-(T)}(K,f^*L)\) and \(t\circ p\) is written as a left roof (see \cite[III.2.8]{GM})
\[
	\xymatrix
	{
			&R\ar[ld]_a\ar[rd]^b 	&\\
		P 	&				&f^*L
	}
\]
where \(R\) is also projective. By adjunction, \(b\) induces a morphism \(b':f_{\#}R\longrightarrow L\). Then we define \(\xi(t)\) to be the left roof
\[
	\xymatrix
	{
				&f_{\#}R\ar[ld]_{f_{\#}a}\ar[rd]^{b'} 	&\\
		f_{\#}P 	&								&L
	}.
\]
This definition is well-defined by Proposition \ref{adapted} applied to \(f_{\#}\).

Finally one checks that \(\theta\) and \(\xi\) are inverse to each other and the statement follows.
\end{proof}

In \cite[Theorem 1.7]{CD}, they defined a model structure \(\mathfrak{M}\) on the category of unbounded complexes of sheaves with \(E\)-transfers over \(S\). This is a cofibrantly generated model structure where the cofibrations are those \(I\)-cofibrations (See \cite[Definition 2.1.7]{Hov}) where \(I\) consists of the morphisms \(S^{n+1}\widetilde{\mathbb{Z}}_S(X)\longrightarrow D^n\widetilde{\mathbb{Z}}_S(X)\) for any \(X\in Sm/S\) (See \cite[1.9]{CD} for notations) and weak equivalences are quasi-morphisms between complexes.
\begin{proposition}\label{cofibrant}
Bounded above projective complexes are cofibrant objects in \(\mathfrak{M}\).
\end{proposition}
\begin{proof}
Suppose \(P\) is a bounded above projective complex and we have an \(I\)-injective (See \cite[Definition 2.1.7]{Hov}) morphism \(f:A\longrightarrow B\) between unbounded complexes with a morphism \(g:P\longrightarrow B\). We have to show that \(g=f\circ h\) for some \(h:P\longrightarrow A\).

Now we construct \(h\) by induction. Suppose for any \(m\geq n\), we have constructed a morphism \(h^m:P^m\longrightarrow A^m\) such that \(g^m=f^m\circ h^m\) and \(d^A\circ h^m=h^m\circ d^P\). Since \(P\) is bounded above, this could be done when \(n\) large enough. We are going to construct an \(h^{n-1}:P^{n-1}\longrightarrow A^{n-1}\) satisfying the same property, that is, making the following diagram commute
\[
	\xymatrix
	{
		A^{n-1}\ar[rr]^{d^A}\ar[dd]_{f^{n-1}}	&												&A^{n}\ar[dd]_<<<<<{f^n}	&\\
									&P^{n-1}\ar[lu]_{h^{n-1}}\ar[ld]^{g^{n-1}}\ar[rr]^<<<<<<{d^P}	&					&P^n\ar[lu]_{h^n}\ar[ld]^{g^n}\\
		B^{n-1}\ar[rr]^{d^B}				&												&B^{n}				&
	}.
\]
We have a splitting surjection \(F\longrightarrow P^{n-1}\) where \(F\) is a free sheaf with \(E\)-transfers. So we may assume \(P^{n-1}\) is free and it's equal to \(\oplus_i\widetilde{\mathbb{Z}}_S(X_i)\) where \(X_i\in Sm/S\). For every \(i\), we have two morphisms:
\[u:\widetilde{\mathbb{Z}}_S(X_i)\longrightarrow P^{n-1}\longrightarrow B^{n-1}\longrightarrow B^{n}\]
and
\[v:\widetilde{\mathbb{Z}}_S(X_i)\longrightarrow P^{n-1}\longrightarrow P^n\longrightarrow A^n.\]
And this two morphisms gives a commutative square with a lifting since \(f\) is \(I\)-injective:
\[
	\xymatrix
	{
		S^n\widetilde{\mathbb{Z}}_S(X_i)\ar[r]^-v\ar[d]			&A\ar[d]_f\\
		D^{n-1}\widetilde{\mathbb{Z}}_S(X_i)\ar[r]^-u\ar[ru]^{w_i}	&B
	}.
\]
And one checks that \(\oplus_iw_i:P^{n-1}\longrightarrow A^{n-1}\) is just what we want.
\end{proof}

Moreover, \(\mathfrak{M}\) is stable and left proper so it induces a triangulated structure \(\mathfrak{T}'\) on \(D(S)\) (See \cite[Theoreme 4.1.49]{A}). The classical triangulated structure of \(D(S)\) or \(D^-(S)\) is denoted by \(\mathfrak{T}\).
\begin{proposition}\label{triangulated}
The natural functor
\[i:(D^-(S),\mathfrak{T})\longrightarrow(D(S),\mathfrak{T}')\]
is fully faithful exact.
\end{proposition}
\begin{proof}
Any distinguished triangle \(T\) in \((D^-(S),\mathfrak{T})\) is isomorphic in \(D^-(S)\) to a distinguished triangle in \(\mathfrak{T}\) like
\[\xymatrix{A\ar[r]^f&B\ar[r]&Cone(f)\ar[r]&A[1]},\]
where all arrows come from explicit morphisms between chain complexes (See \cite[III.3.3 and III.3.4]{GM}). By \cite[Proposition 8.1.23]{Hir}, we could find a commutative diagram
\[
	\xymatrix
	{
		A'\ar[r]^g\ar[d]_a	&B'\ar[d]_b\\
		A\ar[r]^f			&B
	}
\]
such that \((A',a)\) (resp. \((B',b)\)) is a fibrant cofibrant approximation of \(A\) (resp. \(B\)) and \(g\) is a cofibration in \(\mathfrak{M}\). So the triangle \(T\) is isomorphic in \(D(S)\) to the distinguished triangle
\[\xymatrix{A'\ar[r]^g&B'\ar[r]&Cone(g)\ar[r]&A'[1]}\]
in \(\mathfrak{T}\). Hence \(T\) is isomorphic to the triangle above in \(D(S)\). By \cite[Lemma 1.10]{CD} and \cite[Theoreme 4.1.38]{A}, the shift functors \(-[n]\) and \(-[n]'\) in \(\mathfrak{T}\) and \(\mathfrak{T}'\), respectively, coincide on cofibrant objects in \(\mathfrak{M}\). So we have a natural isomorphism \(\eta:-[n]\longrightarrow-[n]'\) where \(\eta_K=id_{K[n]}\) if \(K\) is cofibrant in \(\mathfrak{T}'\). Hence the triangle above is distinguished in \(\mathfrak{T}'\) by \cite[Definition 4.1.45]{A}. So the functor \(i\) is exact. And it's clearly fully faithful.
\end{proof}
By in \cite[Theorem 1.18 and Proposition 2.3]{CD}, we could define \(\otimes_S\), \(f^*\) and \(f_{\#}\) on \(D(S)\).
\begin{proposition}\label{-}
\begin{enumerate}
\item We have a commutative diagram (up to a natural isomorphism)
\[
	\xymatrix
	{
		D^-(S)\times D^-(S)\ar[r]^-{\otimes_S}\ar[d]	&D^-(S)\ar[d]\\
		D(S)\times D(S)\ar[r]^-{\otimes_S}			&D(S)
	}.
\]
\item Suppose \(f:S\longrightarrow T\) is a morphism in \(Sm/k\). We have a commutative diagram (up to a natural isomorphism)
\[
	\xymatrix
	{
		D^-(T)\ar[r]^{f^*}\ar[d]	&D^-(S)\ar[d]\\
		D(T)\ar[r]^{f^*}		&D(S)
	}
\]
\item Suppose \(f:S\longrightarrow T\) is a smooth morphism in \(Sm/k\). We have a commutative diagram (up to a natural isomorphism)
\[
	\xymatrix
	{
		D^-(S)\ar[r]^{f_{\#}}\ar[d]	&D^-(T)\ar[d]\\
		D(S)\ar[r]^{f_{\#}}			&D(T)
	}
\]
\end{enumerate}
\end{proposition}
\begin{proof}
This follows by direct computations.
\end{proof}
\subsubsection{On Categories of Effective Motives}
The next definition comes from \cite[Definition 9.2]{MVW}.
\begin{definition}
Define \(\mathscr{E}_{\mathbb{A}}\) to be the smallest thick subcategory of \(D^-(S)\) such that
\begin{enumerate}
\item\(Cone(\widetilde{\mathbb{Z}}_S(X\times_k\mathbb{A}^1)\longrightarrow\widetilde{\mathbb{Z}}_S(X))\in\mathscr{E}_{\mathbb{A}}.\)
\item\(\mathscr{E}_{\mathbb{A}}\) is closed under arbitrary direct sums.
\end{enumerate}
Set \(W_{\mathbb{A}}\) be the class of morphisms in \(D^-(S)\) whose cone is in \(\mathscr{E}_{\mathbb{A}}\). Define
\[\widetilde{DM}^{eff,-}(S)=D^-(S)[W_{\mathbb{A}}^{-1}]\]
to be the category of motives over \(S\). And morphisms in \(D^-(S)\) becoming isomorphisms after localizing by \(W_{\mathbb{A}}\) are called \(\mathbb{A}^1\)-weak equivalences.
\end{definition}
\begin{definition}
(See \cite[Definition 9.17]{MVW}) A complex \(K\in D^-(S)\) is called \(\mathbb{A}^1\)-local if for every \(\mathbb{A}^1\)-equivalence \(f:A\longrightarrow B\), the induced map
\[Hom_{D^-(S)}(B,K)\longrightarrow Hom_{D^-(S)}(A,K)\]
is an isomorphism.
\end{definition}
\begin{proposition}\label{homotopy invariance}
We say a map \(p:E\longrightarrow X\) in \(Sm/S\) to be an \(\mathbb{A}^n\)-bundle if there is an open covering \(\{U_i\}\) of \(X\) such that \(p^{-1}(U_i)\cong U_i\times_k\mathbb{A}^n\) over \(U_i\). In this case, \(\widetilde{\mathbb{Z}}_S(p):\widetilde{\mathbb{Z}}_S(E)\longrightarrow\widetilde{\mathbb{Z}}_S(X)\) is an isomorphism in \(\widetilde{DM}^{eff,-}(S)\).
\end{proposition}
\begin{proof}
For any \(X\in Sm/S\), the projection \(\widetilde{\mathbb{Z}}_S(X\times_k\mathbb{A}^n)\longrightarrow\widetilde{\mathbb{Z}}_S(X)\) is an \(\mathbb{A}^1\)-weak equivalence by definition. Suppose we have two open sets \(U_1\) and \(U_2\) of \(X\) such that the statement is true over \(U_1\), \(U_2\) and \(U_1\cap U_2\). Set \(E_i=p^{-1}(U_i)\). Then we have a commutative diagram with exact rows
\[
	\xymatrix
	{
		0\ar[r]	&\widetilde{\mathbb{Z}}_S(E_1\cap E_2)\ar[r]\ar[d]	&\widetilde{\mathbb{Z}}_S(E_1)\oplus\widetilde{\mathbb{Z}}_S(E_2)\ar[r]\ar[d]	&\widetilde{\mathbb{Z}}_S(p^{-1}(E_1\cup E_2))\ar[r]\ar[d]	&0\\
		0\ar[r]	&\widetilde{\mathbb{Z}}_S(U_1\cap U_2)\ar[r]		&\widetilde{\mathbb{Z}}_S(U_1)\oplus\widetilde{\mathbb{Z}}_S(U_2)\ar[r]		&\widetilde{\mathbb{Z}}_S(U_1\cup U_2)\ar[r]			&0
	}
\]
by Proposition \ref{MV-sequence}. So the statement is also true over \(U_1\cup U_2\). Then we could pick an finite open covering \(\{U_i\}\) of \(X\) such that \(p^{-1}(U_i)\cong U_i\times_k\mathbb{A}^n\) for every \(i\) and do induction on the number of open sets.
\end{proof}
\begin{proposition}\label{locality}
Suppose \(K\in D^-(S)\).
\begin{enumerate}
\item The natural map \(K\longrightarrow C_*K\) is an \(\mathbb{A}^1\)-weak equivalence.
\item The complex \(C_*K\) is \(\mathbb{A}^1\)-local.
\item The functor \(C_*\) induces an endofunctor of \(D^-(S)\).
\end{enumerate}
\end{proposition}
\begin{proof}
\begin{enumerate}
\item By the same proof as in \cite[Lemma 9.15]{MVW}.
\item By \cite[Proposition 5.2.36]{CD1}.
\item It's easy to check that \(C_*\) induces an endofunctor of \(K^-(S)\). If \(f:K\longrightarrow L\) is a quasi-isomorphism, then \(Cone(f)\) is acyclic. By (1), the natural morphism \(Cone(f)\longrightarrow C_*Cone(f)\) is an \(\mathbb{A}^1\)-equivalence. Hence it's an quasi-isomorphism by (2) and \cite[Lemma 9.21]{MVW}. So \(C_*Cone(f)=Cone(C_*f)\) is acyclic. So \(C_*f\) is a quasi-isomorphism.
\end{enumerate}
\end{proof}
\begin{definition}
(See \cite[Definition 14.17]{MVW}) Let \(X\in Sm/k\) and \(p,q\in\mathbb{Z}, q\geq 0\), we define the groups
\[H_E^{p,q}(X,\mathbb{Z})=Hom_{\widetilde{DM}^{eff,-}(pt)}(\widetilde{\mathbb{Z}}_{pt}(X),\widetilde{\mathbb{Z}}_{pt}(q)[p])\]
to be \(E\)-motivic cohomologies of \(X\).
\end{definition}
\begin{proposition}\label{derived1}
Let \(\varphi\) be the functor as before. We have an exact functor
\[\varphi^*:\widetilde{DM}^{eff,-}(S)\longrightarrow\widetilde{DM}^{eff,-}(T)\]
which is determined by the following commutative diagram
\[
	\xymatrix
	{
		D^-(S)\ar[r]^{\varphi^*}\ar[d]				&D^-(T)\ar[d]\\
		\widetilde{DM}^{eff,-}(S)\ar[r]^{\varphi^*}	&\widetilde{DM}^{eff,-}(T)
	}
\]
\end{proposition}
\begin{proof}
Let \(\mathbb{E}\) be the full subcategory of \(D^-(S)\) which consists of those complexes \(K\in D^-(S)\) who satisfies \(\varphi^*K\in\mathbb{E}_{\mathbb{A}}\). It's a thick subcategory of \(D^-(S)\). For any \(X\in Sm/S\), \(\varphi^*\) maps
\[\widetilde{\mathbb{Z}}_S(X\times_k\mathbb{A}^1)\longrightarrow\widetilde{\mathbb{Z}}_S(X)\]
to
\[\widetilde{\mathbb{Z}}_T((\psi X)\times_k\mathbb{A}^1)\longrightarrow\widetilde{\mathbb{Z}}_T(\psi X).\]
So \(\mathbb{E}_{\mathbb{A}}\subseteq\mathbb{E}\) by definition of \(\mathbb{E}_{\mathbb{A}}\) and exactness of \(\varphi^*\). So \(\varphi^*\) preserves objects in \(\mathbb{E}_{\mathbb{A}}\). Hence \(\varphi^*\) preserves \(\mathbb{A}^1\)-weak equivalences by exactness of \(\varphi^*\). Then the statement follows by \cite[Proposition 4.6.2]{Kra}.
\end{proof}
\begin{proposition}\label{derived2}
\begin{enumerate}
\item There is a tensor product
\[\otimes_S:\widetilde{DM}^{eff,-}(S)\times\widetilde{DM}^{eff,-}(S)\longrightarrow\widetilde{DM}^{eff,-}(S),\]
which is determined by the following commutative diagram
\[
	\xymatrix
	{
		D^-(S)\times D^-(S)\ar[r]^-{\otimes_S}\ar[d]							&D^-(S)\ar[d]\\
		\widetilde{DM}^{eff,-}(S)\times\widetilde{DM}^{eff,-}(S)\ar[r]^-{\otimes_S}	&\widetilde{DM}^{eff,-}(S)
	}.
\]
And for any \(K\in\widetilde{DM}^{eff,-}(S)\), the functor \(K\otimes_S-\) is exact.
\item Suppose \(f:S\longrightarrow T\) is a smooth morphism in \(Sm/k\). There is an exact functor
\[f_{\#}:\widetilde{DM}^{eff,-}(S)\longrightarrow\widetilde{DM}^{eff,-}(T),\]
which is determined by the following commutative diagram
\[
	\xymatrix
	{
		D^-(S)\ar[r]^{f_{\#}}\ar[d]			&D^-(T)\ar[d]\\
		\widetilde{DM}^{eff,-}(S)\ar[r]^{f_{\#}}	&\widetilde{DM}^{eff,-}(T)
	}.
\]
\item Suppose \(f:S\longrightarrow T\) is a morphism in \(Sm/k\). There is an exact functor
\[f^*:\widetilde{DM}^{eff,-}(T)\longrightarrow\widetilde{DM}^{eff,-}(S),\]
which is determined by the following commutative diagram
\[
	\xymatrix
	{
		D^-(T)\ar[r]^{f^*}\ar[d]			&D^-(S)\ar[d]\\
		\widetilde{DM}^{eff,-}(T)\ar[r]^{f^*}	&\widetilde{DM}^{eff,-}(S)
	}.
\]
\end{enumerate}
\end{proposition}
\begin{proof}
\begin{enumerate}
\item Suppose \(Y\in Sm/S\). In the definition of \(\varphi\), we take \((Y, S, T):=(Y, S, S)\). Then \(\varphi^*F\cong F\otimes_S\widetilde{\mathbb{Z}}_S(Y)\) for any \(F\in\widetilde{Sh}(S)\) by Proposition \ref{hom-tensor}.

Now given an \(\mathbb{A}^1\)-weak equivalence \(a\), \(\widetilde{\mathbb{Z}}_S(Y)\otimes_Sa\) is also an \(\mathbb{A}^1\)-weak equivalence by applying Proposition \ref{derived1} to \(\varphi\). Now apply a similar method as in the third paragraph of \cite[Lemma 9.5]{MVW} to show that the functor \(K\otimes_S-:D^-(S)\longrightarrow D^-(S)\) preserves \(\mathbb{A}^1\)-weak equivalence for any \(K\in D^-(S)\). Finally we apply \cite[Proposition 4.6.2]{Kra} to the functor \(K\otimes_S-\).
\item In the definition of \(\varphi\), we take \((Y, S, T):=(S, S, T)\) and apply Proposition \ref{derived1}.
\item In the definition of \(\varphi\), we take \((Y, S, T):=(T, S, T)\) and apply Proposition \ref{derived1}.
\end{enumerate}
\end{proof}
\begin{proposition}\label{jing}
Let \(f:S\longrightarrow T\) be a smooth morphism in \(Sm/k\). We have an adjoint pair
\[f_{\#}:\widetilde{DM}^{eff,-}(S)\rightleftharpoons\widetilde{DM}^{eff,-}(T):f^*.\]
\end{proposition}
\begin{proof}
By the same method as in Proposition \ref{derived adjunction} since \(\varphi^*\) preserves \(\mathbb{E}_{\mathbb{A}}\) by Proposition \ref{derived1}.
\end{proof}
\begin{proposition}\label{operations1}
Suppose \(f:S\longrightarrow T\) is a morphism in \(Sm/k\).
\begin{enumerate}
\item For any \(K, L\in\widetilde{DM}^{eff,-}(T)\), we have
\[f^*(K\otimes_SL)\cong(f^*K)\otimes_S(f^*L).\]
\item If \(f\) is smooth, then for any \(K\in\widetilde{DM}^{eff,-}(S)\) and \(L\in\widetilde{DM}^{eff,-}(T)\), we have
\[f_{\#}(K\otimes_Sf^*L)\cong(f_{\#}K)\otimes_SL.\]
\end{enumerate}
\end{proposition}
\begin{proof}
Directly follows from Proposition \ref{sheaf formulas} and Proposition \ref{operations} since everything works termwise.
\end{proof}

In \cite[Proposition 3.5]{CD} and \cite[Definition 3.2.1]{DF}, they defined \(\widetilde{DM}^{eff}(S)\) as the the Verdier localization of \(D(S)\) with respect to homotopy invariant conditions. So the localization induces a triangulated structure on \(\widetilde{DM}^{eff}(S)\) (See \cite[Lemma 4.3.1]{Kra}). And this is the triangulated structure we will impose on \(\widetilde{DM}^{eff}(S)\).
\begin{proposition}\label{triangulated1}
There is a fully faithful exact functor \(\widetilde{DM}^{eff,-}(S)\longrightarrow\widetilde{DM}^{eff}(S)\) which is determined by the commutative diagram (See Proposition \ref{triangulated})
\[
	\xymatrix
	{
		D^-(S)\ar[r]\ar[d]			&D(S)\ar[d]\\
		\widetilde{DM}^{eff,-}(S)\ar[r]	&\widetilde{DM}^{eff}(S)
	}.
\]
\end{proposition}
\begin{proof}
The functor \(\widetilde{DM}^{eff,-}(S)\longrightarrow\widetilde{DM}^{eff}(S)\) is induced and exact by \cite[Proposition 4.6.2]{Kra}. And for any \(K, L\in\widetilde{DM}^{eff,-}(S)\), we have a commutative diagram
\[
	\xymatrix
	{
		Hom_{\widetilde{DM}^{eff,-}}(K,L)\ar[r]^-u\ar[d]_{\alpha}	&Hom_{\widetilde{DM}^{eff,-}}(C_*K,C_*L)\ar[d]	&Hom_{D^-}(C_*K,C_*L)\ar[l]_-{\gamma}\ar[d]_{\cong}\\
		Hom_{\widetilde{DM}^{eff}}(K,L)\ar[r]^-v				&Hom_{\widetilde{DM}^{eff}}(C_*K,C_*L))	&Hom_{D}(C_*K,C_*L)\ar[l]_-{\beta}
	},
\]
where \(u\), \(v\), \(\gamma\) and \(\beta\) are isomorphisms by Proposition \ref{locality}. So \(\alpha\) is an isomorphism.
\end{proof}
And there is also a compability result between the natural inclusion and \(\otimes_S\), \(f^*\), \(f_{\#}\) like Proposition \ref{-}.
\subsection{For Symmetric Spectra}
Now we are going to discuss operations for spectra, in order to stabilize the category \(\widetilde{DM}^{eff,-}(S)\). The main reference is \cite[5.3]{CD1}.
\subsubsection{Symmetric Spectra}\label{symmetric spectra}
Let \(\mathscr{A}\) be a symmetric closed monoidal abelian category with arbitrary products. We can define the category of symmetric sequences \(\mathscr{A}^\mathfrak{S}\) as in \cite[Definition 5.3.5]{CD1}. It is also a closed symmetric monoidal abelian category by \cite[Definition 5.3.7]{CD1} and \cite[Lemma 2.1.6]{HSS}. Here, if we have two symmetric sequences \(A\) and \(B\), we define \(A\otimes^{\mathfrak{S}}B\) by
\[(A\otimes^{\mathfrak{S}}B)_n=\oplus_pS_n\times_{S_p\times S_{n-p}}(A_p\otimes B_{n-p}).\]
And we define \(\underline{Hom}^{\mathfrak{S}}(A,B)\) by
\[\underline{Hom}^{\mathfrak{S}}(A,B)_n=\prod_p\underline{Hom}_{S_p}(A_p,B_{n+p}),\]
where \(\underline{Hom}_{S_p}(A_p,B_{n+p})\) (with an obvious \(S_n\)-action) is the kernel of the map
\[\xymatrixcolsep{4pc}\xymatrix{\underline{Hom}(A_p,B_{n+p})\ar[r]^-{(\sigma^*-(1\times\sigma)_*)}&\prod_{\sigma\in S_p}\underline{Hom}(A_p,B_{n+p})}.\]
(see \cite[Definition 2.1.3]{HSS} and \cite[Theorem 2.1.11]{HSS})
\begin{proposition}\label{sym hom-tensor}
In the context above, for any symmetric sequences \(A\), \(B\), \(C\), we have
\[Hom(A\otimes^{\mathfrak{S}}B,C)\cong Hom(A,\underline{Hom}^{\mathfrak{S}}(B,C))\]
naturally.
\end{proposition}
\begin{proof}
Giving a morphism from \(A\otimes^{\mathfrak{S}}B\) to \(C\) is equivalent to giving \(S_p\times S_q\)-equivariant maps
\[f_{p,q}:A_p\otimes B_q\longrightarrow C_{p+q}.\]
And that is equivalent to giving \(S_p\)-equivariant maps
\[g_{p,q}:A_p\longrightarrow\underline{Hom}(B_q,C_{p+q})\]
such that for any \(\sigma\in S_q\),
\[\underline{Hom}(\sigma,C_{p+q})\circ g_{p,q}=\underline{Hom}(B_q,id_{S_p}\times\sigma)\circ g_{p,q}.\]
And this just says that \(g_{p,q}\) factor through \(\underline{Hom}_{S_q}(B_q,C_{p+q})\).
\end{proof}
And the abelian structure of \(\mathscr{A}^{\mathfrak{S}}\) is just defined termwise. Moreover, we have adjunctions
\[i_0:\mathscr{A}\rightleftharpoons\mathscr{A}^{\mathfrak{S}}:ev_0\]
and
\[-\{-i\}:\mathscr{A}^{\mathfrak{S}}\rightleftharpoons\mathscr{A}^{\mathfrak{S}}:-\{i\} (i\geq 0)\]
as in \cite[5.3.5.1]{CD1} and \cite[6.4.1]{CD}.

Now suppose \(R\in\mathscr{A}\). Then \(Sym(R)\in\mathscr{A}^{\mathfrak{S}}\) is a commutative monoid object as in \cite[5.3.8]{CD1}. Define \(Sp_R(\mathscr{A})\) to be the category of \(Sym(R)\)-modules in \(\mathscr{A}^{\mathfrak{S}}\). They are called symmetric R-spectra. Then it's also a symmetric closed monoidal abelian category by \cite[Theorem 2.2.10]{HSS} and Proposition \ref{sym hom-tensor}. (The corresponding tensor product and inner-hom are just denoted by \(\otimes\) and \(\underline{Hom}\) for convenience)

We have an adjunction
\[Sym(R)\otimes^{\mathfrak{S}}-:\mathscr{A}^{\mathfrak{S}}\rightleftharpoons Sp_R(\mathscr{A}):U,\]
where \(U\) is the forgetful functor. Thus we get an adjunction
\[\Sigma^{\infty}:\mathscr{A}\rightleftharpoons Sp_R(\mathscr{A}):\Omega^{\infty},\]
where \(\Sigma^{\infty}=(Sym(R)\otimes^{\mathfrak{S}}-)\circ i_0\), \(\Omega^{\infty}=ev_0\circ U\) and \(\Sigma^{\infty}\) is monoidal.

We have a canonical identification
\[A\otimes^{\mathfrak{S}}(B\{-i\})=(A\otimes^{\mathfrak{S}}B)\{-i\}\]
and a morphism
\[A\otimes^{\mathfrak{S}}(B\{i\})\longrightarrow(A\otimes^{\mathfrak{S}}B)\{i\}\]
defined by the composition
\[A\otimes^{\mathfrak{S}}(B\{i\})\longrightarrow(A\otimes^{\mathfrak{S}}(B\{i\}))\{-i\}\{i\}=(A\otimes^{\mathfrak{S}}(B\{i\}\{-i\}))\{i\}\longrightarrow(A\otimes^{\mathfrak{S}}B)\{i\}.\]
Restricting the functors \(-\{-i\}\) and \(-\{i\}\) on spectra, there is also an adjunction
\[-\{-i\}:Sp_R(\mathscr{A})\rightleftharpoons Sp_R(\mathscr{A}):-\{i\},\]
where the module structure \(Sym(R)\otimes^{\mathfrak{S}}(A\{-i\})\longrightarrow A\{-i\}\) of \(A\{-i\}\) is just got by applying \(-\{-i\}\) on that of \(A\) and the module structure \(Sym(R)\otimes^{\mathfrak{S}}(B\{i\})\longrightarrow B\{i\}\) of \(B\{i\}\) is got by the composition
\[Sym(R)\otimes^{\mathfrak{S}}(B\{i\})\longrightarrow(Sym(R)\otimes^{\mathfrak{S}}B)\{i\}\longrightarrow B\{i\},\]
where the last arrow is got by applying \(-\{i\}\) on the module structure of \(B\). And we still have an isomorphism
\[A\otimes_S(B\{-i\})\cong(A\otimes_SB)\{-i\}\]
and a morphism
\[A\otimes_S(B\{i\})\longrightarrow(A\otimes_SB)\{i\}\]
defined by the same way as above.
\begin{definition}
(See \cite[Definition 5.3.16]{CD1}) For any \(S\in Sm/k\), define
\[\mathbbm{1}_S\{1\}=Sym(coker(\widetilde{\mathbb{Z}}_S(S)\longrightarrow\widetilde{\mathbb{Z}}_S(\mathbb{G}_m)))\]
and
\[\mathbbm{1}_S'\{1\}=Sym(coker(\widetilde{c}_S(S)\longrightarrow\widetilde{c}_S(\mathbb{G}_m))).\]
Then define \(Sp(S)\) to be \(Sp_{\mathbbm{1}_S\{1\}}(\widetilde{Sh}(S))\) and \(Sp'(S)\) to be \(Sp_{\mathbbm{1}_S'\{1\}}(\widetilde{PSh}(S))\).
\end{definition}
We have an adjunction
\[\widetilde{a}:\widetilde{PSh}(S)^{\mathfrak{S}}\rightleftharpoons\widetilde{Sh}(S)^{\mathfrak{S}}:\widetilde{o}\]
where both functors are defined termwise (see Proposition \ref{sheafication}) and \(\widetilde{a}\) is monoidal by definition. Restricting the above functors on modules, there is also an adjunction
\[\widetilde{a}:Sp'(S)\rightleftharpoons Sp(S):\widetilde{o},\]
where the module structure \(\mathbbm{1}_S\{1\}\otimes^{\mathfrak{S}}_S\widetilde{a}(A)\longrightarrow\widetilde{a}(A)\) of \(\widetilde{a}(A)\) is just got by applying \(\widetilde{a}\) on that of \(A\) and the module structure \(\mathbbm{1}_S'\{1\}\otimes^{\mathfrak{S}}_S\widetilde{o}(B)\longrightarrow\widetilde{o}(B)\) of \(\widetilde{o}(B)\) is got by precomposing that of \(B\) with the sheafication map \(\mathbbm{1}_S'\{1\}\otimes^{\mathfrak{S}}_S\widetilde{o}(B)\longrightarrow\mathbbm{1}_S\{1\}\otimes^{\mathfrak{S}}_SB\). The functor \(\widetilde{a}\) is again monoidal.

Now let \(f:S\longrightarrow T\) be a morphism in \(Sm/k\). We have an adjunction
\[f^*:\widetilde{Sh}(T)^{\mathfrak{S}}\rightleftharpoons\widetilde{Sh}(S)^{\mathfrak{S}}:f_*\]
where both functors are defined termwise (see Proposition \ref{pullback}) and \(f^*\) is monoidal by Proposition \ref{sheaf formulas}, (4). Restricting the above functors on spectra, there is also an adjunction
\[f^*:Sp(T)\rightleftharpoons Sp(S):f_*,\]
where the module structure \(\mathbbm{1}_S\{1\}\otimes^{\mathfrak{S}}_Sf^*A\longrightarrow f^*A\) of \(f^*A\) is just got by applying \(f^*\) on that of \(A\) and the module structure \(\mathbbm{1}_T\{1\}\otimes^{\mathfrak{S}}_Tf_*B\longrightarrow f_*B\) of \(f_*B\) is got by the composition
\[\mathbbm{1}_T\{1\}\otimes^{\mathfrak{S}}_Tf_*B\longrightarrow f_*(\mathbbm{1}_S\{1\}\otimes^{\mathfrak{S}}_Sf^*f_*B)\longrightarrow f_*(\mathbbm{1}_S\{1\}\otimes^{\mathfrak{S}}_SB)\longrightarrow f_*B,\]
where the last arrow is got by applying \(f_*\) on the module structure of \(B\). The functor \(f^*\) is also monoidal by the construction of the tensor product (see \cite[Lemma 2.2.2]{HSS}). And the same construction gives an another adjunction
\[f^*:Sp'(T)\rightleftharpoons Sp'(S):f_*.\]

Suppose further \(f\) is smooth. We have an adjunction
\[f_{\#}:\widetilde{Sh}(S)^{\mathfrak{S}}\rightleftharpoons\widetilde{Sh}(T)^{\mathfrak{S}}:f^*\]
where both functors are defined termwise (see Proposition \ref{lowerhash}) and
\[f_{\#}(A\otimes^{\mathfrak{S}}_Sf^*B)\cong (f_{\#}A)\otimes^{\mathfrak{S}}_TB\]
also holds by Proposition \ref{operations}, (4). Restricting the above functors on spectra, there is also an adjunction
\[f_{\#}:Sp(S)\rightleftharpoons Sp(T):f^*,\]
where the module structure \(\mathbbm{1}_T\{1\}\otimes^{\mathfrak{S}}f_{\#}A\longrightarrow f_{\#}A\) of \(f_{\#}A\) is got by the composition
\[\mathbbm{1}_T\{1\}\otimes^{\mathfrak{S}}_Tf_{\#}A\cong f_{\#}(\mathbbm{1}_S\{1\}\otimes^{\mathfrak{S}}_SA)\longrightarrow f_{\#}A,\]
where the last arrow is got by applying \(f_{\#}\) on the module structure of \(A\). And we also have
\[f_{\#}(A\otimes_Sf^*B)\cong (f_{\#}A)\otimes_TB\]
for spectra by the construction of the tensor product (see \cite[Lemma 2.2.2]{HSS}). And the same construction gives an another adjunction
\[f_{\#}:Sp'(S)\rightleftharpoons Sp'(T):f^*.\]

One checks that when F=\(-\otimes_SA\), \(f_{\#}\), \(f^*\), \(-\{-i\}\), \(-\{i\}\), \(\Sigma^{\infty}\) or \(\Omega^{\infty}\), there is a natural isomorphism \(\widetilde{a}\circ F\cong F\circ\widetilde{a}\).

Suppose \(i\geq 0\). Then for any \(F\in\widetilde{Sh}(S)\), we have
\[(\Sigma^{\infty}F)\{i\}\cong\Sigma^{\infty}(\widetilde{\mathbb{Z}}_{tr}(\mathbb{G}_m^{\wedge 1})^{\otimes i}\otimes_SF).\]
Moreover, for any \(X\in Sm/S\),
\[Hom_{Sp(S)}((\Sigma^{\infty}\widetilde{\mathbb{Z}}_S(X))\{-i\},A)=A_i(X)\]
and
\[Hom_{Sp'(S)}((\Sigma^{\infty}\widetilde{c}_S(X))\{-i\},B)=B_i(X).\]
So \((\Sigma^{\infty}\widetilde{\mathbb{Z}}_S(X))\{-i\}\) (resp. \((\Sigma^{\infty}\widetilde{c}_S(X))\{-i\}\)) are systems of generators of \(Sp(S)\) (resp. \(Sp'(S)\)) (See \cite[6.7]{CD} and \cite[5.3.11]{CD1}). This enables us to imitate methods in Section \ref{effective}.
\subsubsection{On Derived Categories}
We denote by \(D^-_{Sp}(S)\) (resp. \(D_{Sp}(S)\)) the derived category of bounded above (resp. unbounded) complex of spectra in \(Sp(S)\).
\begin{proposition}\label{spCech}
Let \(X, U\in Sm/S\) and \(p:U\longrightarrow X\) be a Nisnevich covering. Then the complex \((\Sigma^{\infty}\breve{C}(U/X))\{-i\}\) (defined by termise application), is exact after sheafifying as a complex of \(Sp(S)\).
\end{proposition}
\begin{proof}
One easily see that \((\Sigma^{\infty}A)\{-i\}=Sym(\mathbb{Z}_S'\{1\})\otimes^{\mathfrak{S}}_S(i_0(A)\{-i\})\) for any \(A\in\widetilde{PSh}(S)\). Then the statement follows by the equality
\[\breve{C}(U/X)\otimes_S^{pr}\widetilde{c}_S(Y)=\breve{C}(U\times_SY/X\times_SY)\]
for any \(Y\in Sm/S\) and Proposition \ref{Cech}.
\end{proof}
\begin{definition}
We call a spectrum \(A\in Sp'(S)\) free if it's a direct sum of spectra of the form \((\Sigma^{\infty}\widetilde{c}_S(X))\{-i\}\). We call \(A\) projective if it's a direct summand of a free spectrum. A spectrum in \(Sp(S)\) is called free (resp. projective) if it's a sheafication of a free (resp. projective) spectrum in \(Sp'(S)\). A bounded above complex of spectra in \(Sp(S)\) is called free (projective) if all its terms are free (projective).
\end{definition}
\begin{definition}
A projective resolution of a bounded above spectrum complex \(K\) is a projective complex with a quasi-isomorphism \(P\longrightarrow K\). If \(K\) is already projective, we take \(P=K\).
\end{definition}
Now let \(S, T\in Sm/k\), \(j\geq 0\) and \(Y\) be a scheme with morphisms \(\xymatrix{S&Y\ar[l]_f\ar[r]^g&T}\) where \(g\) is smooth. Consider in this section the adjunctions
\[\begin{array}{cccc}\phi^*=\{-j\}\circ g_{\#}\circ f^*:&Sp(S)&\rightleftharpoons&Sp(T):\phi_*=f_*\circ g^*\circ\{j\}\end{array}\]
\[\begin{array}{cccc}\varphi^*=\{-j\}\circ g_{\#}\circ f^*:&Sp'(S)&\rightleftharpoons&Sp'(T):\varphi_*=f_*\circ g^*\circ\{j\}\end{array}\]
and the functor
\[\begin{array}{cccc}\psi:&Sm_S&\longrightarrow&Sm_T\\&X&\longmapsto&X\times_SY\end{array}.\]
They are determined by the quadruple \((Y, S, T, j)\).
\begin{proposition}
For any \(F\in Sp'(S)\),
\[\widetilde{a}((L_i\varphi^*)\widetilde{a}(F))\cong\widetilde{a}((L_i\varphi^*)F)\]
as spectra in \(Sp(S)\) for any \(i\geq 0\), where \(L_i\varphi^*\) means the \(i^{th}\) left derived functor of \(\varphi^*\).
\end{proposition}
\begin{proof}
We show at first that for any \(F\in Sp'(S)\) with \(\widetilde{a}(F)=0\)
\[\widetilde{a}(L_i\varphi^*(F))=0\]
for any \(i\geq 0\). Then for any \(F\in Sp'(S)\), denote by \(\theta\) the natural map \(F\longrightarrow\widetilde{a}(F)\). We have
\[\widetilde{a}(coker(\theta))=\widetilde{a}(ker(\theta))=0.\]
Hence for any \(i\geq 0\), we have
\[\widetilde{a}(L_i\varphi^*\widetilde{a}(F))\cong\widetilde{a}(L_i\varphi^*Im(\theta))\cong\widetilde{a}(L_i\varphi^*F)\]
by using long exact sequences. Hence the statement follows.

Now we prove the claim. We do induction on \(i\). The claim is true for \(i=0\) and suppose it's true for \(i<n\).

For any \(F\in Sp'(S)\), we have a surjection
\[\oplus_{a\in F_t(X), t\geq 0}(\Sigma^{\infty}\widetilde{c}_S(X))\{-t\}\longrightarrow F\]
defined by each section of \(F_t\) on each \(X\in Sm/S\). Since \(\widetilde{a}(F)=0\), for any \(a\in F_t(X), X\in Sm/S\), there is a finite Nisnevich covering \(U_a\longrightarrow X\) of \(X\) such that \(a|_{U_a}=0\). So the composition
\[\oplus_{a\in F_t(X), t\geq 0}(\Sigma^{\infty}\widetilde{c}_S(U_a))\{-t\}\longrightarrow\oplus_{a\in F(X), t\geq 0}(\Sigma^{\infty}\widetilde{c}_S(X))\{-t\}\longrightarrow F\]
is zero. Then we have got a surjection
\[\oplus_{a\in F(X), t\geq 0}H_0((\Sigma^{\infty}\breve{C}(U_a/X))\{-t\})\longrightarrow F\]
with kernel \(K\). Proposition \ref{Cech} implies that
\[\widetilde{a}(H_p((\Sigma^{\infty}\breve{C}(U/X))\{-t\}))=0\]
for any Nisnevich covering \(U\longrightarrow X\), \(t\geq 0\) and \(p\in\mathbb{Z}\). So \(\widetilde{a}(K)=0\) also. We have a hypercohomology spectral sequence
\[(L_p\varphi^*)H_q((\Sigma^{\infty}\breve{C}(U/X))\{-t\})\Longrightarrow(\mathbb{L}_{p+q}\varphi^*)((\Sigma^{\infty}\breve{C}(U/X))\{-t\}).\]
Hence
\[\widetilde{a}((\mathbb{L}_n\varphi^*)((\Sigma^{\infty}\breve{C}(U/X))\{-t\}))\cong\widetilde{a}((L_n\varphi^*)H_0((\Sigma^{\infty}\breve{C}(U/X))\{-t\}))\]
by induction hypothesis. But
\[\widetilde{a}((\mathbb{L}_n\varphi^*)(\Sigma^{\infty}\breve{C}(U/X))\{-t\})\cong\widetilde{a}(H_n(\varphi^*((\Sigma^{\infty}\breve{C}(U/X))\{-t\})))\]
by definition of hypercohomology and the latter one vanishes since we have
\[\varphi^*((\Sigma^{\infty}\breve{C}(U/X))\{-t\})=(\Sigma^{\infty}\breve{C}(\psi U/\psi X))\{-t-j\}.\]
So
\[\widetilde{a}((L_n\varphi^*)H_0((\Sigma^{\infty}\breve{C}(U/X))\{-t\}))=0.\]
So
\[\widetilde{a}(L_n\varphi^*F)\cong\widetilde{a}(L_{n-1}\varphi^*K)=0\]
by long exact sequence and induction hypothesis.
\end{proof}
\begin{proposition}\label{spadapted}
Let functor \(\phi^*\) takes acyclic projective complexes to acyclic projective complexes.
\end{proposition}
\begin{proof}
The same as Proposition \ref{adapted}.
\end{proof}
\begin{proposition}\label{spderived}
We have an exact functor
\[L\phi^*:D^-_{Sp}(S)\longrightarrow D^-_{Sp}(T)\]
which maps any \(K\in D^-_{Sp}(S)\) to \(\phi^*P\), where \(P\) is a projective resolution \(K\).
\end{proposition}
\begin{proof}
The same as Proposition \ref{derived}.
\end{proof}

We will just write \(L\phi^*\) above by \(\phi^*\) for convenience.

Now we apply the general results above to \(\otimes_S\), \(f_{\#}\) and \(f^*\).
\begin{proposition}
\begin{enumerate}
\item There is a tensor product
\[\begin{array}{cccccc}\otimes_S:&D^-_{Sp}(S)&\times&D^-_{Sp}(S)&\longrightarrow&D^-_{Sp}(S)\\&(K&,&L)&\longmapsto&P\otimes_SQ\end{array},\]
where \(P, Q\) are projective resolutions of \(K, L\), respectively and the last tensor means taking the total complex of the bicomplex \(\{P_i\otimes_SQ_j\}\). And for any \(K\in D^-_{Sp}(S)\), the functor \(K\otimes_S-\) is exact.
\item Suppose \(f:S\longrightarrow T\) is a smooth morphism in \(Sm/k\). There is an exact functor
\[\begin{array}{cccc}f_{\#}:&D^-_{Sp}(S)&\longrightarrow&D^-_{Sp}(T)\\&K&\longmapsto&f_{\#}P\end{array},\]
where \(P\) is a projective resolution of \(K\).
\item Suppose \(f:S\longrightarrow T\) is a morphism in \(Sm/k\). There is an exact functor
\[\begin{array}{cccc}f^*:&D^-_{Sp}(T)&\longrightarrow&D^-_{Sp}(S)\\&K&\longmapsto&f^*P\end{array},\]
where \(P\) is a projective resolution of \(K\).
\item Suppose \(i\geq 0\), there is an exact functor
\[\begin{array}{cccc}-\{-i\}:&D^-_{Sp}(S)&\longrightarrow&D^-_{Sp}(S)\\&K&\longmapsto&P\{-i\}\end{array},\]
where \(P\) is a projective resolution of \(K\).
\end{enumerate}
\end{proposition}
\begin{proof}
In (1), (2) and (3), take \(j=0\) in the definition of \(\phi\) and proceed as Proposition \ref{derived functors}. For (4), take the quadruple \((S,S,S,i)\) and use Proposition \ref{spderived}.
\end{proof}
\begin{proposition}\label{spderived adjunction}
\begin{enumerate}
\item Suppose \(f:S\longrightarrow T\) is a smooth morphism in \(Sm/k\), we have an adjoint pair
\[f_{\#}:D^-_{Sp}(S)\rightleftharpoons D^-_{Sp}(T):f^*.\]
\item We have an adjoint pair
\[-\{-i\}:D^-_{Sp}(S)\rightleftharpoons D^-_{Sp}(S):-\{i\}.\]
\end{enumerate}
\end{proposition}
\begin{proof}
The same as Proposition \ref{derived adjunction} since \(-\{i\}\) is an exact functor.
\end{proof}
Now we are going to compare \(D^-_{Sp}(S)\) with \(D^-(S)\) defined in Section \ref{effective}.
\begin{proposition}
The functor \(\Sigma^{\infty}:\widetilde{Sh}(S)\longrightarrow Sp(S)\) takes acyclic projective complexes of sheaves to acyclic projective complexes of spectra.
\end{proposition}
\begin{proof}
Let \(P\) be a projective sheaf. Then
\[(\Sigma^{\infty}P)_n=\mathbbm{1}_S\{1\}^{\otimes n}\otimes_S P_n\]
by definition. And a tensor product between projective sheaves is again projective. So \(\Sigma^{\infty}P\) is projective.

Let \(Q\) be an acyclic projective complex of sheaves. Then \(\Sigma^{\infty}Q\) consists of complexes like \(\mathbbm{1}_S\{1\}^{\otimes n}\otimes_SQ\). And they are all acyclic by Proposition \ref{adapted}.
\end{proof}
\begin{proposition}
There is an exact functor
\[L\Sigma^{\infty}:D^-(S)\longrightarrow D^-_{Sp}(S)\]
which maps \(K\) to \(\Sigma^{\infty}P\), where \(P\) is a projective resolution of \(K\).
\end{proposition}
\begin{proof}
The same as Proposition \ref{derived}.
\end{proof}
As usual, we will write \(L\Sigma^{\infty}\) as \(\Sigma^{\infty}\) for convenience.
\begin{proposition}\label{sigma adjunction}
There is an adjoint pair
\[\Sigma^{\infty}:D^-(S)\rightleftharpoons D^-_{Sp}(S):\Omega^{\infty}.\]
\end{proposition}
\begin{proof}
The same as Proposition \ref{derived adjunction} since \(\Omega^{\infty}\) is an exact functor.
\end{proof}
\begin{proposition}\label{ff}
The functor \(\Sigma^{\infty}:D^-(S)\longrightarrow D^-_{Sp}(S)\) is fully faithful.
\end{proposition}
\begin{proof}
Suppose \(K, L\in D^-(S)\) with projective resolutions \(P, Q\) respectively. Then there is a commutative diagram
\[
	\xymatrix
	{
		Hom_{D^-(S)}(K,L)\ar[r]^-{\Sigma^{\infty}}\ar[d]_{\cong}	&Hom_{D^-_{Sp}(S)}(\Sigma^{\infty}K,\Sigma^{\infty}L)\ar[d]_{\cong}\\
		Hom_{D^-(S)}(P,Q)\ar[r]^-{\Sigma^{\infty}}				&Hom_{D^-_{Sp}(S)}(\Sigma^{\infty}P,\Sigma^{\infty}Q)\ar[d]_{\cong}\\
													&Hom_{D^-(S)}(P,\Omega^{\infty}\Sigma^{\infty}Q)
	}.
\]
And we observe that \(\Omega^{\infty}\Sigma^{\infty}Q=Q\) (The same for \(P\)).
\end{proof}
\begin{proposition}
\begin{enumerate}
\item We have a commutative diagram (up to a canonical isomorphism)
\[
	\xymatrix
	{
		D^-(S)\times D^-(S)\ar[r]^-{\otimes_S}\ar[d]_{\Sigma^{\infty}\times\Sigma^{\infty}}	&D^-(S)\ar[d]_{\Sigma^{\infty}}\\
		D^-_{Sp}(S)\times D^-_{Sp}(S)\ar[r]^-{\otimes_S}							&D^-_{Sp}(S)
	}.
\]
\item Suppose \(f:S\longrightarrow T\) is a morphism in \(Sm/k\). We have a commutative diagram (up to a canonical isomorphism)
\[
	\xymatrix
	{
		D^-(T)\ar[r]^{f^*}\ar[d]_{\Sigma^{\infty}}	&D^-(S)\ar[d]_{\Sigma^{\infty}}\\
		D^-_{Sp}(T)\ar[r]^{f^*}				&D^-_{Sp}(S)
	}
\]
\item Suppose \(f:S\longrightarrow T\) is a smooth morphism in \(Sm/k\). We have a commutative diagram (up to a canonical isomorphism)
\[
	\xymatrix
	{
		D^-(S)\ar[r]^{f_{\#}}\ar[d]_{\Sigma^{\infty}}	&D^-(T)\ar[d]_{\Sigma^{\infty}}\\
		D^-_{Sp}(S)\ar[r]^{f_{\#}}					&D^-_{Sp}(T)
	}
\]
\end{enumerate}
\end{proposition}
\begin{proof}
This follows by direct computations.
\end{proof}

In \cite[Theorem 1.7]{CD}, they defined a model structure \(\mathfrak{M}_{Sp}\) on the category of unbounded complexes of symmetric spectra over \(S\). This is a cofibrantly generated model structure where the cofibrations are those \(I\)-cofibrations (See \cite[Definition 2.1.7]{Hov}) where \(I\) consists of the morphisms \(S^{n+1}(\Sigma^{\infty}\widetilde{\mathbb{Z}}_S(X)\{-i\})\longrightarrow D^n(\Sigma^{\infty}\widetilde{\mathbb{Z}}_S(X)\{-i\})\) for any \(X\in Sm/S\) and \(i\geq 0\) and weak equivalences are quasi-morphisms between complexes.
\begin{proposition}
Bounded above projective complexes are cofibrant objects in \(\mathfrak{M}_{Sp}\).
\end{proposition}
\begin{proof}
The same as Proposition \ref{cofibrant}.
\end{proof}

Moreover, \(\mathfrak{M}_{Sp}\) is stable and left proper so it induces a triangulated structure \(\mathfrak{T}'\) on \(D_{Sp}(S)\) (See \cite[Theoreme 4.1.49]{A}). The classical triangulated structure of \(D_{Sp}(S)\) or \(D^-_{Sp}(S)\) is denoted by \(\mathfrak{T}\).
\begin{proposition}
The natural functor
\[(D^-_{Sp}(S),\mathfrak{T})\longrightarrow(D_{Sp}(S),\mathfrak{T}')\]
is fully faithful exact.
\end{proposition}
\begin{proof}
The same as Proposition \ref{triangulated}.
\end{proof}
And there is also a compability result between the natural inclusion and \(\otimes_S\), \(f^*\), \(f_{\#}\), \(\Sigma^{\infty}\), \(-\{-i\},i\geq 0\) like Proposition \ref{-}.
\subsubsection{On Categories of Effective Motives}\label{effective1}
\begin{definition}
(See \cite[5.2.15]{CD1}) Define \(\mathscr{E}_{\mathbb{A}}\) to be the smallest thick subcategory of \(D^-_{Sp}(S)\) such that
\begin{enumerate}
\item\((\Sigma^{\infty}Cone(\widetilde{\mathbb{Z}}_S(X\times_k\mathbb{A}^1)\longrightarrow\widetilde{\mathbb{Z}}_S(X)))\{-i\}\in\mathscr{E}_{\mathbb{A}}, i\geq 0.\)
\item\(\mathscr{E}_{\mathbb{A}}\) is closed under arbitrary direct sums.
\end{enumerate}
Set \(W_{\mathbb{A}}\) be the class of morphisms in \(D^-_{Sp}(S)\) whose cone is in \(\mathscr{E}_{\mathbb{A}}\). Define
\[\widetilde{DM}^{eff,-}_{Sp}(S)=D^-_{Sp}(S)[W_{\mathbb{A}}^{-1}].\]
And a morphism in \(D^-_{Sp}(S)\) is called a levelwise \(\mathbb{A}^1\)-equivalence if it becomes an isomorphism in \(\widetilde{DM}^{eff,-}_{Sp}(S)\).
\end{definition}
\begin{definition}
(See \cite[5.3.20]{CD1}) A complex \(K\in D^-_{Sp}(S)\) is called levelwise \(\mathbb{A}^1\)-local if for every levelwise \(\mathbb{A}^1\)-equivalence \(f:A\longrightarrow B\), the induced map
\[Hom_{D^-_{Sp}(S)}(B,K)\longrightarrow Hom_{D^-_{Sp}(S)}(A,K)\]
is an isomorphism.
\end{definition}
\begin{proposition}
A complex \(K=(K_n)\in D^-_{Sp}(S)\) is levelwise \(\mathbb{A}^1\)-local if and only if for every \(n\geq 0\), the complex \(K_n\) is \(\mathbb{A}^1\)-local in \(D^-(S)\).
\end{proposition}
\begin{proof}
By the same proof as \cite[Lemma 9.20]{MVW}, \(K\) is levelwise \(\mathbb{A}^1\)-local if and only if for every \(X\in Sm/S\), \(n\in\mathbb{Z}\) and \(i\geq 0\), the map
\[Hom_{D^-_{Sp}(S)}((\Sigma^{\infty}\widetilde{\mathbb{Z}}_S(X))\{-i\}[n],K)\longrightarrow Hom_{D^-_{Sp}(S)}((\Sigma^{\infty}\widetilde{\mathbb{Z}}_S(X\times\mathbb{A}^1))\{-i\}[n],K)\]
is an isomorphism. And then one uses Proposition \ref{spderived adjunction} and Proposition \ref{sigma adjunction}.
\end{proof}

For every \(A=(A_n)\in Sp(S)\) and \(X\in Sm/S\), we define \(A^X\) by \((A_n)^X=(A_n^X)\). And the module structure \(\mathbbm{1}_S\{1\}\otimes^{\mathfrak{S}}A^X\longrightarrow A^X\) is given by the composition
\[\mathbbm{1}_S\{1\}\otimes^{\mathfrak{S}}A^X\longrightarrow(\mathbbm{1}_S\{1\}\otimes^{\mathfrak{S}}A)^X\longrightarrow A^X.\]
And \(A^X\) is contravariant with respect to morphisms in \(Sm/S\). So we could define \(C_*A\) by \((C_*A)_n=C_*A_n\).
\begin{proposition}\label{splocality}
Suppose \(K\in D^-_{Sp}(S)\).
\begin{enumerate}
\item The natural map \(K\longrightarrow C_*K\) is a levelwise \(\mathbb{A}^1\)-equivalence.
\item The complex \(C_*K\) is levelwise \(\mathbb{A}^1\)-local.
\item The functor \(C_*\) induces an endofunctor of \(D^-_{Sp}(S)\).
\end{enumerate}
\end{proposition}
\begin{proof}
\begin{enumerate}
\item We have a natural morphism \(\Sigma^{\infty}\widetilde{\mathbb{Z}}_S(X)\otimes^{\mathfrak{S}}A^X\longrightarrow A\) defined by the composition
\[\mathbbm{1}_S\{1\}^{\otimes p}\otimes_S\widetilde{\mathbb{Z}}_S(X)\otimes_SA_q^X\longrightarrow\widetilde{\mathbb{Z}}_S(X)\otimes_S(\mathbbm{1}_S\{1\}^{\otimes p}\otimes_SA_q)^X\longrightarrow\widetilde{\mathbb{Z}}_S(X)\otimes_SA_{p+q}^X\longrightarrow A_{p+q}\]
for every \(p, q\geq 0\). This morphism is compatible with module action so it induces an morphism
\[\Sigma^{\infty}\widetilde{\mathbb{Z}}_S(X)\otimes_SA^X\longrightarrow A.\]
Then we get a morphism
\[A^X\longrightarrow\underline{Hom}(\Sigma^{\infty}\widetilde{\mathbb{Z}}_S(X),A).\]

Then we could use the same proof as in \cite[Lemma 9.15]{MVW} to conclude.
\item By the proposition above and Proposition \ref{locality}.
\item By Proposition \ref{locality} since quasi-isomorphisms in \(D^-_{Sp}(S)\) are defined levelwise.
\end{enumerate}
\end{proof}
\begin{proposition}\label{levelwise equivalence}
A morphism \(f:A\longrightarrow B\) in \(D^-_{Sp}(S)\) is a levelwise \(\mathbb{A}^1\)-equivalence if and only if for every \(n\geq 0\), the morphism
\[f_n=\Omega^{\infty}(f\{n\}):A_n\longrightarrow B_n\]
is an \(\mathbb{A}^1\)-equivalence in \(D^-(S)\).
\end{proposition}
\begin{proof}
The morphism \(f\) is a levelwise \(\mathbb{A}^1\)-equivalence if and only if \(C_*f\) is a quasi-isomorphism by Proposition \ref{splocality}. And the latter property is levelwise.
\end{proof}
\begin{proposition}\label{spderived1}
Let \(\phi\) be the functor as before. We have an exact functor
\[\phi^*:\widetilde{DM}^{eff,-}_{Sp}(S)\longrightarrow\widetilde{DM}^{eff,-}_{Sp}(T)\]
which is determined by the following commutative diagram
\[
	\xymatrix
	{
		D^-_{Sp}(S)\ar[r]^{\phi^*}\ar[d]			&D^-_{Sp}(T)\ar[d]\\
		\widetilde{DM}^{eff,-}_{Sp}(S)\ar[r]^{\phi^*}	&\widetilde{DM}^{eff,-}_{Sp}(T)
	}
\]
\end{proposition}
\begin{proof} For any \(X\in Sm/S\), \(\phi^*\) maps
\[\Sigma^{\infty}(\widetilde{\mathbb{Z}}_S(X\times_k\mathbb{A}^1)\longrightarrow\widetilde{\mathbb{Z}}_S(X))\{-i\}\]
to
\[\Sigma^{\infty}(\widetilde{\mathbb{Z}}_T((\psi X)\times_k\mathbb{A}^1)\longrightarrow\widetilde{\mathbb{Z}}_T(\psi X))\{-i-j\}.\]
So the statement follows by the same method as in Proposition \ref{derived1}.
\end{proof}
\begin{proposition}\label{spderived2}
\begin{enumerate}
\item There is a tensor product
\[\otimes_S:\widetilde{DM}^{eff,-}_{Sp}(S)\times\widetilde{DM}^{eff,-}_{Sp}(S)\longrightarrow\widetilde{DM}^{eff,-}_{Sp}(S),\]
which is determined by the following commutative diagram
\[
	\xymatrix
	{
		D^-_{Sp}(S)\times D^-_{Sp}(S)\ar[r]^-{\otimes_S}\ar[d]						&D^-_{Sp}(S)\ar[d]\\
		\widetilde{DM}^{eff,-}_{Sp}(S)\times\widetilde{DM}^{eff,-}_{Sp}(S)\ar[r]^-{\otimes_S}	&\widetilde{DM}^{eff,-}_{Sp}(S)
	}.
\]
And for any \(K\in\widetilde{DM}^{eff,-}_{Sp}(S)\), the functor \(K\otimes_S-\) is exact.
\item Suppose \(f:S\longrightarrow T\) is a smooth morphism in \(Sm/k\). There is an exact functor
\[f_{\#}:\widetilde{DM}^{eff,-}_{Sp}(S)\longrightarrow\widetilde{DM}^{eff,-}_{Sp}(T),\]
which is determined by the following commutative diagram
\[
	\xymatrix
	{
		D^-_{Sp}(S)\ar[r]^{f_{\#}}\ar[d]			&D^-_{Sp}(T)\ar[d]\\
		\widetilde{DM}^{eff,-}_{Sp}(S)\ar[r]^{f_{\#}}	&\widetilde{DM}^{eff,-}_{Sp}(T)
	}.
\]
\item Suppose \(f:S\longrightarrow T\) is a morphism in \(Sm/k\). There is an exact functor
\[f^*:\widetilde{DM}^{eff,-}_{Sp}(T)\longrightarrow\widetilde{DM}^{eff,-}_{Sp}(S),\]
which is determined by the following commutative diagram
\[
	\xymatrix
	{
		D^-_{Sp}(T)\ar[r]^{f^*}\ar[d]			&D^-_{Sp}(S)\ar[d]\\
		\widetilde{DM}^{eff,-}_{Sp}(T)\ar[r]^{f^*}	&\widetilde{DM}^{eff,-}_{Sp}(S)
	}.
\]
\item Suppose \(i\geq 0\). There is an exact functor
\[-\{-i\}:\widetilde{DM}^{eff,-}_{Sp}(S)\longrightarrow\widetilde{DM}^{eff,-}_{Sp}(S),\]
which is determined by the following commutative diagram
\[
	\xymatrix
	{
		D^-_{Sp}(S)\ar[r]^{-\{-i\}}\ar[d]			&D^-_{Sp}(S)\ar[d]\\
		\widetilde{DM}^{eff,-}_{Sp}(S)\ar[r]^{-\{-i\}}	&\widetilde{DM}^{eff,-}_{Sp}(S)
	}.
\]
\item Suppose \(i\geq 0\). There is an exact functor
\[-\{i\}:\widetilde{DM}^{eff,-}_{Sp}(S)\longrightarrow\widetilde{DM}^{eff,-}_{Sp}(S),\]
which is determined by the following commutative diagram
\[
	\xymatrix
	{
		D^-_{Sp}(S)\ar[r]^{-\{i\}}\ar[d]			&D^-_{Sp}(S)\ar[d]\\
		\widetilde{DM}^{eff,-}_{Sp}(S)\ar[r]^{-\{i\}}	&\widetilde{DM}^{eff,-}_{Sp}(S)
	}.
\]
\end{enumerate}
\end{proposition}
\begin{proof}
For (1), (2), (3), take \(j=0\) in the definition of \(\phi\) and proceed as Proposition \ref{derived2} by using Proposition \ref{spderived1}. For (4), take the quadruple \((S,S,S,i)\) and use Proposition \ref{spderived1}. And (5) holds by Proposition \ref{levelwise equivalence}.
\end{proof}
\begin{proposition}\label{spjing}
\begin{enumerate}
\item Let \(f:S\longrightarrow T\) be a smooth morphism in \(Sm/k\). We have an adjoint pair
\[f_{\#}:\widetilde{DM}^{eff,-}_{Sp}(S)\rightleftharpoons\widetilde{DM}^{eff,-}_{Sp}(T):f^*.\]
\item We have an adjoint pair
\[-\{-i\}:\widetilde{DM}^{eff,-}_{Sp}(S)\rightleftharpoons\widetilde{DM}^{eff,-}_{Sp}(S):-\{i\}.\]
\end{enumerate}
\end{proposition}
\begin{proof}
The same as Proposition \ref{derived adjunction}.
\end{proof}
\begin{proposition}
\begin{enumerate}
\item There is an exact functor
\[\Sigma^{\infty}:\widetilde{DM}^{eff,-}(S)\longrightarrow\widetilde{DM}^{eff,-}_{Sp}(S)\]
determined by the following commutative diagram
\[
	\xymatrix
	{
		D^-(S)\ar[r]^{\Sigma^{\infty}}\ar[d]				&D^-_{Sp}(S)\ar[d]\\
		\widetilde{DM}^{eff,-}(S)\ar[r]^{\Sigma^{\infty}}	&\widetilde{DM}^{eff,-}_{Sp}(S)\\
	}
\]
\item There is an exact functor
\[\Omega^{\infty}:\widetilde{DM}^{eff,-}_{Sp}(S)\longrightarrow\widetilde{DM}^{eff,-}(S)\]
determined by the following commutative diagram
\[
	\xymatrix
	{
		D^-_{Sp}(S)\ar[r]^{\Omega^{\infty}}\ar[d]				&D^-(S)\ar[d]\\
		\widetilde{DM}^{eff,-}_{Sp}(S)\ar[r]^{\Omega^{\infty}}	&\widetilde{DM}^{eff,-}(S)\\
	}
\]
\end{enumerate}
\end{proposition}
\begin{proof}
\begin{enumerate}
\item This is because tensor products with \(\widetilde{\mathbb{Z}}_S(\mathbb{G}_m^{\wedge 1})\) preserves \(\mathbb{A}^1\)-equivalences and Proposition \ref{levelwise equivalence}.
\item This follows by Proposition \ref{levelwise equivalence}.
\end{enumerate}
\end{proof}
\begin{proposition}
There is an adjoint pair
\[\Sigma^{\infty}:\widetilde{DM}^{eff,-}(S)\rightleftharpoons\widetilde{DM}^{eff,-}_{Sp}(S):\Omega^{\infty}.\]
\end{proposition}
\begin{proof}
The same as Proposition \ref{derived adjunction}.
\end{proof}
\begin{proposition}\label{ff1}
The functor \(\Sigma^{\infty}:\widetilde{DM}^{eff,-}(S)\longrightarrow\widetilde{DM}^{eff,-}_{Sp}(S)\) is fully faithful.
\end{proposition}
\begin{proof}
The same as Proposition \ref{ff}.
\end{proof}
\begin{proposition}
\begin{enumerate}
\item We have a commutative diagram (up to a canonical isomorphism)
\[
	\xymatrix
	{
		\widetilde{DM}^{eff,-}(S)\times\widetilde{DM}^{eff,-}(S)\ar[r]^-{\otimes_S}\ar[d]_{\Sigma^{\infty}\times\Sigma^{\infty}}	&\widetilde{DM}^{eff,-}(S)\ar[d]_{\Sigma^{\infty}}\\
		\widetilde{DM}^{eff,-}_{Sp}(S)\times\widetilde{DM}^{eff,-}_{Sp}(S)\ar[r]^-{\otimes_S}								&\widetilde{DM}^{eff,-}_{Sp}(S)
	}.
\]
\item Suppose \(f:S\longrightarrow T\) is a morphism in \(Sm/k\). We have a commutative diagram (up to a canonical isomorphism)
\[
	\xymatrix
	{
		\widetilde{DM}^{eff,-}(T)\ar[r]^{f^*}\ar[d]_{\Sigma^{\infty}}	&\widetilde{DM}^{eff,-}(S)\ar[d]_{\Sigma^{\infty}}\\
		\widetilde{DM}^{eff,-}_{Sp}(T)\ar[r]^{f^*}					&\widetilde{DM}^{eff,-}_{Sp}(S)
	}
\]
\item Suppose \(f:S\longrightarrow T\) is a smooth morphism in \(Sm/k\). We have a commutative diagram (up to a canonical isomorphism)
\[
	\xymatrix
	{
		\widetilde{DM}^{eff,-}(S)\ar[r]^{f_{\#}}\ar[d]_{\Sigma^{\infty}}	&\widetilde{DM}^{eff,-}(T)\ar[d]_{\Sigma^{\infty}}\\
		\widetilde{DM}^{eff,-}_{Sp}(S)\ar[r]^{f_{\#}}					&\widetilde{DM}^{eff,-}_{Sp}(T)
	}
\]
\end{enumerate}
\end{proposition}
\begin{proof}
This follows by direct computations.
\end{proof}

In \cite[Proposition 3.5]{CD} and \cite[Proposition 5.2.16]{CD1}, they defined \(\widetilde{DM}^{eff}_{Sp}(S)\) as the the Verdier localization of \(D_{Sp}(S)\) with respect to homotopy invariant conditions. So the localization induces a triangulated structure on \(\widetilde{DM}^{eff}_{Sp}(S)\) (See \cite[Lemma 4.3.1]{Kra}). And this is the triangulated structure we will impose on \(\widetilde{DM}^{eff}_{Sp}(S)\).
\begin{proposition}\label{ff2}
There is a fully faithful exact functor \(\widetilde{DM}^{eff,-}_{Sp}(S)\longrightarrow\widetilde{DM}^{eff}_{Sp}(S)\) which is determined by the commutative diagram
\[
	\xymatrix
	{
		D^-_{Sp}(S)\ar[r]\ar[d]				&D_{Sp}(S)\ar[d]\\
		\widetilde{DM}^{eff,-}_{Sp}(S)\ar[r]	&\widetilde{DM}^{eff}_{Sp}(S)
	}.
\]
\end{proposition}
\begin{proof}
The same as Proposition \ref{triangulated1} by using Proposition \ref{splocality}.
\end{proof}
And there is also a compability result between the natural inclusion and \(\otimes_S\), \(f^*\), \(f_{\#}\), \(\Sigma^{\infty}\), \(-\{-i\},i\geq 0\) like Proposition \ref{-}.
\subsubsection{On Categories of Stabilized Motives}\label{stable}
\begin{definition}\label{st}
(See \cite[5.3.23]{CD1}) Define \(\mathscr{E}_{\Omega}\) to be the smallest thick subcategory of \(\widetilde{DM}^{eff,-}_{Sp}(S)\) such that
\begin{enumerate}
\item\(Cone((\Sigma^{\infty}\widetilde{\mathbb{Z}}_S(X)\{1\}\{-1\}\longrightarrow\Sigma^{\infty}\widetilde{\mathbb{Z}}_S(X))\{-i\})\in\mathscr{E}_{\Omega}\)
for every \(X\in Sm/S, i\geq 0\).
\item\(\mathscr{E}_{\Omega}\) is closed under arbitrary direct sums.
\end{enumerate}
Set \(W_{\Omega}\) be the class of morphisms in \(\widetilde{DM}^{eff,-}_{Sp}(S)\) whose cone is in \(\mathscr{E}_{\Omega}\). Define
\[\widetilde{DM}^-(S)=\widetilde{DM}^{eff,-}_{Sp}(S)[W_{\Omega}^{-1}]\]
to be the category of stabilized motives over \(S\). And a morphism in \(\widetilde{DM}^{eff,-}_{Sp}(S)\) is called a stable \(\mathbb{A}^1\)-equivalence if it becomes an isomorphism in \(\widetilde{DM}^-(S)\).
\end{definition}
\begin{definition}
A complex \(K\in\widetilde{DM}^{eff,-}_{Sp}(S)\) is called \(\Omega\)-local if for every stable \(\mathbb{A}^1\)-equivalence \(f:A\longrightarrow B\), the induced map
\[Hom_{\widetilde{DM}^{eff,-}_{Sp}(S)}(B,K)\longrightarrow Hom_{\widetilde{DM}^{eff,-}_{Sp}(S)}(A,K)\]
is an isomorphism.
\end{definition}
By the same method as in Proposition \ref{spderived2}, we have the following proposition
\begin{proposition}
\begin{enumerate}
\item There is a tensor product
\[\otimes_S:\widetilde{DM}^-(S)\times\widetilde{DM}^-(S)\longrightarrow\widetilde{DM}^-(S),\]
which is determined by the following commutative diagram
\[
	\xymatrix
	{
		\widetilde{DM}^{eff,-}_{Sp}(S)\times\widetilde{DM}^{eff,-}_{Sp}(S)\ar[r]^-{\otimes_S}\ar[d]	&\widetilde{DM}^{eff,-}_{Sp}(S)\ar[d]\\
		\widetilde{DM}^-(S)\times\widetilde{DM}^-(S)\ar[r]^-{\otimes_S}							&\widetilde{DM}^-(S)
	}.
\]
And for any \(K\in\widetilde{DM}^-(S)\), the functor \(K\otimes_S-\) is exact.
\item Suppose \(f:S\longrightarrow T\) is a smooth morphism in \(Sm/k\). There is an exact functor
\[f_{\#}:\widetilde{DM}^-(S)\longrightarrow\widetilde{DM}^-(T),\]
which is determined by the following commutative diagram
\[
	\xymatrix
	{
		\widetilde{DM}^{eff,-}_{Sp}(S)\ar[r]^{f_{\#}}\ar[d]	&\widetilde{DM}^{eff,-}_{Sp}(T)\ar[d]\\
		\widetilde{DM}^-(S)\ar[r]^{f_{\#}}				&\widetilde{DM}^-(T)
	}.
\]
\item Suppose \(f:S\longrightarrow T\) is a morphism in \(Sm/k\). There is an exact functor
\[f^*:\widetilde{DM}^-(T)\longrightarrow\widetilde{DM}^-(S),\]
which is determined by the following commutative diagram
\[
	\xymatrix
	{
		\widetilde{DM}^{eff,-}_{Sp}(T)\ar[r]^{f^*}\ar[d]	&\widetilde{DM}^{eff,-}_{Sp}(S)\ar[d]\\
		\widetilde{DM}^-(T)\ar[r]^{f^*}					&\widetilde{DM}^-(S)
	}.
\]
\item Suppose \(i\geq 0\). There is an exact functor
\[-\{-i\}:\widetilde{DM}^-(S)\longrightarrow\widetilde{DM}^-(S),\]
which is determined by the following commutative diagram
\[
	\xymatrix
	{
		\widetilde{DM}^{eff,-}_{Sp}(S)\ar[r]^{-\{-i\}}\ar[d]	&\widetilde{DM}^{eff,-}_{Sp}(S)\ar[d]\\
		\widetilde{DM}^-(S)\ar[r]^{-\{-i\}}				&\widetilde{DM}^-(S)
	}.
\]
\end{enumerate}
\end{proposition}
We denote by \(\Sigma^{\infty,st}\) the composition
\[\xymatrix{\widetilde{DM}^{eff,-}(S)\ar[r]^{\Sigma^{\infty}}&\widetilde{DM}^{eff,-}_{Sp}(S)\ar[r]&\widetilde{DM}^-(S)}.\]
\begin{lemma}
Let \(\mathscr{C}\) be a symmetric monoidal category and \(T\in\mathscr{C}\). If there is a \(U\in\mathscr{C}\) such that \(U\otimes T\cong\mathbbm{1}\), then there are isomorphisms
\[ev:U\otimes T\longrightarrow\mathbbm{1},coev:\mathbbm{1}\longrightarrow T\otimes U\]
such that \(T\) is strongly dualizable (See \cite[2.4.30]{CD1}) with the dual \(U\) under these two maps.
\end{lemma}
\begin{proof}
Let \(F=-\otimes U\) and \(G=-\otimes T\). Then the condition gives an endoequivalence
\[F:\mathscr{C}\rightleftharpoons\mathscr{C}:G\]
i.e. two natural isomorphisms \(a:FG\longrightarrow id\) and \(b:id\longrightarrow GF\). So we could construct the following two morphisms
\[\theta:\xymatrix{Hom(FX,Y)\ar[r]^-G&Hom(GFX,GY)\ar[r]^-{b^*}&Hom(X,GY)}\]
and
\[\eta:\xymatrix{Hom(X,GY)\ar[r]^-F&Hom(FX,FGY)\ar[r]^-{a_*}&Hom(FX,Y)}\]
for every \(X,Y\in\mathscr{C}\). Let \(\theta_1\) be the composition
\[\xymatrix{F\ar[r]^-{id_F\times b}&FGF\ar[r]^-{a\times id_F}&F}\]
and \(\theta_2\) be
\[\xymatrix{G\ar[r]^-{b\times id_G}&GFG\ar[r]^-{id_G\times a}&G}.\]
Then \((\eta\circ\theta)(f)=\theta_1(X)\circ f\) and \((\theta\circ\eta)(g)=g\circ\theta_2(Y)\). So \(\theta\) is an isomorphism, hence \(F\) is a left adjoint of \(G\) (vice versa).
\end{proof}
\begin{proposition}\label{invert}
The element \(\Sigma^{\infty,st}(\mathbbm{1}_S\{1\})\) has a strong dual \((\Sigma^{\infty,st}\mathbbm{1}_S)\{-1\}\) in \(\widetilde{DM}^-(S)\) with the evaluation and coevaluation maps being isomorphisms.
\end{proposition}
\begin{proof}
By Definition \ref{st} and the lemma above.
\end{proof}
Hence we define \(C(i)\) to be \(C\otimes_S\Sigma^{\infty,st}(\mathbbm{1}_S(i))\) and \(C(-i)\) to be \(C\{-i\}[i]\) for any \(C\in\widetilde{DM}^-\) and \(i\geq 0\).
\begin{proposition}\label{ff3}
(See \cite[Proposition 5.3.25]{CD1}) Suppose \(E=\widetilde{CH}\). The functor
\[\Sigma^{\infty,st}:\widetilde{DM}^{eff,-}(pt)\longrightarrow\widetilde{DM}^-(pt)\]
is fully faithful.
\end{proposition}
\begin{proof}
We first prove that for every projective \(E\in\widetilde{DM}^{eff,-}(pt)\), \(\Sigma^{\infty}E\in\widetilde{DM}^{eff,-}_{Sp}(pt)\) is \(\Omega\)-local. By the same method as in \cite[Lemma 9.20]{MVW}, this is equivalent to for any \(X\in Sm/S\), \(i\geq 0\) and \(n\in\mathbb{Z}\), the morphism
\[Hom(\Sigma^{\infty}\widetilde{\mathbb{Z}}_{pt}(X)\{1\}\{-1\}\{-i\},\Sigma^{\infty}(E[n]))\longrightarrow Hom(\Sigma^{\infty}\widetilde{\mathbb{Z}}_{pt}(X)\{-i\},\Sigma^{\infty}(E[n]))\]
is an isomorphism. And this follows by the following commutative diagram
\[
	\xymatrix
	{
		Hom(\Sigma^{\infty}\widetilde{\mathbb{Z}}_{pt}(X)\{1\}\{-1\}\{-i\},\Sigma^{\infty}(E[n]))\ar[r]\ar[d]	&Hom(\Sigma^{\infty}\widetilde{\mathbb{Z}}_{pt}(X)\{-i\},\Sigma^{\infty}(E[n]))\ar[d]\\
		Hom(\Sigma^{\infty}\widetilde{\mathbb{Z}}_{pt}(X)\{1\}\{-1\},\Sigma^{\infty}(E[n])\{i\})\ar[r]\ar[d]	&Hom(\Sigma^{\infty}\widetilde{\mathbb{Z}}_{pt}(X),\Sigma^{\infty}(E[n])\{i\})\ar[d]\ar[dl]_{-\{1\}}\\
		Hom(\Sigma^{\infty}\widetilde{\mathbb{Z}}_{pt}(X)\{1\},\Sigma^{\infty}(E[n])\{i+1\})\ar[d]			&Hom(\widetilde{\mathbb{Z}}_S(X),\mathbbm{1}_{pt}\{1\}^{\otimes i}\otimes(E[n]))\ar[dl]_{\mathbbm{1}_{pt}\{1\}\otimes-}\\
		Hom(\mathbbm{1}_{pt}\{1\}\otimes\widetilde{\mathbb{Z}}_{pt}(X),\mathbbm{1}_{pt}\{1\}^{\otimes i+1}\otimes(E[n]))	&
	}
\]
and \cite[Theorem 3.3.8]{DF}. Suppose \(K,L\in\widetilde{DM}^{eff,-}(pt)\) with projective resolutions \(P,Q\) respectively. The statement follows by the following commutative diagram
\[
	\xymatrix
	{
		Hom_{\widetilde{DM}^{eff,-}(pt)}(K,L)\ar[r]^-{\Sigma^{\infty,st}}\ar[d]_{\cong}				&Hom_{\widetilde{DM}^-(pt)}(\Sigma^{\infty,st}K,\Sigma^{\infty,st}L)\ar[d]_{\cong}\\
		Hom_{\widetilde{DM}^{eff,-}(pt)}(P,Q)\ar[r]^-{\Sigma^{\infty,st}}\ar[rd]^-{\Sigma^{\infty}}	&Hom_{\widetilde{DM}^-(pt)}(\Sigma^{\infty,st}P,\Sigma^{\infty,st}Q)\\
																				&Hom_{\widetilde{DM}^{eff,-}_{Sp}(pt)}(\Sigma^{\infty}P,\Sigma^{\infty}Q)\ar[u]^{\cong}
	}
\]
and Proposition \ref{ff1}.
\end{proof}
\begin{proposition}
Let \(f:S\longrightarrow T\) be a smooth morphism in \(Sm/k\). We have an adjoint pair
\[f_{\#}:\widetilde{DM}^-(S)\rightleftharpoons\widetilde{DM}^-(T):f^*.\]
\end{proposition}
\begin{proof}
The same as Proposition \ref{derived adjunction}.
\end{proof}
\begin{proposition}
Suppose \(f:S\longrightarrow T\) is a morphism in \(Sm/k\).
\begin{enumerate}
\item For any \(K, L\in\widetilde{DM}^-(T)\), we have
\[f^*(K\otimes_SL)\cong(f^*K)\otimes_S(f^*L).\]
\item If \(f\) is smooth, then for any \(K\in\widetilde{DM}^-(S)\) and \(L\in\widetilde{DM}^-(T)\), we have
\[f_{\#}(K\otimes_Sf^*L)\cong(f_{\#}K)\otimes_SL.\]
\end{enumerate}
\end{proposition}
\begin{proof}
This is because everything works termwise for spectra by discussion in Section \ref{symmetric spectra}.
\end{proof}

In \cite[Proposition 5.3.23]{CD1}, they defined \(\widetilde{DM}(S)\) as the the Verdier localization of \(\widetilde{DM}^{eff,-}_{Sp}(S)\) with respect to \(W_{\Omega}\). So the localization induces a triangulated structure on \(\widetilde{DM}(S)\) (See \cite[Lemma 4.3.1]{Kra}). And this is the triangulated structure we will impose on \(\widetilde{DM}(S)\).

Here is a weak result which is enough for our purpose:
\begin{proposition}\label{fff}
There is an exact functor \(\widetilde{DM}^-(S)\longrightarrow\widetilde{DM}(S)\) which is determined by the commutative diagram
\[
	\xymatrix
	{
		\widetilde{DM}^{eff,-}_{Sp}(S)\ar[r]\ar[d]	&\widetilde{DM}^{eff}_{Sp}(S)\ar[d]\\
		\widetilde{DM}^-(S)\ar[r]				&\widetilde{DM}(S)
	}.
\]
And when \(E=\widetilde{CH}\) and \(S=pt\), the morphism
\[Hom_{\widetilde{DM}^-(S)}(X,Y)\longrightarrow Hom_{\widetilde{DM}(S)}(L(X),L(Y))\]
is an isomorphism if \(X\) and \(Y\) are of the form \((\Sigma^{\infty,st}A)\{-i\}, i\geq 0\).
\end{proposition}
\begin{proof}
The functor is induced and exact by \cite[Proposition 4.6.2]{Kra}. We have thus a commutative diagram (up to a natural isomorphism)
\[
	\xymatrix
	{
		\widetilde{DM}^{eff,-}(S)\ar[r]\ar[d]_{\Sigma^{\infty,st}}	&\widetilde{DM}^{eff}(S)\ar[d]_{\Sigma^{\infty,st}}\\
		\widetilde{DM}^-(S)\ar[r]							&\widetilde{DM}(S)
	}.
\]

Now let \(E=\widetilde{CH}\) and \(S=pt\). We denote \(\mathscr{P}(X,Y)\) be the property that the statement holds for \(X\) and \(Y\). Then if \(\mathscr{P}(X,Y)\) is true, then for any \(X'\cong X\) and \(Y'\cong Y\), \(\mathscr{P}(X',Y')\) is also true. And by Proposition \ref{invert} and the fact that \(L\) is monoidal, \(\mathscr{P}(X\{-1\},Y\{-1\})\) is also true. 

By Proposition \ref{triangulated1}, Proposition \ref{ff3} and the diagram above, \(\mathscr{P}(\Sigma^{\infty,st}A,\Sigma^{\infty,st}B)\) is true for any \(A, B\in\widetilde{DM}^{eff,-}(pt)\). Hence the statement follows.
\end{proof}
And there is also a compability result between the natural inclusion and \(\otimes_S\), \(f^*\), \(f_{\#}\), \(-\{-i\},i\geq 0\) like Proposition \ref{-}.
\section{Orientations on Symplectic Bundles and Applications}\label{Orientations on Symplectic Bundles and Applications}
\subsection{Orientations on Symplectic Bundles}
In this section, we want to extend the quaternionic projective bundle theorem and Thom isomorphism in \cite{Y} to arbitrary smooth base \(S\). Recall the definition of orientable bundles from \cite[Definition 4.1]{Y}.
\begin{definition}\label{symplectic oriented}
Let \(E\) be a correspondence theory, we call it symplectic oriented if for every \(X\in Sm/k\) and rank two symplectic bundle \(E\) over \(X\), there is an isomorphism
\[s_E:E\longrightarrow 0\]
in \(\mathscr{P}_X\) such that
\begin{enumerate}
\item For every isomorphism \(\varphi:E_1\longrightarrow E_2\) between rank two symplectic bundles (preserving inner products), we have \(s_{E_1}=s_{E_2}\circ\varphi\).
\item For every \(f:X\longrightarrow Y\) in \(Sm/k\), we have \(s_{f^*E}=f^*(s_E)\).
\end{enumerate}
\end{definition}
It is then clear that MW-correspondence is symplectic oriented. From now on, we assume \(E=\widetilde{CH}\) until the end of this section.
\begin{definition}
Let \(\mathscr{E}\) be an orientable bundle over \(X\in Sm/S\) with an orientation \(s\) and rank \(n\), it has a map
\[e(\mathscr{E}):\widetilde{CH}^0(X)\longrightarrow\widetilde{CH}^n(X)\]
by \cite[Definition 4.2]{Y}. So we have an element
\[e(\mathscr{E})(1)\in Hom_{DM^{eff,-}(pt)}(\widetilde{\mathbb{Z}}_{pt}(X),\widetilde{\mathbb{Z}}_{pt}(n)[2n]).\]
It induces a morphism
\[\theta:\widetilde{\mathbb{Z}}_S(X)\longrightarrow\widetilde{\mathbb{Z}}_S(n)[2n]\]
by Proposition \ref{jing}, which is called the Euler class of \(\mathscr{E}\) over \(S\).

If \(n=2\), then \(-\theta\) is called the first Pontryagin class under the orientation \(s\) of \(\mathscr{E}\) over \(S\), which is denoted by \(p_1(\mathscr{E})\).
\end{definition}

Now let \(X\in Sm/S\). Suppose we have two morphisms \(f_i:\widetilde{\mathbb{Z}}_S(X)\longrightarrow C_i, i=1,2\) in \(\widetilde{DM}^{eff,-}(S)\), we denote by \(f_1\boxtimes f_2\) the composition
\[\xymatrix{\widetilde{\mathbb{Z}}_S(X)\ar[r]^-{\triangle}&\widetilde{\mathbb{Z}}_S(X)\otimes_S\widetilde{\mathbb{Z}}_S(X)\ar[r]^-{f_1\otimes f_2}&C_1\otimes_SC_2}.\]
Suppose we have two morphisms \(f_i:\widetilde{\mathbb{Z}}_S(X)\longrightarrow\widetilde{\mathbb{Z}}_S(n_i)[2n_i]\) in \(\widetilde{DM}^{eff,-}(S)\), we denote by \(f_1f_2\) the composition
\[\xymatrix{\widetilde{\mathbb{Z}}_S(X)\ar[r]^-{f_1\boxtimes f_2}&\widetilde{\mathbb{Z}}_S(n_1)[2n_1]\otimes\widetilde{\mathbb{Z}}_S(n_2)[2n_2]\ar[r]^-{\otimes}&\widetilde{\mathbb{Z}}_S(n_1+n_2)[2(n_1+n_2)]}.\]

The following proposition clarifies the relationship between MW-motivic cohomologies and Chow-Witt groups.
\begin{proposition}\label{base}
For any \(X\in Sm/k\), and \(i\geq 0\), we have an isomorphism (Set \(\widetilde{DM}=\widetilde{DM}^{eff,-}(pt)\))
\[Hom_{\widetilde{DM}}(\widetilde{\mathbb{Z}}_{pt}(X),\widetilde{\mathbb{Z}}_{pt}(i)[2i])\cong\widetilde{CH}^i(X)\]
which is compactible with pull-backs and sends \(id_{\mathbb{Z}_{pt}}\) to \(1\) when \(i=0\) and \(X=pt\).

Moreover, the following diagram commutes for any \(i,j\geq 0\)
\[
	\xymatrix@C=1.0em
	{
		Hom_{\widetilde{DM}}(\widetilde{\mathbb{Z}}_{pt}(X),\widetilde{\mathbb{Z}}_{pt}(i)[2i])\times Hom_{\widetilde{DM}}(\widetilde{\mathbb{Z}}_{pt}(X),\widetilde{\mathbb{Z}}_{pt}(j)[2j])\ar[r]\ar[d]^{\cdot}	&\widetilde{CH}^i(X)\times\widetilde{CH}^j(X)\ar[d]^{\cdot}\\
		Hom_{\widetilde{DM}}(\widetilde{\mathbb{Z}}_{pt}(X),\widetilde{\mathbb{Z}}_{pt}(i+j)[2(i+j)])\ar[r]															&\widetilde{CH}^{i+j}(X)
	}
\]
where the right-hand map is the intersection product on Chow-Witt groups.
\end{proposition}
\begin{proof}
See \cite[Corollary 4.2.6]{DF}.
\end{proof}

In \cite[Definition 3.5]{Y}, we defined \(HP^n\) (Originally defined in \cite[Section 3]{PW}) as an open set of \(Gr(2,2n+2)\) consisting of those two dimensional subspaces \(V\subseteq k^{\oplus 2n+2}\) such that the trivial symplectic form on \(k^{\oplus 2n+2}\) does not vanish on \(V\). We denote \(HP^n\times_kS\) by \(HP^n_S\). And for any symplectic bundle (See \cite[Definition 3.3]{Y}) \(\mathscr{E}\) of rank two over \(X\in Sm/S\), there is a canonical orientation \(O_X\longrightarrow det(\mathscr{E}^{\vee})\) induced by the inner product \(\mathscr{E}\times\mathscr{E}\longrightarrow O_X\) (See \cite[Definition 4.3]{Y}). And we have a dual tautological bundle (See discussion after \cite[Definition 3.6]{Y}) \(\mathscr{U}^{\vee}_S\) over \(HP^n_S\) satisfying \(\mathscr{U}^{\vee}_S=p^*\mathscr{U}^{\vee}_{pt}\) as symplectic bundles where \(p:HP^n_S\longrightarrow HP^n\) is the projection map.
\begin{theorem}\label{splitting1}
The map
\[\xymatrixcolsep{5pc}\xymatrix{\widetilde{\mathbb{Z}}_S(HP^n_S)\ar[r]^-{p_1(\mathscr{U}^{\vee}_S)^i}&\oplus_{i=0}^{n}\widetilde{\mathbb{Z}}_S(2i)[4i]}\]
is an isomorphism in \(\widetilde{DM}^{eff,-}(S)\). Here, $\mathscr{U}^{\vee}$ is endowed with its canonical orientation.
\end{theorem}
\begin{proof}
We have the projection \(p:HP^n_S\longrightarrow HP^n\) as before. Now we have a commutative diagram
\[
	\xymatrixcolsep{4pc}
	\xymatrix
	{
		p^*\widetilde{\mathbb{Z}}_{pt}(HP^n)\ar[r]^-{p^*(p_1(\mathscr{U}^{\vee}_{pt}))}		&p^*\widetilde{\mathbb{Z}}_{pt}(2)[4]\\
		\widetilde{\mathbb{Z}}_S(HP^n_S)\ar[u]^-{\cong}\ar[r]^-{p_1(\mathscr{U}^{\vee}_S)}	&\widetilde{\mathbb{Z}}_S(2)[4]\ar[u]^{\cong}
	}.
\]
Hence the result follows by the commutative diagram
\[
	\xymatrixcolsep{5pc}
	\xymatrix
	{
		p^*\widetilde{\mathbb{Z}}_{pt}(HP^n)\ar[r]^-{p^*(p_1(\mathscr{U}^{\vee}_{pt}))^i}	&\oplus_{i=0}^{n}p^*\widetilde{\mathbb{Z}}_{pt}(2i)[4i]\\
		\widetilde{\mathbb{Z}}_S(HP^n_S)\ar[r]^-{p_1(\mathscr{U}^{\vee}_S)^i}\ar[u]^{\cong}	&\oplus_{i=0}^{n}\widetilde{\mathbb{Z}}_S(2i)[4i]\ar[u]^{\cong}
	},
\]
where the upper horizontal arrow is an isomorphism by \cite[Theorem 4.2]{Y}.
\end{proof}

Recall in \cite[Proposition 3.4]{Y}, we defined \(HGr_X(\mathscr{E})\) for any symplectic bundle \(\mathscr{E}\) over \(X\). It parameterizes rank two symplectic subbundles of \(\mathscr{E}\). When \(\mathscr{E}\) is the trivial symplectic bundle \(\left(O_{X}^{\oplus 2n+2},\left(\begin{array}{cc}&I\\-I&\end{array}\right)\right)\), \(HGr_{X}(\mathscr{E})\cong HP^{n}\times_{k}X\) over \(X\). It associates a dual tautological bundle \(\mathscr{U}^{\vee}\), which is a rank two symplectic bundle.
\begin{theorem}\label{quaternionic projective bundle theorem}
Let \(X\in Sm/S\) and let \((\mathscr{E},m)\) be a symplectic vector bundle of rank \(2n+2\) on \(X\). Let $\pi:HGr_X(\mathscr{E})\to X$ be the projection. Then, the map
\[\xymatrixcolsep{6pc}\xymatrix{\widetilde{\mathbb{Z}}_S(HGr_X(\mathscr{E}))\ar[r]^-{\pi\boxtimes p_1(\mathscr{U}^{\vee})^i}&\oplus_{i=0}^{n}\widetilde{\mathbb{Z}}_S(X)(2i)[4i]}\]
is an isomorphism in \(\widetilde{DM}^{eff,-}(S)\), functorial for $X$ in \(Sm/k\). Here, $\mathscr{U}^{\vee}$ is endowed with its canonical orientation.
\end{theorem}
\begin{proof}
The same as \cite[Theorem 4.3]{Y} since we have MV-sequences (Proposition \ref{Cech}) and Theorem \ref{splitting1} over any base \(S\).
\end{proof}
\begin{definition}
Let \(X\in Sm/S\) and \(Y\subseteq X\) be a closed subset. Consider the quotient sheaf with \(E\)-transfers
\[\widetilde{M}_Y(X):=\widetilde{\mathbb{Z}}_S(X)/\widetilde{\mathbb{Z}}_S(X\setminus Y).\]
Its image in $\widetilde{DM}^{eff,-}(S)$ will be called the relative motive of \(X\) with support in \(Y\) (see \cite[Definition 2.2]{D} and the remark before \cite[Corollary 5.3]{SV}). By abuse of notation, we still denote it by $\widetilde{M}_Y(X)$.
\end{definition}
\begin{definition}
Suppose \(X\in Sm/S\) and \(E\) is a vector bundle over \(X\). Define \(Th_S(E)=\widetilde{M}_X(E)\) where \(X\subseteq E\) is the zero section of \(E\).
\end{definition}
\begin{proposition}\label{additivity}
\begin{enumerate}
\item Suppose \(f:S\longrightarrow T\) is a morphism in \(Sm/k\), \(X\in Sm/T\) and \(E\) is a vector bundle over \(X\). Then we have
\[f^*Th_T(E)\cong Th_S(f^*E)\]
in \(\widetilde{DM}^{eff,-}(S)\), where \(f^*E\) is a vector bundle over \(X^S\).
\item Suppose \(f:S\longrightarrow T\) is a smooth morphism in \(Sm/k\), \(X\in Sm/S\) and \(E\) is a vector bundle over \(X\). Then we have
\[f_{\#}Th_S(E)\cong Th_T(E)\]
in \(\widetilde{DM}^{eff,-}(S)\).
\item (See \cite[Remark 2.4.15]{CD1}) Suppose \(E_1\) and \(E_2\) are vector bundles over \(X\in Sm/k\). Then
\[Th_X(E_1)\otimes_XTh_X(E_2)\cong Th_X(E_1\oplus E_2)\]
in \(\widetilde{DM}^{eff,-}(X)\).
\end{enumerate}
\end{proposition}
\begin{proof}
(1) and (2) are easy. Let's prove (3). The total space of \(E_1\oplus E_2\) is just \(E_1\times_XE_2\). By definition, for any vector bundle \(E\) over \(X\), \(Th_X(E)\) is just the complex
\[\widetilde{\mathbb{Z}}_S(E\setminus X)\longrightarrow\widetilde{\mathbb{Z}}_S(E).\]
Hence the left hand side is the total complex
\[\widetilde{\mathbb{Z}}_S((E_1\setminus X)\times_X(E_2\setminus X))\longrightarrow\widetilde{\mathbb{Z}}_S((E_1\setminus X)\times_XE_2)\oplus\widetilde{\mathbb{Z}}_S(E_1\times_X(E_2\setminus X))\longrightarrow\widetilde{\mathbb{Z}}_S(E_1\times_XE_2).\]
And by Proposition \ref{Cech}, the complex
\[\widetilde{\mathbb{Z}}_S((E_1\setminus X)\times_X(E_2\setminus X))\longrightarrow\widetilde{\mathbb{Z}}_S((E_1\setminus X)\times_XE_2)\oplus\widetilde{\mathbb{Z}}_S(E_1\times_X(E_2\setminus X))\]
is quasi-isomorphic to
\[0\longrightarrow\widetilde{\mathbb{Z}}_S((E_1\times_XE_2)\setminus X)\]
since
\[(E_1\times_XE_2)\setminus X=(E_1\setminus X)\times_XE_2\cup E_1\times_X(E_2\setminus X).\]
Hence we have a quasi-isomorphism
\[
	\xymatrixcolsep{1pc}
	\xymatrix
	{
		\widetilde{\mathbb{Z}}_S((E_1\setminus X)\times(E_2\setminus X))\ar[r]\ar[d]	&\widetilde{\mathbb{Z}}_S((E_1\setminus X)\times E_2)\oplus\widetilde{\mathbb{Z}}_S(E_1\times(E_2\setminus X))\ar[r]\ar[d]	&\widetilde{\mathbb{Z}}_S(E_1\times E_2)\ar@{=}[d]\\
		0\ar[r]														&\widetilde{\mathbb{Z}}_S((E_1\times E_2)\setminus X)\ar[r]													&\widetilde{\mathbb{Z}}_S(E_1\times E_2)
	}.
\]
\end{proof}
\begin{proposition}\label{excision}
Let \(f:X\longrightarrow Y\) be an \'etale morphism in \(Sm/S\), \(Z\subseteq Y\) be a closed subset of \(Y\) such that the map \(f:f^{-1}(Z)\longrightarrow Z\) is an isomorphism (the schemes are endowed with their reduced structure), then the map \(\widetilde{M}_{f^{-1}(Z)}(X)\longrightarrow\widetilde{M}_Z(Y)\) is an isomorphism of sheaves with \(E\)-transfers.
\end{proposition}
\begin{proof}
The same as \cite[Proposition 5.1]{Y} by using Proposition \ref{Cech}.
\end{proof}
\begin{theorem}\label{Thom}
Let \(E\) be a symplectic bundle of rank \(2n\) over \(X\in Sm/S\). Then we have
\[Th_S(E)\cong\widetilde{\mathbb{Z}}_S(X)(2n)[4n]\]
in \(\widetilde{DM}^{eff,-}(S)\).
\end{theorem}
\begin{proof}
The same as \cite[Theorem 5.1]{Y} by using Theorem \ref{quaternionic projective bundle theorem}.
\end{proof}

The following observation is quite interesting:
\begin{proposition}
Let \(E\) be a vector bundle of rank \(n\) over \(X\in Sm/k\). Then we have
\[(\Sigma^{\infty,st}Th_X(E))^{-1}\cong(\Sigma^{\infty,st}Th_X(E^{\vee}))(-2n)[-4n]\]
in \(\widetilde{DM}^-(X)\).
\end{proposition}
\begin{proof}
By Proposition \ref{additivity} and Theorem \ref{Thom}, we have
\[Th_X(E)\otimes_XTh_X(E^{\vee})\cong Th_X(E\oplus E^{\vee})\cong\mathbbm{1}_X(2n)[4n]\]
in \(\widetilde{DM}^{eff,-}(X)\). Now the statement follows from Proposition \ref{invert} and the fact that \(\Sigma^{\infty,st}\) is monoidal.
\end{proof}
Since we have a monoidal exact functor \(L:\widetilde{DM}^{eff,-}(X)\longrightarrow\widetilde{DM}^{eff}(X)\), by the same proof as above, we also have
\begin{proposition}\label{symplectic orientation}
Let \(E\) be a vector bundle of rank \(n\) over \(X\in Sm/k\). Then we have
\[(\Sigma^{\infty,st}Th_X(E))^{-1}\cong(\Sigma^{\infty,st}Th_X(E^{\vee}))(-2n)[-4n]\]
in \(\widetilde{DM}(X)\).
\end{proposition}
\subsection{Duality for Proper Schemes and Embedding Theorem}
In this section, we are going to prove \(\widetilde{\mathbb{Z}}_{pt}(X)\) is strongly dualizable in \(\widetilde{DM}^-(pt)\) for proper \(X\in Sm/k\). And we calculate its dual by using orientations on symplectic bundles. And finally we will prove the embedding theorem in MW-motivic cohomology. For this we need to involve the stable \(\mathbb{A}^1\)-derived category \(D_{\mathbb{A}^1}(S)\) over \(S\) introduced in \cite[Section 1]{DF1} and \cite[Example 5.3.31]{CD1} and use the duality result on that category. For clarity, we describe our procedure like the following:
\[\textrm{Duality in }D_{\mathbb{A}^1}(S)\Longrightarrow\textrm{Duality in }\widetilde{DM}(S)\Longrightarrow\textrm{Duality in }\widetilde{DM}^-(S)\Longrightarrow\textrm{Embedding Theorem.}\]

Let's briefly review the construction of \(D_{\mathbb{A}^1}(S)\), the reader may also refer to \cite[Section 5]{CD1} and \cite[Section 1]{DF1}.

Define \(Sh(S)\) to be category of Nisnevich sheaves of abelian groups on \(Sm/S\). The Yoneda representative of \(F\longmapsto F(X)\) for any \(X\in Sm/S\) is denoted by \(\mathbb{Z}_S(X)\). The functor \(\gamma:Sm/S\longrightarrow\widetilde{Cor}_S\) in Proposition \ref{functor} and Lemma \ref{adjunction} give us an adjunction
\[\widetilde{\gamma}^*:Sh(S)\rightleftharpoons\widetilde{Sh}(S):\widetilde{\gamma}_*.\]
The category \(Sh(S)\) is a symmetric monoidal category with \(\mathbb{Z}_S(X)\otimes_S\mathbb{Z}_S(Y)\cong\mathbb{Z}_S(X\times_SY)\) and \(\gamma^*\) is a monoidal functor. For any \(f:S\longrightarrow T\) in \(Sm/k\), there is an adjunction
\[f^*:Sh(T)\rightleftharpoons Sh(S):f_*\]
by the same method as in Proposition \ref{pullback} and \(f^*\gamma^*\cong\gamma^*f^*\) since there is an similar equality for their right adjoints. If \(f\) is smooth, there is an adjunction
\[f_{\#}:Sh(S)\rightleftharpoons Sh(T):f^*\]
as in Proposition \ref{lowerhash}. Then \(f_{\#}\gamma^*\cong\gamma^*f_{\#}\) since there is an similar equality for their right adjoints.

By the same method as in Section \ref{symmetric spectra}, we define \(SSp(S)\) to be the category of symmetric \(\mathbbm{1}_S\{1\}\)-spectra by \(Sh(S)\), where
\[\mathbbm{1}_S\{1\}=Coker(\mathbb{Z}_S(S)\longrightarrow\mathbb{Z}_S(\mathbb{G}_m)).\]
There are adjunctions
\[\Sigma^{\infty}:Sh(S)\rightleftharpoons SSp(S):\Omega^{\infty}\]
and
\[\widetilde{\gamma}^*:SSp(S)\rightleftharpoons Sp(S):\widetilde{\gamma}_*.\]
And we could also define \(\otimes_S\), \(f^*\), \(f_*\), \(f_{\#}\), \(-\{-i\}\) and \(-\{i\}\) (\(i\geq 0\)) on \(SSp(S)\). Moreover, \(\gamma^*\) commutes with \(f^*\) and \(f_{\#}\) and it's monoidal as above.

In \cite[Theorem 1.7]{CD}, they defined a model structure \(\mathfrak{M}_{S}\) on the category of unbounded complexes of \(Sh(S)\). This is a cofibrantly generated model structure where the cofibrations are those \(I\)-cofibrations where \(I\) consists of the morphisms \(S^{n+1}(\mathbb{Z}_S(X))\longrightarrow D^n(\mathbb{Z}_S(X))\) for any \(X\in Sm/S\) and weak equivalences are quasi-morphisms between complexes. The homotopy category of \(\mathfrak{M}_{S}\) is denoted by \(D_{S}(S)\). Moreover, \(\mathfrak{M}_{S}\) is stable and left proper so it induces a triangulated structure on \(D_{S}(S)\).

By localizing \(D_{S}(S)\) with morphisms
\[\mathbb{Z}_S(X\times_k\mathbb{A}^1)\longrightarrow\mathbb{Z}_S(X)\]
as in Section \ref{effective}, we get a category \(D_{\mathbb{A}^1}^{eff}(S)\) with the induced triangulated structure.

In \cite[Theorem 1.7]{CD}, they defined a model structure \(\mathfrak{M}_{SSp}\) on the category of unbounded complexes of symmetric spectra of \(Sh(S)\). This is a cofibrantly generated model structure where the cofibrations are those \(I\)-cofibrations where \(I\) consists of the morphisms \(S^{n+1}(\Sigma^{\infty}\mathbb{Z}_S(X)\{-i\})\longrightarrow D^n(\Sigma^{\infty}\widetilde{\mathbb{Z}}_S(X)\{-i\})\) for any \(X\in Sm/S\) and \(i\geq 0\) and weak equivalences are quasi-morphisms between complexes. The homotopy category of \(\mathfrak{M}_{SSp}\) is denoted by \(D_{SSp}(S)\). Moreover, \(\mathfrak{M}_{SSp}\) is stable and left proper so it induces a triangulated structure on \(D_{SSp}(S)\).

By localizing \(D_{SSp}(S)\) with morphisms
\[(\Sigma^{\infty}\mathbb{Z}_S(X\times_k\mathbb{A}^1)\longrightarrow\Sigma^{\infty}\mathbb{Z}_S(X))\{-i\},i\geq 0\]
as in Section \ref{effective1}, we get a category with the induced triangulated structure. And localizing that category by
\[(\Sigma^{\infty}\mathbb{Z}_S(X)\{1\}\{-1\}\longrightarrow\Sigma^{\infty}\mathbb{Z}_S(X))\{-i\}\]
as in Section \ref{effective1}, we've got our category \(D_{\mathbb{A}^1}(S)\), with the induced triangulated structure. And we have an exact functor
\[\Sigma^{\infty,st}:D_{\mathbb{A}^1}^{eff}(S)\longrightarrow D_{\mathbb{A}^1}(S).\]

Here are some further properties needed of the stable \(\mathbb{A}^1\)-derived categories, they come from \cite[1.1.7 and Theorem 1.1.10]{DF1}.
\begin{proposition}\label{!}
\begin{enumerate}
\item For any \(f:S\longrightarrow T\) in \(Sm/k\), we have an adjoint pair of exact functors
\[f^*:D_{\mathbb{A}^1}(T)\rightleftharpoons D_{\mathbb{A}^1}(S):f_*.\]
\item For any smooth \(f:S\longrightarrow T\) in \(Sm/k\), we have an adjoint pair of exact functors
\[f_{\#}:D_{\mathbb{A}^1}(S)\rightleftharpoons D_{\mathbb{A}^1}(T):f^*.\]
And for any \(A\in D_{\mathbb{A}^1}(S)\) and \(B\in D_{\mathbb{A}^1}(T)\), we have
\[(f_{\#}A)\otimes B\cong f_{\#}(A\otimes f^*B).\]
\item For any \(f:S\longrightarrow T\) in \(Sm/k\), we have a functor
\[f_!:D_{\mathbb{A}^1}(S)\longrightarrow D_{\mathbb{A}^1}(T).\]
If \(f\) is proper, we have
\[f_!\cong f_*.\]
If \(f\) is smooth, we have
\[f_!\cong f_{\#}(-\otimes(\Sigma^{\infty,st}Th_S(T_{S/T}))^{-1}).\]
\end{enumerate}
\end{proposition}
\begin{proposition}\label{duality}
Let \(S\in Sm/k\) and \(f:X\longrightarrow S\) be a smooth proper morphism. Then \(\Sigma^{\infty,st}\mathbb{Z}_S(X)\in D_{\mathbb{A}^1}(S)\) is strongly dualizable with dual \(f_{\#}(\Sigma^{\infty,st}Th_X(T_{X/S})^{-1})\).
\end{proposition}
\begin{proof}
Pick any \(A, B\in D_{\mathbb{A}^1}(S)\), we have
\begin{align*}
	&Hom_{D_{\mathbb{A}^1}(S)}(\Sigma^{\infty,st}\mathbb{Z}_S(X)\otimes_SA,B)\\
\cong	&Hom_{D_{\mathbb{A}^1}(S)}(f_{\#}f^*A,B)\\
	&\textrm{by Proposition \ref{!}, (2)}\\
\cong	&Hom_{D_{\mathbb{A}^1}(S)}(A,f_*f^*B)\\
	&\textrm{by Proposition \ref{!}, (1) and (2)}\\
\cong	&Hom_{D_{\mathbb{A}^1}(S)}(A,f_!f^*B)\\
	&\textrm{by Proposition \ref{!}, (3)}\\
\cong	&Hom_{D_{\mathbb{A}^1}(S)}(A,f_{\#}(f^*B\otimes_X(\Sigma^{\infty,st}Th_X(T_{X/S}))^{-1}))\\
	&\textrm{by Proposition \ref{!}, (3)}\\
\cong	&Hom_{D_{\mathbb{A}^1}(S)}(A,B\otimes_Sf_{\#}(\Sigma^{\infty,st}Th_X(T_{X/S})^{-1}))\\
	&\textrm{by Proposition \ref{!}, (2)}.
\end{align*}.
\end{proof}
\begin{proposition}
Let \(S\in Sm/k\) and \(f:X\longrightarrow S\) be a smooth proper morphism. Then \(\Sigma^{\infty,st}\widetilde{\mathbb{Z}}_S(X)\in\widetilde{DM}(S)\) is strongly dualizable with dual
\[(\Sigma^{\infty,st}Th_S(\Omega_{X/S}))(-2d)[-4d],\]
where \(d=d_X-d_S\).
\end{proposition}
\begin{proof}
Since we have a monoidal exact functor \(\gamma^*:D_{\mathbb{A}^1}(S)\longrightarrow\widetilde{DM}(S)\) which commutes with \(f_{\#}\) up to a natural isomorphism, \(\Sigma^{\infty,st}\widetilde{\mathbb{Z}}_S(X)\in\widetilde{DM}(S)\) is strongly dualizable with dual \(f_{\#}(\Sigma^{\infty,st}Th_X(T_{X/Y})^{-1})\) by Proposition \ref{duality}. And by Proposition \ref{symplectic orientation},
\[(\Sigma^{\infty,st}Th_X(T_{X/S}))^{-1}\cong(\Sigma^{\infty,st}Th_X(\Omega_{X/S}))(-2d)[-4d].\]
Finally, we have
\[f_{\#}((\Sigma^{\infty,st}Th_X(\Omega_{X/S}))(-2d)[-4d])\cong(\Sigma^{\infty,st}Th_S(\Omega_{X/S}))(-2d)[-4d].\]
\end{proof}
Now we have a monoidal exact functor \(L:\widetilde{DM}^-(pt)\longrightarrow\widetilde{DM}(pt)\) which commutes with \(-\{-i\}, i\geq 0\) up to a natural isomorphism. And by Proposition \ref{fff}, we have
\begin{proposition}\label{duality1}
Let \(X\in Sm/k\) be a proper scheme. Then \(\Sigma^{\infty,st}\widetilde{\mathbb{Z}}_{pt}(X)\in\widetilde{DM}^-(pt)\) is strongly dualizable with dual
\[(\Sigma^{\infty,st}Th_{pt}(\Omega_{X/k}))(-2d_X)[-4d_X].\]
\end{proposition}
\begin{theorem}\label{embedding}
Let \(X, Y\in Sm/k\) with \(Y\) proper, then we have
\[Hom_{\widetilde{DM}^{eff,-}(pt)}(\widetilde{\mathbb{Z}}_{pt}(X),\widetilde{\mathbb{Z}}_{pt}(Y))\cong\widetilde{CH}^{d_Y}(X\times Y,-T_{X\times Y/X}).\]
\end{theorem}
\begin{proof}
Let \(p:Y\longrightarrow pt\) be the structure map of \(Y\) and \(q:X\times Y\longrightarrow Y\) be the second projection. We have
\begin{align*}
	&Hom_{\widetilde{DM}^{eff,-}(pt)}(\widetilde{\mathbb{Z}}_{pt}(X),\widetilde{\mathbb{Z}}_{pt}(Y))\\
\cong	&Hom_{\widetilde{DM}^-(pt)}(\Sigma^{\infty,st}\widetilde{\mathbb{Z}}_{pt}(X),\Sigma^{\infty,st}\widetilde{\mathbb{Z}}_{pt}(Y))\\
	&\textrm{by Proposition \ref{ff3}}\\
\cong	&Hom_{\widetilde{DM}^-(pt)}(\Sigma^{\infty,st}\widetilde{\mathbb{Z}}_{pt}(X)\otimes(\Sigma^{\infty,st}Th_{pt}(\Omega_{Y/k}))(-2d_Y)[-4d_Y],\Sigma^{\infty,st}\mathbbm{1}_{pt})\\
	&\textrm{by Proposition \ref{duality1}}\\
\cong	&Hom_{\widetilde{DM}^-(pt)}(\Sigma^{\infty,st}\widetilde{\mathbb{Z}}_{pt}(X)\otimes(\Sigma^{\infty,st}Th_{pt}(\Omega_{Y/k})),\Sigma^{\infty,st}\mathbbm{1}_{pt}(2d_Y)[4d_Y])\\
	&\textrm{by Proposition \ref{invert}}\\
\cong	&Hom_{\widetilde{DM}^-(pt)}(\Sigma^{\infty,st}(\widetilde{\mathbb{Z}}_{pt}(X)\otimes Th_{pt}(\Omega_{Y/k})),\Sigma^{\infty,st}\mathbbm{1}_{pt}(2d_Y)[4d_Y])\\
\cong	&Hom_{\widetilde{DM}^{eff,-}(pt)}(\widetilde{\mathbb{Z}}_{pt}(X)\otimes Th_{pt}(\Omega_{Y/k}),\mathbbm{1}_{pt}(2d_Y)[4d_Y])\\
	&\textrm{by Proposition \ref{ff3}}\\
\cong	&Hom_{\widetilde{DM}^{eff,-}(pt)}(\widetilde{\mathbb{Z}}_{pt}(X)\otimes p_{\#}Th_Y(\Omega_{Y/k}),\mathbbm{1}_{pt}(2d_Y)[4d_Y])\\
	&\textrm{by Proposition \ref{additivity}}\\
\cong	&Hom_{\widetilde{DM}^{eff,-}(pt)}(p_{\#}(p^*\widetilde{\mathbb{Z}}_{pt}(X)\otimes Th_Y(\Omega_{Y/k})),\mathbbm{1}_{pt}(2d_Y)[4d_Y])\\
	&\textrm{by Proposition \ref{operations1}}\\
\cong	&Hom_{\widetilde{DM}^{eff,-}(pt)}(p_{\#}(\widetilde{\mathbb{Z}}_Y(X\times Y)\otimes Th_Y(\Omega_{Y/k})),\mathbbm{1}_{pt}(2d_Y)[4d_Y])\\
\cong	&Hom_{\widetilde{DM}^{eff,-}(pt)}(p_{\#}(q_{\#}(\mathbbm{1}_{X\times Y})\otimes Th_Y(\Omega_{Y/k})),\mathbbm{1}_{pt}(2d_Y)[4d_Y])\\
\cong	&Hom_{\widetilde{DM}^{eff,-}(pt)}(p_{\#}(q_{\#}q^*Th_Y(\Omega_{Y/k})),\mathbbm{1}_{pt}(2d_Y)[4d_Y])\\
	&\textrm{by Proposition \ref{operations1}}\\
\cong	&Hom_{\widetilde{DM}^{eff,-}(pt)}(Th_{pt}(\Omega_{X\times Y/X}),\mathbbm{1}_{pt}(2d_Y)[4d_Y])\\
	&\textrm{by Proposition \ref{additivity}}\\
\cong	&\widetilde{CH}^{d_Y}(X\times Y,-T_{X\times Y/X})\\
	&\textrm{by discussion after \cite[Remark 4.2.7]{DF}}.
\end{align*}
\end{proof}

Now we could define a category of \(\widetilde{Cor}_k'\) to be a category with objects being proper schemes in \(Sm/k\) and
\[Hom_{\widetilde{Cor}_k'}(X,Y)=\widetilde{CH}^{d_Y}(X\times Y,-T_{X\times Y/X}).\]
It's a category because by the same proof as in Proposition \ref{associativity} and Proposition \ref{identity}.
\begin{definition}\label{effective Chow-Witt}
Define the category of effective Chow-Witt motives \(\widetilde{CH}^{eff}\) by the idempotent completion (See \cite[Definition 1.2]{BS}) of the opposite category of \(\widetilde{Cor}_k'\).
\end{definition}
\begin{proposition}
The opposite category of \(\widetilde{CH}^{eff}\) is a full-subcategory of \(\widetilde{DM}^{eff,-}(pt)\).
\end{proposition}
\begin{proof}
We have a fully faithful functor
\[\widetilde{Cor}_k'\longrightarrow\widetilde{DM}^{eff,-}(pt)\]
by Theorem \ref{embedding}. The category \(\widetilde{DM}^{eff}(pt)\) has infinite direct sums so it's idempotent complete (See \cite[Proposition 1.6.8]{N}). Now suppose \(g:K\longrightarrow K\) is an idempotent in \(\widetilde{DM}^{eff,-}(pt)\). It splits in \(\widetilde{DM}^{eff}(pt)\) with the image morphism \(f:K\longrightarrow I\) and the section \(s:I\longrightarrow K\). Then \(C_*(g)\) also splits with the image morphism \(C_*(f)\) and the section \(C_*(s)\) by \cite[Corollary 3.2.11]{DF}. Now \(f\) and \(s\) comes from \(D(pt)\) by \textit{loc. cit.}, hence \(C_*(I)\) is isomorphic to an object in \(D^-(pt)\). Hence \(I\) is isomorphic to an object in \(\widetilde{DM}^{eff,-}(pt)\). Hence \(\widetilde{DM}^{eff,-}(pt)\) is idempotent complete thus we get a functor (See \cite[Proposition 1.3]{BS})
\[\widetilde{CH}^{eff}\longrightarrow\widetilde{DM}^{eff,-}(pt)\]
which is also fully faithful by the construction of idempotent completion.
\end{proof}
\textbf{Acknowledgements. }The author would like to thank his PhD advisor J. Fasel for giving me this subject and helping during the subsequent research, and F. D\'eglise for helpful discussions.\\


\begin{thebibliography}{}
\bibitem[Ayo]{A} J. Ayoub, {\it Les Six Op\'erations de Grothendieck et le Formalisme des Cycle \'Evanescents dans le Monde Motivique}, user.math.uzh.ch/ayoub/PDF-Files/THESE.PDF
\bibitem[AF16]{AF} A. Asok, J. Fasel, {\it Comparing Euler Classes}, Q. J. Math. 67, no.4 (2016), 603-635.
\bibitem[BS01]{BS} P. Balmer, M. Schlichting, {\it Idempotent Completion of Triangulated Categories}, Journal of Algebra \textbf{236}, 819-834 (2001).
\bibitem[CD07]{CD} D. C. Cisinski, F. D\'eglise, {\it Local and Stable Homological Algebra in Grothendieck Abelian Categories}, arXiv: 0712. 3296v1.
\bibitem[CD13]{CD1} D. C. Cisinski, F. D\'eglise, {\it Triangulated Categories of Mixed Motives}, preprint, (2013).
\bibitem[CF14]{CF} B. Calm\`es, J. Fasel, {\it The Category of Finite Chow-Witt Correspondences}, arXiv:1412.2989v2 (2017).
\bibitem[CF18]{CF1} B. Calm\`es, J. Fasel, {\it A Primer for Milnor-Witt K-Theory Sheaves and Their Cohomology}, in preparation (2018).
\bibitem[D\'eg]{D} F. D\'eglise, {\it Finite Correspondences and Transfers over a Regular Base}, www.math.uiuc.edu/K-theory/765/regular\_base.pdf
\bibitem[Del87]{D1} P. Deligne, {\it Le d\'eterminant de la cohomologie}, Current Trends in Arithmetic Algebraic Geometry (K. A. Ribet, ed.), Contemporary Mathematics, vol. 67, AMS, 1987.
\bibitem[DF17]{DF} F. D\'eglise, J. Fasel, {\it MW-Motivic Complexes}, arXiv:1708. 06095, (2017).
\bibitem[DF17a]{DF1} F. D\'eglise, J. Fasel, {\it The Milnor-Witt Motivic Ring Spectrum and its Associated Theories}, arXiv:1708.06102v1, (2017).
\bibitem[Fas08]{F} J. Fasel, {\it Groupes de Chow-Witt} M\'em. Soc. Math. Fr. (N.S.) (2008), no.113, viii+197.
\bibitem[GM03]{GM} S. I. Gelfand, Y. Manin, {\it Methods of Homological Algebra (Second Edition)}, Springer-Verlag Berlin Heidelberg, (2003).
\bibitem[Har77]{H} R. Hartshorne, {\it Algebraic Geometry}, Springer-Verlag New York Inc. (1977).
\bibitem[Hir03]{Hir} P. S. Hirschhorn, {\it Model Categories and Their Localizations}, American Mathematical Society, (2003).
\bibitem[Hov07]{Hov} M. Hovey, {\it Model Categories}, American Mathematical Society, (2007).
\bibitem[HSS98]{HSS} M. Hovey, B. Shipley, J. Smith, {\it Symmetric Spectra}, arXiv: 9801077v2 (1998).
\bibitem[Kra09]{Kra} H. Krause, {\it Localization Theory for Triangulated Categories}, arXiv:0806.1324v2 (2009).
\bibitem[Mil80]{M} J. S. Milne, \'Etale Cohomology, Princeton University Press, Princeton, New Jersey, (1980).
\bibitem[Mac63]{Mac} S. MacLane, {\it Natural Assciativity and Commutativity}, Rice Institute Pamphlet - Rice University Studies, 49, no. 4 (1963).
\bibitem[Mor12]{Mo} F. Morel, {\it \(\mathbb{A}^1\)-Algebraic Topology over a Field}, volume 2052 of Lecture Notes in Math. Springer, New York (2012).
\bibitem[MVW06]{MVW} C. Mazza, V. Voevodsky, C. Weibel, {\it Lecture Notes on Motivic Cohomology}, American Mathematical Society, Providence, RI, for the Clay Mathematics Institute, Cambridge, MA (2006).
\bibitem[N01]{N} A. Neeman, {\it Triangulated Categories}, Princeton University Press, (2001).
\bibitem[PW10]{PW} I. Panin, C. Walter, {\it Quaternionic Grassmannians and Pontryagin Classes in Algebraic Geometry}, arXiv:1011.0649 (2010).
\bibitem[Pro]{SP} Stacks Project, {\it Limit of Schemes}, stacks.math.columbia.edu/download/limits.pdf
\bibitem[SV]{SV} A. Suslin, V. Voevodsky, {\it Bloch-Kato Conjecture and Motivic Cohomology with Finite Coefficients}, www.math.uiuc.edu/K-theory/0341/susvoenew.pdf.
\bibitem[Yan17]{Y} N. Yang, {\it Quaternionic Projective Bundle Theorem and Gysin Triangle in MW-Motivic Cohomology}, arXiv: 1703. 02877v4 (2018).
\end{thebibliography}
\end{document}